\documentclass[sn-mathphys-num]{sn-jnl}

\usepackage{graphicx}%
\usepackage{multirow}%
\usepackage{mathrsfs}%
\usepackage[title]{appendix}%
\usepackage{xcolor}%
\usepackage{textcomp}%
\usepackage{manyfoot}%
\usepackage{booktabs}%
\usepackage{algorithm}%
\usepackage{algorithmicx}%
\usepackage{algpseudocode}%
\usepackage{listings}%

\usepackage{amsmath,amsfonts,amssymb,amsthm,enumerate,graphicx,mathtools,tikz,caption}
\usepackage{hyperref}
 
\usepackage{xypic} 
\usepackage{mathabx}
\usepackage{mathrsfs} 
\usepackage{xspace}
\usepackage[normalem]{ulem}
\usepackage{color}

\newcommand{\ol}{\overline}

\newcommand{\Univ}{\mathbf{U}}

\newcommand{\defbold}{\textbf}

\newcommand{\inv}{^{-1}}

\newcommand{\tdlc}{t.d.l.c.\@\xspace}
\newcommand{\scpo}{scopo\@\xspace}
\newcommand{\scpos}{scopos\@\xspace}

\newcommand{\triv}{\{1\}}

\newcommand{\rist}{\mathrm{rist}}

\newcommand{\id}{\mathrm{id}}

\newcommand{\Aut}{\mathrm{Aut}}

\newcommand{\Alt}{\mathrm{Alt}}
\newcommand{\Sym}{\mathrm{Sym}}

\newcommand{\Res}{\mathrm{Res}}

\newcommand{\bs}{\backslash}

\newcommand{\propP}[1]{(\mathrm{P}_{#1})}

\newcommand{\Nb}{\mathbb{N}}

\newcommand{\Zb}{\mathbb{Z}}

\newcommand{\mc}[1]{\mathcal{#1}}
\newcommand{\ms}[1]{\mathscr{#1}}

\newcommand{\grp}[1]{\langle #1 \rangle}

\usepackage[T1]{fontenc}
\usepackage{tikz}
\usetikzlibrary{positioning}
\usetikzlibrary{intersections, decorations.pathreplacing, decorations.markings}
\usetikzlibrary{calc}
\usetikzlibrary{shapes.misc}
\usepackage{xifthen}
\tikzstyle{Vertex}=[shape=circle, draw=black, fill=black, inner sep=0pt, minimum size=3]
\tikzstyle{OpenVertex}=[shape=circle, draw=black, fill=white, inner sep=0pt, minimum size=3]
\tikzstyle{RedVertex}=[shape=circle, draw=red, fill=red, inner sep=0pt, minimum size=3]
\tikzstyle{TextNode}=[shape=rectangle, inner sep=0pt, minimum size=3]
\usetikzlibrary{decorations.pathmorphing}
\tikzset{snake it/.style={decorate, decoration=snake}}
\tikzset{->-/.style={decoration={
  markings,
  mark=at position .5 with {\arrow{>}}},postaction={decorate}}}

\theoremstyle{thmstyleone}%

\newtheorem{thm}{Theorem}[section]
\newtheorem{lem}[thm]{Lemma}
\newtheorem{prop}[thm]{Proposition}
\newtheorem{cor}[thm]{Corollary}

\newtheorem{restatedthm}{Theorem}
\renewcommand{\therestatedthm}{0}
\newtheorem{restatedcor}{Corollary}
\renewcommand{\therestatedcor}{0}

\theoremstyle{thmstyletwo}%
\newtheorem{ex}[thm]{Example}
\newtheorem{que}{Question}
\newtheorem{rem}[thm]{Remark}

\theoremstyle{thmstylethree}%
\newtheorem{defn}[thm]{Definition}

\begin{document}

\title[Groups acting on trees with Tits' independence property $\propP{}$]{Groups acting on trees with Tits' independence property $\propP{}$}

\subtitle{With an appendix by Stephan Tornier}

\author[1]{\fnm{Colin D.} \sur{Reid}}\email{C.Reid5@westernsydney.edu.au}

\author*[2]{\fnm{Simon M.} \sur{Smith}}\email{SiSmith@lincoln.ac.uk}

\equalcont{Both authors contributed equally to this work.}

\affil[1]{
\orgname{Western Sydney University},
\orgdiv{School of Computer, Data and Mathematical Sciences},
\orgaddress{
	\city{Penrith},
	\postcode{NSW 2751},
	\country{Australia}}}

\affil*[2]{\orgdiv{Charlotte Scott Research Centre for Algebra}, \orgname{University of Lincoln}, \orgaddress{\street{Brayford Pool}, \city{Lincoln}, \postcode{LN6 7TS}, \country{United Kingdom}}}

\abstract{
\begingroup
\renewcommand\thefootnote{}\footnote{Appendix by Stephan Tornier, The University of Newcastle, Australia; Email
stephan.tornier@newcastle.edu.au.\\

\today. This version of the article has been accepted for publication in Mathematische Annalen, after peer review, but it is
not the Version of Record and does not reflect post-acceptance improvements, or any corrections. The Version of Record is available online (open access) at \url{https://doi.org/10.1007/s00208-026-03412-w}.}
\addtocounter{footnote}{-1}
\endgroup
Local actions (actions of a vertex stabiliser on the neighbours of that vertex) have become an important approach to group actions on trees since J.~Tits' introduction in 1970 of the independence property $(\mathrm{P})$ and especially since a 2000 paper by M.~Burger and Sh.~Mozes.  This `local-to-global' approach has been critical in the development of the theory of totally disconnected locally compact groups because it allows the construction of  nondiscrete group actions on trees while keeping control over the action of a vertex stabiliser, in a way that is not practical under the classical Bass--Serre approach. 
The majority of constructions of nonlinear nondiscrete locally compact simple groups use $(\mathrm{P})$ and its generalisations.

In this article we give a full classification and description of all closed group actions on trees with Tits' independence property $(\mathrm{P})$ using a new coherent theory for local actions that applies to all actions on trees. 
This theory is a `local action' complement to classical Bass--Serre theory.
On the one hand, our theory gives a decomposition of a group acting on a tree into a `local action diagram' (a decorated graph that encodes all `local' information), and on the other hand a construction of a group acting on a tree from a given local action diagram. One can read directly from the local action diagram whether the resulting group has certain properties, like geometric density, compact generation and simplicity.
}

\keywords{Group actions on trees, totally disconnected locally compact groups, permutation groups}

\pacs[MSC Classification]{22D05, 20E08, 20B07}

\maketitle

\tableofcontents

\section{Introduction}

Actions on trees have long played an important role in group theory.  The most well-established perspective is that of Bass--Serre theory and the theory of ends of groups, in which actions on trees are interpreted as a generalisation of the free product and HNN constructions.  In addition, a complementary approach has emerged based on local actions, that is, the action of a vertex stabiliser on the neighbouring vertices.  In particular, two articles concerning groups acting on trees have been very important for the recent development of the theory of totally disconnected, locally compact (\tdlc) groups: a 1970 article \cite{Tits70} of Jacques Tits, which introduced property $\propP{}$ as a condition to produce the first examples of nonlinear nondiscrete locally compact simple groups, showing for instance that the automorphism group of a regular tree is virtually simple; and a 2000 article \cite{BurgerMozes} of Marc Burger and Shahar Mozes, which rediscovered the approach of Tits and used it to produce an interesting class of (virtually) simple \tdlc groups acting on trees with property $\propP{}$, which moreover arise naturally in the study of lattices in products of trees.  Since then, the majority of new constructions of compactly generated simple \tdlc groups have used the ideas of \cite{Tits70} and \cite{BurgerMozes}.  More recently in \cite{SmithDuke}, the second named author generalised the Burger--Mozes construction to obtain a kind of product of permutation groups, often resulting in a permutation group that is both primitive and simple; this was used to show that there are $2^{\aleph_0}$ isomorphism types of nondiscrete compactly generated simple \tdlc groups.

The groups constructed by Burger and Mozes in \cite{BurgerMozes}, and by the second author in \cite{SmithDuke}, are foundational constructions in a growing body of work that might now be called the {\it local-to-global theory of groups acting on trees}. In these two constructions, for an infinite tree $T$, some desired `local action' is specified on balls of radius $1$ in $T$, and the construction, subject to some conditions on the specified local action, then yields a subgroup $U \leq \Aut(T)$ that is `universal' with respect to the specified local action; this is, $U$ contains an $\Aut(T)$-conjugate of all subgroups of $\Aut(T)$ that have the specified local action. These universal groups have property $\propP{}$ and are closed as subgroups of $\Aut(T)$.

The local-to-global approach has a significant advantage over Bass--Serre theory for constructing nondiscrete groups acting on trees; that is actions $(T,G)$ in which no pointwise stabiliser of finitely many vertices fixes every vertex in $T$. The reason for this is that in Bass--Serre theory, each vertex group in a graph of groups is an entire vertex stabiliser of the group action, and consequently the graph of groups contains both local (in our sense) and global information (since a stabiliser $G_v$ includes information about how it acts on vertices at all distances from $v$).
This issue can easily be seen for example if one tries to use Bass--Serre theory to construct a nondiscrete group $G$ of automorphisms of the $n\geq 3$ regular tree $T_n$ such that $G$ has some specified local action $F \lneqq S_n$. To use Bass--Serre theory one needs to know {\it a priori} that such a group exists, since the input data in the graph of groups includes the vertex stabiliser of such an action.  This limitation disappears when using a local-to-global approach.
Conversely, given a group acting on a tree, the local-to-global approach gives a more parsimonious description of the local information, since it is not necessary to specify vertex stabilisers in full, only their action on the immediate neighbours of the vertex.

As is the case with Bass--Serre theory, the local-to-global approach only sees the faithful quotient of the action $(T, G)$ acting on the tree. For this reason (unless otherwise stated) we implicitly associate tree actions $(T,G)$ with their image in the automorphism group $\Aut(T)$ of $T$.

If a group $G \leq \Aut(T)$ has property $\propP{}$ and is closed, we say that it is \defbold{$\propP{}$-closed}. For any group $G \leq \Aut(T)$ there is a smallest $\propP{}$-closed supergroup of $G$ in $\Aut(T)$, and we call this the \defbold{$\propP{}$-closure of $G$}. The universal groups described above are $\propP{}$-closed. In this paper we give a complete description of all $\propP{}$-closed groups: they are all `universal groups' of a natural combinatorial object we call a local action diagram. 
From this one immediately obtains a detailed understanding of all group actions on trees with property $\propP{}$: if $G \leq \Aut(T)$ has property $\propP{}$ then its closure $\overline{G}$ is the universal group of a local action diagram.

Property $\propP{}$ admits a natural generalisation to $\propP{k}$, introduced in Christopher Banks, Murray Elder and George A. Willis' paper \cite{BanksElderWillis}, where 
for closed subgroups of $\Aut(T)$ 
property $\propP{1}$ is just property $\propP{}$.
Property $\propP{}$ and this 
generalisation provide a general tool for understanding all actions on trees: one starts by defining something called the $\propP{k}$-closure of any action on a tree for all $k \ge 1$
 (see Definition~\ref{PropPk}); these 
$\propP{k}$-closures give rise to a series of approximations to the original action determined by how the original action behaves on balls of radius $k$; the $\propP{k}$-closures then converge to the closure of the original action (in the permutation topology of $\Aut(T)$). 
 A group $G \leq \Aut(T)$ is $\propP{1}$-closed if and only if it is $\propP{}$-closed. 

In their 2011 paper \cite{CapDeM}, Pierre-Emmanuel Caprace and Tom de~Medts focus on 
compactly generated locally compact 
 $\propP{}$-closed groups acting on trees 
 and state that the central theme of their work is ``to investigate to what extent the global structure of these groups is determined by their local structure'', where again the `local structure' of a group refers to actions of vertex stabilisers on neighbours. 
An immediate consequence of Banks, Elder and Willis' paper  \cite{BanksElderWillis} is that, for locally finite trees, all properties of these groups must be completely determined by their `local structure'. Exploring the consequences of this local-to-global relationship has been hindered by the absence of a formal notion of what constitutes `local structure' for groups acting on trees. 

Our theory of local action diagrams addresses this deficiency (for all trees, not just those that are locally finite): local action diagrams are just graphs decorated with colours and groups, and they offer a concrete and accessible way to perfectly describe the notion of `local structure' in general. Because these local action diagrams correspond (ignoring some minor topological issues) precisely to the $\propP{1}$-closures of all actions on trees, we see that nothing in a local action diagram is superfluous and nothing is omitted. Using the theory we can thus precisely answer questions of the form ``for groups acting on trees, is some property $\mathcal{X}$ a property of local structure?'' The answer is ``no'' if and only if there are two actions with the same local action diagram, with one action satisfying $\mathcal{X}$ and the other action not.\\

The goal of the present article is twofold: (i) to advance the local-to-global theory of groups acting on trees by developing a `local action' complement to classical Bass--Serre theory, and (ii) to describe and classify all possible closed actions of groups on trees with the independence property $\propP{}$, or equivalently, by describing all possible
$\propP{}$-closures of actions on trees. In fact our `local action' complement to classical Bass--Serre theory precisely gives this classification and description of all $\propP{}$-closures of actions on trees.

Our classification is achieved using a graph decorated with colours and groups representing `local actions', and is called the {\em local action diagram}; it is analogous to a graph of groups in Bass--Serre theory.
Under this analogy, the universal groups in \cite{BurgerMozes} are a special case of HNN extensions, and the universal groups in \cite{SmithDuke} play the role of the amalgamated free products. As in \cite{BurgerMozes} and \cite{SmithDuke}, we wish to describe the group in terms of its local actions; but in contrast to these papers, we make no assumptions about the homogeneity of the local actions or the structure of the orbits of the group on vertices or arcs of the tree. Indeed, even in the vertex-transitive case, we obtain a larger class of groups than those considered by Burger and Mozes; see Section~\ref{sec:vertex_transitive}.  As a result, we need a slightly more complicated way to describe the local actions.

\begin{defn}
A \defbold{local action diagram $\Delta = (\Gamma,(X_a),(G(v)))$} consists of the following information:
\begin{itemize}
\item A connected graph $\Gamma$.  (We define graphs in the sense of Serre, except that a loop may or may not be equal to its own reverse: see Section~\ref{sec:graphs}.)
\item For each arc $a$ of $\Gamma$, a nonempty set $X_a$ (called the \defbold{colour set} of $a$).
\item For each vertex $v$ of $\Gamma$, a group $G(v)$ (called the \defbold{local action} at $v$) with the following properties: write $X_v$ to denote the disjoint union $\bigsqcup_{a \in o\inv(v)}X_a$, then the group $G(v)$ is a closed subgroup of $\Sym(X_v)$ and the sets $X_a$ are the orbits of $G(v)$ on $X_v$.
\end{itemize}
\end{defn}

There is a natural notion of isomorphism of local action diagrams.  For actions on trees the natural way to define isomorphism is conjugacy, where we say $(T,G)$ and $(T',G')$ are conjugate if there is a graph isomorphism $\theta: T \rightarrow T'$ that intertwines the two actions.  For a local action diagram $\Delta$ arising from $(T, G)$, we then define a group \defbold{$\Univ(\Delta)$} that is universal among those groups acting on $T$ with associated local action diagram $\Delta$.

A central theorem of this article, which we prove in Section~\ref{sec:correspondence} using these universal groups, is as follows.
The theorem is a classification and a complete description of closed groups with Tits' independence property $\propP{}$.

\begin{thm}\label{thm:correspondenceIntro}
There is a natural one-to-one correspondence between conjugacy classes of $\propP{}$-closed actions on trees and isomorphism classes of local action diagrams. 
\end{thm}

As with the Fundamental Theorem of Bass--Serre theory, the power of the correspondence lies not just in its existence, but in the richness of the relationship between the action on the tree and the associated diagram.  Here is a summary of Section~\ref{sec:correspondence}, highlighting the similarities with Bass--Serre theory. The various identifications and isomorphisms occurring in this summary are described precisely in the section itself.

{\it
\begin{enumerate}
\item
	For an action $(T,G)$ of a group $G$ on a tree $T$, we have an associated local action diagram $\Delta$ (akin to a graph of groups in Bass--Serre theory).
\item
	The local action diagram admits a $\Delta$-tree, $\mathbf{T}$ (akin to the universal cover of a graph of groups in Bass--Serre theory).
\item
	The $\Delta$-tree $\mathbf{T}$ admits a universal group $\Univ(\Delta)$ which acts naturally on $\mathbf{T}$ (akin to the fundamental group of a graph of groups in Bass--Serre theory).
\item
	The action $(\mathbf{T}, \Univ(\Delta))$ is in fact an action  of $\Univ(\Delta)$ on the original tree $T$; this latter action is $\propP{}$-closed and has associated local action diagram $\Delta$. Moreover, the action $(\mathbf{T}, \Univ(\Delta))$ is isomorphic to the $\propP{}$-closure of $(T, G)$ (akin to the Fundamental Theorem of Bass--Serre theory, \cite[\S5.4]{Serre:trees}).
\item
	In particular, we have that every closed action $(T,G)$ with Tits' independence property $\propP{}$ is equal to $(\mathbf{T}, \Univ(\Delta))$, where $\Delta$ is the local action diagram of $(T,G)$; and moreover, every local action diagram $\Delta$ gives rise to a closed action $(\mathbf{T}, \Univ(\Delta))$ that enjoys Tits' independence property $\propP{}$.
\end{enumerate}
}

In this way, we obtain a complete description of all closed actions of groups acting on trees with Tits' independence property $\propP{}$: they are precisely the universal groups of local action diagrams.

It should be emphasised that this description of closed actions $(T,G)$ with Tits' independence property $\propP{}$ via local action diagrams is usable, in that one can construct novel examples by drawing new local action diagrams, or analyse existing examples by examining their local action diagrams. 
Importantly, in the former one retains full control of the tree and the resulting action on the tree, which is often impossible to achieve in Bass--Serre theory.
As we describe below, important global properties of the group can be read directly from the local action diagram.

There are no surprises in how a local action diagram is obtained from a ($\propP{}$-closed) action of a group $G$ on a tree $T$.  The graph $\Gamma$ is the quotient graph $G \backslash T$; each local action $G(v)$ represents the closure of the action of a vertex stabiliser $G_{v^*}$ (where $v^* \in VT$ lies in the preimage of $v$) on the arcs $o\inv(v^*)$ of $T$ originating at $v^*$; those arcs are partitioned into $G_{v^*}$-orbits, represented by the colour sets, with the result that there is a natural one-to-one correspondence between $o\inv(v)$ and $G_{v^*}$-orbits on $o\inv(v^*)$.

The significance of the correspondence, then, is in the following two observations.
\begin{enumerate}[(1)]
\item The local action diagram exactly describes the $\propP{}$-closure of the original action up to conjugacy.  In particular, any `large-scale' information about the original group can be recovered from the quotient graph $\Gamma$ together with the local actions.
\item All possible local action diagrams arise in this manner.  In particular, $\Gamma$ can be any connected graph in our sense, and apart from how $\Gamma$ limits the number of orbits of the local actions, there are no compatibility conditions on which local actions can be combined.
\end{enumerate}
For comparison, the Burger--Mozes framework corresponds to the case when $\Gamma$ is a single vertex with a set of loops, each of which is its own reverse; the framework of \cite{SmithDuke} (when the local actions are transitive) corresponds to the case that $\Gamma$ has two vertices and no loops.

As an example of what this means in practice, consider the class $\mc{C}(n,d)$ of $\propP{}$-closed actions on trees $(T,G)$ such that $G$ has at most $n$ orbits on vertices and no vertex has degree greater than $d$.  Theorem~\ref{thm:correspondenceIntro} immediately shows that for given natural numbers $n$ and $d$, there are only finitely many conjugacy classes of actions in $\mc{C}(n,d)$; but because of all the possible graphs and decorations, the number of conjugacy classes will grow quite rapidly with $n$ and $d$.\\

As noted previously, important global properties of the action $(T,G)$ can be read directly from the local action diagram. To that end, the next part of the article is concerned with characterising various natural properties of interest for groups acting on trees in terms of the local action diagram, including various properties of topological groups.

Recall that Tits' main theorem on property~$\propP{}$, ensuring that the subgroup $G^+$ generated by arc stabilisers is trivial or simple, only applies to $\propP{}$-closed actions that are  \defbold{geometrically dense}, meaning that there is no proper invariant subtree or fixed end.  Fixed ends and invariant subtrees can be recognised in the local action diagram, since they correspond to so-called scopos. This means that sufficient conditions for the simplicity of $G^+$ can be read directly from the local action diagram for $G$.

\begin{defn} \label{def:introScopoDelta}
Let $\Delta = (\Gamma,(X_a),(G(v)))$ be a local action diagram.  A \defbold{strongly confluent partial orientation} (henceforth, \defbold{\scpo}) of $\Delta$ is a subset $O$ of $A\Gamma$ such that:
\begin{enumerate}[(i)]
\item If $a \in O$, then $\ol{a} \not\in O$ and $|X_a|=1$;
\item For all $v \in V\Gamma$, if $O$ contains an arc $a$ originating at $v$, then $O$ contains all arcs other than $\ol{a}$ that terminate at $v$.
\end{enumerate}
\end{defn}

\begin{thm}[See Section~\ref{sec:invariants}]\label{thm:invariants}
Given a group $G$ acting on a tree $T$, then the invariant subtrees and fixed ends of the action naturally correspond to \scpos of the local action diagram, with the empty \scpo corresponding to $T$ itself.
\end{thm}

In particular, the action $(T,G)$ is geometrically dense if its local action diagram is \defbold{irreducible}, meaning that the only \scpo is the empty one.  The \scpos corresponding to fixed ends, and the \scpos corresponding to invariant subtrees, can be described precisely (see Section~\ref{TypesOfActionRevisited}).

Since \scpos are quite special, it is easy to write down sufficient conditions for a local action diagram to give rise to an action on the tree that is geometrically dense. One of the obstacles to a geometrically dense action is when the local action diagram $\Delta = (\Gamma,(X_a),(G(v)))$ has a stray leaf: a \defbold{leaf} of a graph is a vertex $v$ with exactly one outgoing edge, such that the edge is not a loop, and a \defbold{stray leaf} of $\Delta$ is a leaf $v$ of $\Gamma$ such that $|X_v| = 1$ (or equivalently, $v$ is a leaf of $\Gamma$ such that $G(v)=\triv$).  The other kinds of obstacle, which will be defined in Section~\ref{Section:TitsRevisited}, are \defbold{focal cycle}, \defbold{horocyclic end}, and \defbold{stray half-tree}. As we shall see in Proposition~\ref{prop:irreducible_check}, the local action diagram $\Delta$ is irreducible if and only if it is not a focal cycle and has no horocyclic ends, no stray half-trees and no stray leaves.  With these we obtain the following result, allowing one to read the simplicity of $G^+$ directly from the local action diagram of $G$.

\begin{cor} \label{cor:DirectlyReadingSimplicity} Let $T$ be a tree, and let $G \leq \Aut(T)$ have $\propP{}$. Suppose $\Delta$ is the local action diagram of $(T, G)$. Then $\Delta$ is not a focal cycle and has no horocyclic ends, no stray half-trees and no stray leaves if and only if $G$ is geometrically dense. In particular, if $\Delta$ is not a focal cycle and has no horocyclic ends, no stray half-trees and no stray leaves then $G^+$ is abstractly simple or trivial.
\end{cor}

In particular, for the local action diagram $\Delta = (\Gamma,(X_a),(G(v)))$, if $\Gamma$ is a finite graph that is not a cycle graph (where a \defbold{cycle graph} is a finite connected graph in which all vertices have degree $2$), then one sees from the definitions that $\Delta$ is not a focal cycle and has no horocyclic ends or stray half-trees, so irreducibility is equivalent to the absence of stray leaves.

\begin{cor}\label{cor:DirectlyReadingSimplicity:finite} Let $T$ be a tree, and let $G \leq \Aut(T)$ have $\propP{}$. Suppose $\Delta = (\Gamma,(X_a),(G(v)))$ is the local action diagram of $(T, G)$, such that $\Gamma$ is finite and not a cycle graph. Then $G$ is geometrically dense if and only if for every leaf $v$ of $\Gamma$ we have $|X_v|>1$. In particular, if $\Delta$ has no stray leaves then $G^+$ is abstractly simple or trivial.\end{cor}

Note that the details of the local actions are not important here, only the structure of the quotient graph $\Gamma$ and the sizes of the colour sets. A small amount of `local' information is enough to determine whether or not $G^+$ is trivial: $G^+$ being trivial is equivalent (see Lemma~\ref{lem:free_diagram}) to all of the local actions in $\Delta$ being free (i.e. semi-regular).
In the case that $(T,G)$ is not geometrically dense, we can also describe which of the degenerate cases of actions on trees (if any) it falls into using the local action diagram (see Sections~\ref{sec:Tits_revisited} and \ref{TypesOfActionRevisited}).

The next theorem is almost an exact characterisation of simplicity for groups with faithful $\propP{}$-closed actions, with two minor caveats: one is that we need to exclude a couple of cases that are degenerate from the perspective of local-to-global structure, and the other is that we need to ensure that the group has closed action on any invariant subtree.

\begin{defn}
Say that $G \le \Aut(T)$ is \defbold{strongly closed} if for every $G$-invariant subtree $T'$ of $T$, the action of $G$ on $T'$ is closed.
\end{defn}

In particular, every closed geometrically dense action is strongly closed. For our purposes, the requirement of being strongly closed is easily achieved: as we will see in
Corollary~\ref{cor:StronglyClosedIsNoBigDeal}, a locally compact $\propP{}$-closed subgroup of $\Aut(T)$ that acts with translation (i.e.~contains a translation) is always strongly closed.

\begin{thm}[See Section~\ref{sec:Tits_revisited}]\label{thm:UDelta_simple}
Let $(T,G)$ be a faithful $\propP{}$-closed and strongly closed action on a tree $T$.  Then the following are equivalent:
\begin{enumerate}[(i)]
\item $G$ is a simple group, $G$ acts with translation, and there is no finite set of vertices whose pointwise stabiliser is trivial.
\item  \label{item:lad_description} There is an invariant subtree $T'$ (possibly equal to $T$) which is infinite and on which $G$ acts faithfully.  Moreover, letting $\Delta = (\Gamma,(X_a),(G(v)))$ be the local action diagram of $(T',G)$, then $\Delta$ is irreducible; $\Gamma$ is a tree; and each of the groups $G(v)$ is closed and generated by point stabilisers, with $G(v) \neq \triv$ for some $v \in V\Gamma$.
\end{enumerate}
Furthermore, in (\ref{item:lad_description}) the action $(T', G)$ is $(T', \Univ(\Delta))$.
\end{thm}

Note that the condition that there is no finite set of vertices whose pointwise stabiliser is trivial is equivalent to saying that $G$ is nondiscrete in the permutation topology on $T$ (i.e.~the $\Aut(T)$ topology). \\

Next, we describe some topological properties of $\propP{}$-closed subgroups of $\Aut(T)$ with the permutation topology; these are already well-understood in the locally finite case, but in the present context we are making no assumptions about the degree of $T$.  We highlight the following special case.
Note that a permutation group is \defbold{subdegree-finite} if all orbits of point-stabilisers are finite.

\begin{thm}[See Section~\ref{sec:topology:proofs}]\label{thm:comp_gen+geom_dense}
Let $\Delta = (\Gamma,(X_a),(G(v)))$ be a local action diagram.  Then the following are equivalent:
\begin{enumerate}[(i)]
\item $\Univ(\Delta)$ is compactly generated, locally compact and has geometrically dense action on its associated tree;
\item $\Delta$ is irreducible; $\Gamma$ is finite; and each of the groups $G(v)$ is compactly generated and subdegree-finite.
\end{enumerate}
Moreover, if (i) and (ii) hold, then $\Univ(\Delta)$ is Polish, acting on a countable tree, and all arc stabilisers of $\Univ(\Delta)$ are compact.
\end{thm}

Combining Corollary~\ref{cor:DirectlyReadingSimplicity:finite}, Theorem~\ref{thm:UDelta_simple} and Theorem~\ref{thm:comp_gen+geom_dense} gives the following corollary, where $\ms{S}$ denotes the class of nondiscrete compactly generated, topologically simple \tdlc groups.
Understanding $\ms{S}$ and constructing novel examples of groups in $\ms{S}$ is a central theme of research in the theory of locally compact groups (see \cite{CapraceSimple}, for example).
Our corollary allows us to construct local action diagrams that yield new groups in $\ms{S}$. Note that the conditions in part (ii) of the corollary also imply that the local action diagram $\Delta$ is irreducible.

\begin{cor}\label{cor:comp_gen+simple}
Let $(T,G)$ be a faithful $\propP{}$-closed and strongly closed action on a tree $T$.  Then the following are equivalent:
\begin{enumerate}[(i)]
\item We have $G \in \ms{S}$ and the action does not fix any vertex of $T$.
\item \label{item:lad_description:comp_gen} There is a unique smallest invariant subtree $T'$ (possibly equal to $T$) on which $G$ acts faithfully.  Moreover, letting $\Delta = (\Gamma,(X_a),(G(v)))$ be the local action diagram of $(T',G)$, then $\Gamma$ is a finite tree, and each of the groups $G(v)$ is closed, compactly generated, subdegree-finite and generated by point stabilisers, with $G(v) \neq \triv$ for every leaf $v$ of $\Gamma$.
\end{enumerate}
Furthermore, in (\ref{item:lad_description:comp_gen}) the action $(T', G)$ is $(T', \Univ(\Delta))$.
\end{cor}

We also consider the structure of open subgroups. Tits' main theorem from \cite{Tits70} imposes a restriction on the normal subgroups of a $\propP{}$-closed group.  In the case of compact arc stabilisers, we find comparable restrictions on the closed subgroups that can be constructed from open subgroups, including all closed subnormal subgroups; unlike Tits' theorem we do not need to make any assumption about the minimality of the action.  See Section~\ref{sec:open_hyperbolic}. 

 Our study of open subgroups also leads to the following theorem, which  should be compared with a theorem of Caprace--De Medts, \cite[Theorem~A and Theorem~3.9]{CapDeM}. (Note that in \cite{CapDeM}, unlike in the present article, it is assumed that the tree is locally finite.) 
 Our theorem includes as a special case $2^{\aleph_0}$ nonisomorphic groups in $\ms{S}$ that can be constructed following \cite[Remark 40]{SmithDuke}.

\begin{thm}[See Section~\ref{sec:local_ends}]\label{thm:open_primitive}
Let $T$ be a tree such that every vertex of $T$ has at least three neighbours, and let $G$ be a nondiscrete $\propP{}$-closed subgroup of $\Aut(T)$, such that $G$ does not preserve any proper subtree of $T$.
Then the following are equivalent:
\begin{enumerate}[(i)]
\item Every proper open subgroup of $G$ has bounded action on $T$ and point stabilisers in  $G$ are pairwise incomparable with respect to the subgroup inclusion partial order.
\item There are adjacent vertices $v,w$ of $T$ such that $G_v$ and $G_w$ are distinct maximal subgroups of $G$.
\item $G$ preserves the natural bipartition of the vertices of $T$ and acts primitively on each part.
\item Letting $\Delta = (\Gamma,(X_a),(G(v)))$ be the local action diagram of $(T,G)$, then $|V\Gamma| = 2$ and for all $v \in V\Gamma$, $G(v)$ is primitive but not regular.
\end{enumerate}
Moreover, if (i)--(iv) hold then $G$ is simple; $\Gamma$ consists of a single undirected edge with two distinct endpoints; the action of $G$ on $T$ is geometrically dense; and $G$ is an amalgamated free product $G_v \ast_{G_{(v,w)}} G_w$.
\end{thm}

Section~\ref{sec:examples} is devoted to examples.  We show how the local action diagram can be used for classifying group actions on trees by classifying all 70 of the $\propP{}$-closed vertex-transitive actions on trees of degree $0 \le d \le 5$, and give a GAP (\cite{GAP4}) implementation due to Stephan Tornier that can be used to classify vertex-transitive actions of larger degrees.\\

Finally, we give an example of how Corollary~\ref{cor:comp_gen+simple} can be used to produce more groups in $\ms{S}$, which demonstrates that within $\ms{S}$, the groups $\Univ(\Delta)$ have a certain universality property.

\begin{thm}[See Section~\ref{sec:new_simple}]\label{thm:combination_simple}
Let $n$ be a positive integer and let $G_1,\dots,G_n$ be nontrivial compactly generated \tdlc groups, such that for each $G_i$ there is a compact open subgroup $U_i$ such that $G_i = \grp{gU_ig\inv \mid g \in G_i}$ and $\bigcap_{g \in G_i}gU_ig\inv = \triv$. For example, we can take $G_i \in \ms{S}$ and $U_i$ to be any compact open subgroup.  Then there exists $\Univ(\Delta) \in \ms{S}$ acting continuously on a countable tree $T$, vertex stabilisers $O_1,\dots,O_n$ of $\Univ(\Delta)$ and compact normal subgroups $K_i$ of $O_i$, such that $O_i \cong K_i \rtimes G_i$ for $1 \le i \le n$.
\end{thm}

$\,$

The structure of the paper is as follows. Let $T$ be a tree with $G \leq \Aut(T)$.
In Section~\ref{Preliminaries} we introduce background material from permutation groups (\S\ref{sec:groups}), graph theory (\S\ref{sec:graphs}) and groups acting on trees (\S\ref{sec:action_types}--\S\ref{BassSerreTheory}). In particular, in \S\ref{sec:action_types} we describe the six classes of group action on a tree: Fixed vertex, Inversion, Lineal, Horocyclic, Focal and General (summarised in Table~\ref{fig:types} on page \pageref{fig:types}). In \S\ref{subsec:P-closure} we give Tits' independence property $\propP{}$ and consider related notions like $\propP{}$-closure and $\propP{k}$-closure. In \S\ref{BassSerreTheory} we summarise the Bass--Serre Theory we require.

Section~\ref{sec:correspondence} contains the core of our theory of local action diagrams. Here we define (associated) local action diagrams, $\Delta$-trees, $\Delta$-colourings and universal groups. We also define isomorphisms and automorphism for these structures and prove that the various choices made during construction do not (up to isomorphism) matter. We prove the correspondence result (Theorem~\ref{thm:correspondenceIntro}) from the introduction showing that local action diagrams completely describe $\propP{}$-closed actions on trees. We conclude with two explicit examples of local action diagrams: the automorphism group of a specific tree and a Burger--Mozes group.

In Section~\ref{sec:tree} we investigate subgroups of $\propP{}$-closed groups that are themselves $\propP{}$-closed.
To do this we introduce the natural notion of a local subaction diagram in \S\ref{LocalSubactionDiagrams}. We then present sufficient conditions for a subgroup of a $\propP{}$-closed group $G$ to be $\propP{}$-closed in \S\ref{BeingPClosed}. From this we go on to consider specific types of subgroup: vertex stabilisers (\S\ref{VertexStabsAndPClosure}), the subgroup $G^+$ generated by all arc-stabilisers (\S\ref{sec:G+}), open subgroups containing a translation (\S\ref{sec:open_hyperbolic}) and stabilisers of so-called locally invariant ends (these are ends $\xi$ whose stabiliser $G_\xi$ is open) (\S\ref{sec:local_ends}). 	 

In Section~\ref{InvariantStructures} we determine precisely how to detect certain kinds of $G$-invariant structures in $T$ from the local action diagram $\Delta$ of the action $(T, G)$. In \S\ref{sec:invariants} we focus on detecting $G$-invariant ends and proper $G$-invariant subtrees. From this we derive our precise characterisation of geometrically dense actions (Theorem~\ref{thm:invariants}) for $\propP{}$-closed groups. 
We define \scpos and cotrees for graphs and local action diagrams. We show that all \scpos of connected graphs and local action diagrams are one of three types ((\ref{scposTypeI})--(\ref{scposTypeIII})), with each type arising from either a cotree or an end. We develop tools for recognising the type of a  \scpo and we relate $G$-invariant ends and subtrees of $T$ to \scpos of $\Delta$.
In \S\ref{sec:Tits_revisited} we revisit Tits' Theorem (Theorem~\ref{prop:TitsClassification}) with our now complete understanding of property $\propP{}$ and geometric density, via our theory of local action diagrams. We prove 
Corollaries~\ref{cor:DirectlyReadingSimplicity} and \ref{cor:DirectlyReadingSimplicity:finite} and Theorem~\ref{thm:UDelta_simple} from the introduction, and in Lemma~\ref{lem:free_diagram} we determine from the local action diagram precisely when $G^+$ is trivial. In \S\ref{TypesOfActionRevisited} we revisit the six types of action on a tree (Fixed vertex, Inversion, Lineal, Horocyclic, Focal and General) and give rules for detecting these actions from the local action diagram.

In Section~\ref{sec:topology} we explore how some of the topological properties of the local actions in the local action diagram relate to the topological properties of the action $(T, G)$. We characterise Polish $\propP{}$-closed groups in \S\ref{PolishPropPGroups} and locally compact $\propP{}$-closed groups in \S\ref{LocallyCompactPClosedGroups}, and in the latter subsection we go on to examine compact generation. In \S\ref{sec:topology:proofs} we prove the topological results from the introduction: Theorem~\ref{thm:comp_gen+geom_dense} and Corollary~\ref{cor:comp_gen+simple}.

In Section~\ref{sec:examples} we give three example applications of our theory of local action diagrams. The first example is \S\ref{sec:vertex_transitive}, where we use our theory to list all $\propP{}$-closed actions on trees whose degree is at most $5$. Our second example is \S\ref{sec:conn_one}, where we use our theory to determine all automorphism groups of simple, nontrivial, vertex-transitive graphs with vertex connectivity one (the automorphism groups of such graphs have been of interest to graph theorists since the 1970s --- see \cite{Watkins78} for example). Our third example is \S\ref{sec:new_simple} where we use our theory to give a new construction technique for combining simple groups in $\mathscr{S}$ to make new simple groups in $\mathscr{S}$.

In Section~\ref{questions}, we list open questions arising from our theory and ideas for extending the results.

Appendix~\ref{app:vt_1cl_gap} is a GAP implementation due to Stephan Tornier.

\paragraph{Acknowledgements}
The first author is a Friedrich Wilhelm Bessel awardee, with research supported by ARC grant FL170100032 and the Alexander von Humboldt Foundation.  Research was conducted while the first author was employed at the University of Newcastle, Australia. The second author was supported for part of this project by an EPSRC Standard Grant (\#EP/V036874/1). The research was also supported by a London Mathematical Society Research in Pairs (Scheme 4) grant  \#41745.

\section{Preliminaries} \label{Preliminaries}

In this section we briefly introduce preliminary concepts we require. We consider permutation groups in \S\ref{sec:groups} and graph theory in \S\ref{sec:graphs}. In \S\ref{sec:action_types}--\S\ref{BassSerreTheory} we turn our attention to groups acting on trees, describing the six classes of group action on a tree in \S\ref{sec:action_types},
Tits' independence property $\propP{}$, $\propP{}$-closure and $\propP{k}$-closure in \S\ref{subsec:P-closure} and Bass--Serre Theory in \S\ref{BassSerreTheory}.

\subsection{Permutation groups}\label{sec:groups}

For further background information on permutation groups see \cite{DixonMortimer}, and for the relationship between permutation groups and topological groups see \cite{KronMoller}.

For a set $\Omega$ the group of all permutations of $\Omega$ is $\Sym(\Omega)$. Actions in this paper are from the left, and so we follow this convention for permutation groups. If $\omega \in \Omega$ and $G \leq \Sym(\Omega)$ we denote the stabiliser of $\omega$ in $G$ by $G_\omega$. An orbit of a stabiliser in $G$ is called a \defbold{suborbit} of $G$ and the cardinalities of these suborbits are called the \defbold{subdegrees} of $G$. As noted previously, a group is \defbold{subdegree-finite} if all its suborbits are finite.
If $G_\omega$ is trivial for all $\omega \in \Omega$ we say that $G$ is \defbold{semi-regular} (this type of action is also called \defbold{free}). If $G$ is semi-regular and transitive it is \defbold{regular}. Notice that if $\Lambda$ is a subset of $\Omega$ and $G$ leaves $\Lambda$ invariant (in other words, $G$ setwise stabilises $\Lambda$) then the action of $G$ on $\Lambda$ induces a subgroup of $\Sym(\Lambda)$ on $\Lambda$.

If $\Omega'$ is a set, then $G \leq \Sym(\Omega)$ and $H \leq \Sym(\Omega')$ are \defbold{permutation isomorphic} via the permutation isomorphism $(\theta, \varphi)$ if $\theta : \Omega \rightarrow \Omega'$ is a bijection and $\varphi: G \rightarrow H$ is an isomorphism such that $\theta(g \omega) = \varphi(g)\theta(\omega)$ holds for all $g \in G$ and all $\omega \in \Omega$.

Now suppose $G$ is transitive on $\Omega$. A nonempty subset $\Lambda$ of $\Omega$ is called a \defbold{block} if for all $g \in G$ the image $g\Lambda$ and $\Lambda$ are either equal or disjoint. A group always admits $\Omega$ as a block and singletons $\{\omega\}$ as blocks for all $\omega \in \Omega$. Such blocks are called \defbold{trivial blocks}, and any other blocks are called \defbold{nontrivial blocks}. If $G$ admits no nontrivial blocks then we say that $G$ is \defbold{primitive}; otherwise we say $G$ is imprimitive. If $\Lambda$ is a block then the set of images $\Sigma = \{g\Lambda : g \in G\}$ is called a \defbold{system of imprimitivity}.

Thus we see that any permutation group $G \leq \Sym(\Omega)$ can be decomposed first into its transitive constituents: the transitive permutation groups induced by the action of $G$ on each of its orbits. Each transitive constituent $H$ can be further decomposed into two ``constituent''  transitive actions in the following way: one finds a system of imprimitivity $\Sigma$ for $H$ and considers on the one hand the action induced on $\Sigma$ by $H$, and on the other hand the action induced on any block $\Lambda \in \Sigma$ by the setwise stabiliser of $\Lambda$ in $H$.  (The above decomposition process can then be applied individually to these two constituent transitive actions, and so on.)  Primitive permutation groups are then precisely those permutation groups that cannot be nontrivially decomposed using the above process of decomposition.

There are two further ways that are commonly used to identify primitive permutation groups. Suppose $G$ is transitive and $\Omega$ contains at least two elements. Then $G$ is primitive if and only if all point stabilisers in $G$ are maximal subgroups of $G$, and this is true if and only if the only $G$-invariant equivalence classes on $\Omega$ are the trivial relation (in which each element in $\Omega$ is related only to itself) and the universal relation (in which all elements in $\Omega$ are pairwise related).

Now let $G$ be any group acting on $\Omega$. The action is \defbold{faithful} if the pointwise stabiliser of $\Omega$ in $G$ is trivial. There is a natural topology on $G$ called the \defbold{permutation topology} in which a neighbourhood basis of the identity consists of all pointwise stabilisers in $G$ of finite subsets of $\Omega$. Endowing $\Omega$ with the discrete topology and considering elements of $G$ as maps from $\Omega$ to $\Omega$, the permutation topology is equal to the topology of pointwise convergence. The permutation topology on $G$ is $T_0$ (that is, it distinguishes any two points in $G$) if and only if the action is faithful.  However if the action is faithful, then under the permutation topology, $G$ is Hausdorff, with a base of topology consisting of clopen sets. 
If $G$ is faithful then we consider $G \leq \Sym(\Omega)$. If $G \leq \Sym(\Omega)$ is closed under the permutation topology, then $G$ is compact if and only if all orbits of $G$ are finite, because in the latter case $G$ is a closed subgroup of a direct product of finite groups (the symmetric groups on each of the finite orbits) and the permutation topology coincides with the product topology (see Lemma~\ref{lem:lc_permutation}). Thus, every closed and subdegree-finite permutation group is a totally disconnected and locally compact topological group.
As we shall see in Section~\ref{subsec:P-closure}, this relationship works also in the other direction: every totally disconnected and locally compact topological group is isomorphic as a topological group to a closed subdegree-finite permutation group.

\subsection{Graphs}\label{sec:graphs}

 A \defbold{graph} $\Gamma = (V,A,o,r)$ consists of a vertex set $V = V\Gamma$, a set $A = A\Gamma$ of arcs, a map $o:A \rightarrow V$ assigning to each arc an \defbold{origin} (or \defbold{initial}) \defbold{vertex}, and a bijection $r: A \rightarrow A$, denoted $a \mapsto \ol{a}$ and called \defbold{edge reversal} (or sometimes \defbold{edge inversion}), such that $r^2 = \mathrm{id}$. The \defbold{terminal vertex} of an edge is $t(a) := o(\ol{a})$. A \defbold{loop} is an arc $a$ such that $o(a) = t(a)$.  If $a$ is a loop, it is important that we allow both $\ol{a} = a$ and $\ol{a} \neq a$ as possibilities. A \defbold{leaf} is a vertex $v$ such that $|o\inv(v)| = 1$ and the arc in $o\inv(v)$ is not a loop.
 We call the pair $\{a, \overline{a}\}$ an \defbold{edge} between the vertices $o(a)$ and $t(a)$. Two vertices are \defbold{adjacent} if there is an edge between them. 

A \defbold{subgraph} $\Gamma' = (V',A',o',r')$ of $\Gamma$ is a graph such that $V' \subseteq V$ and $A' \subseteq A$ with $o' : A' \rightarrow V'$ (resp.~$r' : A' \rightarrow A'$) equal to the restriction  of $o$ to $A'$ (resp. $r'$ to $A'$). For a subset $V' \subseteq V$ the \defbold{subgraph of $\Gamma$ induced by $V'$} is the subgraph of $\Gamma$ with vertex set $V'$ and arc set $\{a \in A\Gamma : o(a), t(a) \in V'\}$.

Since graphs in this paper are not simple, the graph subtraction operation is not well behaved and so we avoid it. Instead, we define for subsets $V' \subseteq V\Gamma$ and $A' \subseteq A\Gamma$ the following graphs: \defbold{$\Gamma \smallsetminus V'$} is the subgraph of $\Gamma$ induced by $V\Gamma \smallsetminus V'$, and \defbold{$\Gamma \smallsetminus A'$} is the subgraph of $\Gamma$ with vertex set $V\Gamma$ and arc set $A\Gamma \smallsetminus (A' \cup \overline{A'})$.

For an interval $I \subseteq \Zb$, let $\hat{I} = \{i \in I : i+1 \in I\}$. A \defbold{path} indexed by an interval $I \subseteq \Zb$ in $\Gamma$ is then a sequence of vertices $(v_i)_{i \in I}$ and edges $(\{a_i, \overline{a_i}\})_{i \in \hat{I}}$ such that $\{a_i, \overline{a_i}\}$ is an edge in $\Gamma$ between $v_i$ and $v_{i+1}$ for all $i \in \hat{I}$. Where there is no ambiguity, we will sometimes specify a path simply by giving the sequence of vertices or of arcs.  If $I$ is finite, say $I = \{0,\dots,n\}$, we say this is a path from $v_0$ to $v_{n}$ (or between $v_0$ and $v_n$) and the \defbold{length} of the path is $n$ (that is, the length of the sequence of arcs).  We allow paths of length $0$, that is, a single vertex with no arcs is a path.
The path \defbold{backtracks} if there is some $i \in I$ such that $a_{i+1} = \ol{a_i}$, in other words the same undirected edge is used twice in a row.  The path is \defbold{simple} if all vertices are distinct; we will sometimes regard a simple path as a subgraph of $\Gamma$.  If $n > 0$ and $I = \{0,\dots,n\}$, $v_0 = v_{n}$ and all vertices $v_0,v_1, \dots, v_{n-1}$ are distinct, then the path is called a \defbold{cycle} of length $n$.  Note that a single vertex $v$ together with a loop at $v$ together form a cycle.  The same definitions apply to \defbold{directed} paths, except that we only include the arcs $a_i$ where $o(a_i) = v_i$ and $t(a_i) = v_{i+1}$ for $i \in \hat{I}$.  If there is a path between two vertices $v$, $w$, then there is a (directed) path of minimal length, called a \defbold{shortest (directed) path} from $v$ to $w$, and the length of this shortest path is the \defbold{distance} between $v$ and $w$.  If there is no path then the distance is taken to be infinite.

An \defbold{orientation} of a graph $\Gamma$ is a subset $O$ of $A\Gamma$ such that for all $a \in A\Gamma$, either $a$ or $\overline{a}$ is in $O$, but not both.  In particular, an \defbold{orientable graph} is a graph in the sense of Serre, in other words, there are no edges such that $a = \overline{a}$.  More generally, a \defbold{partial orientation} is any (possibly empty) subset $O$ of $A\Gamma$ such that given $a \in A\Gamma$, $O$ does not contain both of $a$ and $\overline{a}$, but it could contain neither of them.  Given a group $G$ of automorphisms of $\Gamma$, we say the (partial) orientation is \defbold{$G$-invariant} if $gO = O$ for all $g \in G$.

The cardinality of $\Gamma$ is defined to be that of the set $V\Gamma \sqcup A\Gamma$.
 The \defbold{degree} of a vertex $v \in V$ is $\deg(v):=|o\inv (v)|$, and the graph is \defbold{locally finite} if every vertex has finite degree.  The \defbold{degree} of the graph is defined to be 
\[
\deg(\Gamma) := \sup_{v \in V\Gamma} \deg(v).
\]

 The graph $\Gamma$ is \defbold{simple} if the map $A \rightarrow V \times V$ by $a \mapsto (o(a),t(a))$ is injective and no arc is a loop.  In this case, the arc $a$ is sometimes identified with the pair $(o(a),t(a))$.  The graph is \defbold{connected} if there is a path between any two distinct vertices. In a simple graph $\Gamma$ a \defbold{ray} is a one-way infinite simple path. The \defbold{ends} of $\Gamma$ are equivalence classes\footnote{There are competing inequivalent definitions of an end that are all equivalent in the locally finite case. The definition we give here gives rise to so-called \defbold{vertex-ends} in the non-locally-finite case; other types of ends are edge-ends and metric ends. See \cite{KronEnds} for more details.}
of rays, in which two rays $R_1, R_2$ lie in the same end if and only if there is another ray $R$ in $\Gamma$ that contains infinitely many vertices of $R_1$ and infinitely many vertices of $R_2$.
 A \defbold{tree} is a nonempty simple, connected graph that contains no cycles.  In a tree, there is a unique shortest path between any two vertices $v$ and $w$, which we denote $[v,w]$.
  
A \defbold{graph homomorphism} $\theta: \Gamma \rightarrow \Gamma'$ is a pair of maps $\theta_V: V\Gamma \rightarrow V\Gamma'$ and $\theta_A: A\Gamma \rightarrow A\Gamma'$ that respect origin vertices and edge reversal: $\theta_V(o(a)) = o(\theta_A(a))$ and $\ol{\theta_A(a)}=\theta_A(\ol{a})$.  We say $\theta$ is an \defbold{isomorphism} if in addition, $\theta_V$ and $\theta_A$ are both bijections, and an \defbold{automorphism} if it is an isomorphism such that $\Gamma = \Gamma'$.  The automorphisms of $\Gamma$ form a group, denoted $\Aut(\Gamma)$.  When $\Gamma$ is a simple graph, the automorphisms of $\Gamma$ act faithfully as the group of permutations of $V$ that respect the edge relation in $V\times V$. In this case we identify $\Aut(\Gamma)$ with the corresponding subgroup of $\Sym(V)$.

For $G$ a group acting on a graph $\Gamma$ and a vertex, arc or edge $e$ of $G$, the orbit of $e$ under $G$ is denoted $Ge$. The action of $G$ gives a \defbold{quotient graph} $G \backslash \Gamma$ as follows: the vertex set $V_G$ is the set of $G$-orbits on $V$ and the arc set $A_G$ is the set of $G$-orbits on $A$. The origin map $\tilde{o}:A_G\rightarrow A_G$ is defined by $\tilde{o}(Ga):=Go(a)$; this is well-defined since graph automorphisms send origin vertices to origin vertices. The reversal $\tilde{r}:A_G\rightarrow A_G$ is given by $Ga\mapsto G\ol{a}$; this map is also well-defined. We will abuse notation and write $o$ and $r$ for $\tilde{o}$ and $\tilde{r}$. We denote the quotient map of the action $(\Gamma, G)$ by $\pi_{(\Gamma,G)}$.

A subset of the vertices of a connected graph is \defbold{bounded}, respectively \defbold{unbounded}, if it has finite, respectively infinite diameter in the graph metric.  Given a vertex $v$ of a graph $\Gamma$, write $B_n(v)$ (the \defbold{ball of radius $n$}) for the induced subgraph formed by all vertices $w$ such that $d_{\Gamma}(v,w) \le n$, and $S_n(v)$ (the \defbold{sphere of radius $n$}) for the set of vertices $w$ such that $d_{\Gamma}(v,w) = n$.

A graph homomorphism $\theta: \Gamma \rightarrow \Gamma'$ is \defbold{surjective} if it is surjective on both vertices and edges, and is \defbold{locally surjective} if for each $v \in V\Gamma$, we have $o\inv_{\Gamma'}(\theta(v)) = \theta(o\inv_{\Gamma}(v))$.  Note that for a locally surjective map $\theta: \Gamma \rightarrow \Gamma'$, if $\Gamma$ is nonempty and $\Gamma'$ is connected (as will typically be the case throughout this paper), then $\theta$ is surjective; indeed, for all edges $e$ of $\Gamma'$, if $o(e)$ or $t(e)$ is in $\theta(\Gamma)$, then $e$ is also in $\theta(\Gamma)$.

\begin{lem}\label{lem:locally_surjective}
Let $\Gamma$ be a graph and let $G \le \Aut(\Gamma)$.  Then $\pi_{(\Gamma,G)}$ is locally surjective.
\end{lem}

\begin{proof}
Let $\pi = \pi_{(\Gamma,G)}$ and let $\Gamma' = \pi(\Gamma)$.
Let $v \in V\Gamma$ and $a' \in o\inv_{\Gamma'}(\pi(v))$.  Since $\pi$ is surjective, there exists $a \in A\Gamma$ such that $\pi(a) = a'$, and hence
\[
\pi(o(a)) = o(\pi(a)) = o(a') = \pi(v);
\]
there is then $g \in G$ such that $go(a) = v$, and hence $ga \in o\inv(v)$.  

Thus $\pi(ga) = \pi(a) = a'$; in particular, $a' \in \pi(o\inv_{\Gamma}(v))$.  Given the choice of $v$ and $a$, we conclude that $\pi$ is locally surjective.
\end{proof}

\subsection{Types of action on a tree}\label{sec:action_types}

In this subsection we describe the various types of action on a tree (see Theorem~\ref{thm:types}) and explore their properties (summarised in Table~\ref{fig:types} on page \pageref{fig:types}). We revisit these types in Section~\ref{TypesOfActionRevisited} where we characterise them in terms of their local action diagrams.

If $T$ is a tree, a \defbold{line} in $T$ is a two-way infinite simple path.  A \defbold{translation} of a line $L$ is an orientation-preserving automorphism of $L$ that does not fix any point on the line.

We recall that automorphisms of $T$ come in three types.

\begin{prop}[{\cite[Proposition~3.2]{Tits70}}]\label{prop:TitsClassification}
Let $T$ be a tree and let $g \in \Aut(T)$.  Then exactly one of the following holds:
\begin{enumerate}
\item $g$ fixes a vertex;
\item $g$ inverts an edge; 
\item there is a unique line $L$ in $T$ (called the \defbold{axis} of $g$) on which $g$ induces a translation.
\end{enumerate}
\end{prop}

If $g$ inverts an edge, we call it an \defbold{inversion} of $T$, and if $g$ translates a line we call it a \defbold{translation} of $T$ (or say it is \defbold{hyperbolic} on $T$).  Elements that are not translations are called \defbold{elliptic}.  Notice that an inversion fixes no ends of the tree, and a translation fixes exactly two ends (namely, the ends of its axis).

The \defbold{natural bipartition} or \defbold{vertex parity} of $VT$ is the partition of vertices of $T$ into two classes $V_0T$ and $V_1T$, such that two vertices are in the same class if and only if the distance between them is even.  We say an automorphism or group of automorphisms is \defbold{parity-preserving} if it preserves each part of the parity and \defbold{parity-reversing} otherwise.  In general, given $G \le \Aut(T)$, the parity-preserving automorphisms form a normal subgroup of $G$ of index at most $2$, called the \defbold{parity-preserving subgroup} of $G$.  Note that automorphisms that fix a vertex are parity-preserving, whereas inversions are parity-reversing.  A translation is parity-preserving if it translates vertices along its axis by an even distance and parity-reversing otherwise.

Given $a \in AT$, we define the associated \defbold{half-tree} to be the subgraph $T_a$ induced on the vertices $v$ such that $d(t(a),v) < d(o(a),v)$.  Note that $VT = VT_a \sqcup VT_{\ol{a}}$ and every end of $T$ is an end of exactly one of $T_a$ and $T_{\ol{a}}$, however $a$ and $\ol{a}$ do not lie in $T_a$ nor do they lie in $T_{\ol{a}}$.

A group $G$ acting on a tree $T$ is said to act \defbold{without inversion} if there is no pair $a \in AT$ and $g \in \Aut(T)$ such that $ga = \overline{a}$. The group $G$ acts \defbold{without translation} if no element of $G$ is a translation; otherwise we say $G$ acts \defbold{with translation} or $G$ \defbold{translates} $T$.  
An action on a tree $T$ is said to be \defbold{minimal} if it does not preserve any proper subtree of $T$.
An action is \defbold{geometrically dense} if it is minimal and does not preserve any end of $T$.
  Notice that for $G$ acting on $T$, the set of leaves $L$ of $T$ is $G$-invariant, and $T \smallsetminus L$ is connected and therefore a $G$-invariant subtree. Thus if the action of $G$ is geometrically dense, then either $T$ is leafless or $T$ consists of two vertices connected by a single edge, and $G$ contains element inverting that edge.

Geometrically dense actions are characterised by the action on the set of half-trees ordered by inclusion. The following is similar in spirit to many arguments concerning groups acting on trees; we can find no reference to this precise formulation so include a proof for completeness.

\begin{lem}\label{lem:geom_dense}
Let $T$ be a tree and let $G \le \Aut(T)$.  Then the following are equivalent:
\begin{enumerate}[(i)]
\item $G$ is geometrically dense;
\item For all $a \in AT$, there exists $g \in G$ such that $T_{ga} \subseteq T_{\ol{a}}$;
\item For all $a,b \in AT$, there exists $g \in G$ such that $T_{gb} \subseteq T_a$.
\end{enumerate}
\end{lem}

\begin{proof}
If $T$ has no edges then all three statements are vacuously true. If $T$ has only one edge, then (i)--(iii) are true if $G$ flips the edge and false otherwise.  If $T$ has more than one edge, but has a leaf, it is easy to see that (i)--(iii) are false. Thus we may assume $|o\inv(v)| \ge 2$ for every $v \in VT$.  Given $a,b \in AT$, note that $T_b \subseteq T_a$ if and only if $T_b$ is disjoint from $T_{\ol{a}}$.

Suppose (i) holds and let $a \in AT$; we suppose for a contradiction that $T_{\ol{a}}$ does not contain any $G$-translate of $T_a$.  Then for all $g,h \in G$, the half-trees $T_{ga}$ and $T_{ha}$ intersect.  Let $B_a$ be the set of half-trees $T_b$ such that $b \in Ga$ and $T_b \subseteq T_a$.  Given $T_{b_1},T_{b_2} \in B_a$, since these half-trees are not disjoint and both are contained in $T_a$, we see that one of $T_{b_1}$ and $T_{b_2}$ contains the other; that is, $B_a$ is totally ordered by inclusion.  Moreover, we see that $B_a$ has no infinite ascending chain, by considering the length of the path between $t(a)$ and $o(b)$ for $T_b \in B_a$.  If $B_a = \{T_a\}$, then we see that $t(a)$ is contained in $T_{ga}$ for every $g \in G$; in that case, the intersection $\bigcap_{g \in G}T_{ga}$ is a proper nonempty $G$-invariant subtree, contradicting (i).  Thus there is some $s(a) \in Ga$ such that $T_{s(a)}$ is maximal among elements of $B_a \smallsetminus \{T_a\}$; since $B_a$ is totally ordered, in fact the arc $s(a)$ is unique.  We can then form a sequence $a_i$ of arcs with $a_0 = a$ and $a_{i+1} = s(a)$, so that the trees $T_{a_i}$ form a descending chain; there is then a unique end $\xi$ representable by a directed ray that contains $a_i$ for all $i \ge 0$.  Since $G$ is geometrically dense, it does not fix this end, so $g\xi \neq \xi$ for some $g \in G$.  We then see that for some $i,j \ge 0$, the half-trees $T_{a_i}$ and $T_{ga_j}$ are disjoint, a contradiction.  Thus in fact $T_{\ol{a}}$ does contain a $G$-translate of $T_a$, proving (ii).

Suppose (ii) holds and let $a,b \in AT$.  By considering the relative orientation of $a$ and $b$ on a path containing both, we see that $T_{a'}$ contains $T_{b'}$ for some $a' \in \{a,\ol{a}\}$ and $b' \in \{b,\ol{b}\}$.  Then by (ii), $T_a$ contains a $G$-translate of $T_{a'}$, while $T_{b'}$ contains a $G$-translate of $T_b$.  Thus $T_a$ contains a $G$-translate of $T_b$ and (iii) follows.

Suppose (iii) holds.  Then clearly no $G$-orbit on vertices is confined to a half-tree; since any proper subtree of $T$ is contained in a half-tree, it follows that $G$ does not preserve any proper subtree.  Given an end $\xi$, then there is an arc $a$ such that $\xi$ is an end of $T_a$ and not an end of $T_{\ol{a}}$.  There is then $g \in G$ such that $T_{ga} \subseteq T_{\ol{a}}$, and so $g\xi \neq \xi$.  Thus $G$ does not fix an end, and (i) follows, completing the proof that (i)--(iii) are equivalent.
\end{proof}

Groups acting without translation have the following fixed point property.

\begin{lem}[{See \cite[Proposition~3.4]{Tits70}}]\label{lem:fixed_point}
Let $T$ be a tree and let $G \le \Aut(T)$.  Suppose $G$ acts without translation.  Then $G$ fixes a vertex, preserves an undirected edge, or fixes an end.
\end{lem}

For the purposes of studying groups acting on trees, it is useful to split the actions into six cases, as follows.

\begin{thm}[See \cite{Serre:trees}]\label{thm:types}
Let $G$ be a group acting on a tree $T$.  Then exactly one of the following holds:
\begin{description}
\item[(Fixed vertex)] $G$ fixes some vertex (not necessarily unique);
\item[(Inversion)] $G$ preserves a unique undirected edge and includes an inversion of that edge;
\item[(Lineal)] $G$ fixes exactly two ends and translates the line between them;
\item[(Horocyclic)] $G$ fixes a unique end, does not fix any vertices, and acts without translation;
\item[(Focal)] $G$ fixes a unique end and includes a translation towards this end;
\item[(General type)] $G$ acts with translation and does not fix any end.
\end{description}
\end{thm}

\begin{proof}
We may suppose that $G$ does not fix any vertex.  Suppose $Gv$ has finite diameter for some $v \in VT$.  Then there is a smallest subtree $T'$ containing $Gv$, and $T'$ has finite diameter.  In particular, $G$ preserves a subtree $T'$ of finite diameter.  By repeatedly pruning the leaves of $T'$, we arrive at a $G$-invariant subtree $T''$ in which all vertices are leaves.  Since $G$ does not fix any vertex, $T''$ must be a single undirected edge, and $G$ includes a reversal of that edge.  Thus $G$ is of inversion type.

From now on we may suppose that $G$ has unbounded action (that is, all orbits of vertices under $G$ are of unbounded diameter).  Suppose that $\xi_1$ and $\xi_2$ are distinct ends of $T$ fixed by $G$.  Then there is a unique line $L$ with ends $\xi_1$ and $\xi_2$, which is therefore $G$-invariant.  Since $G$ has unbounded action, it must translate $L$; it then follows that $G$ cannot fix any more ends.  Thus $G$ has lineal action.

From now on we may suppose that $G$ fixes at most one end.  Suppose that $G$ fixes an end $\xi$.  If $G$ acts without translation, then it is horocyclic, whereas if $G$ acts with translation, then it is focal.

In the remaining case, $G$ is unbounded and fixes no ends.  Then $G$ acts with translation by Lemma~\ref{lem:fixed_point}.
\end{proof}

From a geometric perspective, the fixed vertex and inversion cases can be grouped together as `bounded' (in other words, every orbit has finite diameter), but from the perspective of local actions it is useful to distinguish them.  Note that compact subgroups of $\Aut(T)$ have finite orbits, so they are bounded; the converse is true if $T$ is locally finite, but not otherwise.

If $G$ has a horocyclic action then every element in $G$ must be elliptic, and no element can be an inversion (because inversions fix no end); thus every element must fix a vertex. 
A corollary to our next lemma shows that such a group $G$ cannot be compactly generated.

\begin{lem}\label{lem:Serre}
Let $T$ be a tree and let $G$ be a compactly generated locally compact group acting continuously on $T$, such that every element of $G$ fixes a vertex of $T$.  Then $G$ fixes a vertex of $T$.
\end{lem}

\begin{proof}
Note that the hypotheses imply that $G$ acts without inversion. In the case that $G$ is finitely generated, the conclusion is \cite[I.6.5, Corollary 3]{Serre:trees}.

It is a general fact (see \cite[Lemma 2]{MollerFC}, for example) that given any compactly generated \tdlc group $G$ and compact open subgroup $U$ of $G$, there is a finitely generated subgroup $H$ of $G$ such that $G = UH$.  Since $G$ acts continuously on $T$, the action of $U$ has finite orbits.  Now $H$ fixes a vertex $v$ of $T$, which means that the orbit of $v$ under $G = UH$ is finite.  Let $S$ be the smallest subtree of $T$ spanned by $Gv$.  Then $S$ is a finite tree, so it has a canonical centre, which is either a vertex or a pair of adjacent vertices.  Thus $G$ either fixes a vertex or preserves a pair of adjacent vertices.  Since $G$ acts without inversion, in fact $G$ fixes a vertex.
\end{proof}

\begin{cor}\label{cor:horocyclic}
Let $T$ be a tree and let $G$ be a compactly generated locally compact group acting continuously on $T$.  Then the action is not horocyclic.  Consequently, either the action of $G$ is bounded, or $G$ acts with translation.
\end{cor}

Lineal and focal actions imply a special group structure.  To explain why, it is useful to borrow from CAT(0) geometry an idea introduced by H. Busemann (see \cite{Busemann}).  The interpretation of this idea for trees is easy to understand combinatorially.  The next two lemmas are well-known, but we give proofs here for clarity.

\begin{defn}
Let $T$ be a tree and let $\xi$ be an end of $T$.  A \defbold{(combinatorial) Busemann function with focus $\xi$} is a function $b: VT \rightarrow \Zb$ such that for all arcs $a$ pointing towards $\xi$ (that is, such that $\xi$ is an end of $T_a$), we have $b(o(a)) - b(t(a)) = 1$.  A \defbold{horosphere}, respectively \defbold{horoball centred at $\xi$} is a set of the form $\{v \in VT \mid b(v)=n\}$, respectively $\{v \in VT \mid b(v) \le n\}$, for some Busemann function $b$ and $n \in \Zb$.
\end{defn}

\begin{lem}\label{lem:Busemann}
Let $T$ be a tree.  Choose some vertex $v_0 \in VT$ and an end $\xi$ of $T$.
\begin{enumerate}[(i)]
\item For each $n \in \Zb$ there is exactly one Busemann function $b_n$ with focus $\xi$ such that $b_n(v_0) = n$.  Moreover, $b_n(v) = b_0(v)+n$ for all vertices $v$.
\item Let $b$ be a Busemann function with focus $\xi$.  Then for all $g \in \Aut(T)$ the function $g.b$ is a Busemann function with focus $g.\xi$, where $g.b(v) := b(g\inv v)$.  Consequently, given a group $G \le \Aut(T)$ such that $G$ fixes $\xi$, the function
\[
\beta_{\xi}: G \rightarrow \Zb; \; \beta_{\xi}(g) = g.b(v) - b(v) 
\]
is a homomorphism that does not depend on the choices of $b$ or $v$.
\item \label{lem:Busemann:item:III} Let $G$ be a group of automorphisms of $T$ fixing $\xi$, let $\beta_{\xi}$ be as in (ii) and let $g \in G$.  If $\beta_{\xi}(g) > 0$, then $g$ is a translation with attracting end $\xi$; if $\beta_{\xi}(g)<0$, then $g$ is a translation with repelling end $\xi$; and if $\beta_{\xi}(g)=0$, then $g$ fixes pointwise a ray representing $\xi$.
\end{enumerate}
\end{lem}

\begin{proof}
Fix a ray $(v_0,v_1,\dots)$ of vertices representing $\xi$.

(i)
We set
\[
b_0(v) = \lim_{t \rightarrow \infty} d(v,v_t) - t
\]
where $d$ is the usual graph metric on $VT$.  It is easy to see that $b_0(v_t) = -t$ for all $t \ge 0$.  For a general vertex $v$, we see that for all but finitely many $t \ge 0$ and all $t' > t$, the path from $v$ to $v_{t'}$ passes through $v_t$; we then see that
\[
d(v,v_{t'}) - t' = d(v,v_t) - t,
\]
and hence for any such choice of $t$, we have $b_0(v) = d(v,v_t) - t$.  Thus $b_0$ is a well-defined integer-valued function.  Given an arc $a$ terminating at $v$ such that $\xi$ is an end of $T_a$, then the path from $o(a)$ to $v_{t'}$ passes through $v$ and $v_t$ for all sufficiently large $t$ and $t' > t$; hence, for $t$ sufficiently large we have
\[
b_0(o(a)) = d(o(a),v_t) - t = 1 + d(v,v_t) - t = b_0(t(a)) + 1.
\]
Thus $b_0$ is a Busemann function with $b_0(v_0) = 0$.  Clearly also $b_n(v) := b_0(v)+n$ defines a Busemann function such that $b_n(v_0)=n$, for all $n \in \Zb$.  To see that $b_n$ is unique for all $n \in \Zb$, note that any Busemann function $b$ is uniquely specified by its value at a given vertex: once we know $b(v)$ for some vertex $v$, then for $a \in o\inv(v)$, we have $b(t(a)) = b(v)-1$ if $a$ points towards $\xi$ and $b(t(a))=b(v)+1$ otherwise, so we also know $b(w)$ for every neighbour $w$ of $v$, and hence by connectedness the value is specified at every vertex.

(ii)
Let $g \in \Aut(T)$.  If $a \in AT$ points towards $g.\xi$, then $g\inv a$ points towards $\xi$, so
\[
g.b(o(a)) - g.b(t(a)) = b(g\inv o(a)) - b(g\inv t(a)) =  b(o(g\inv a)) - b(t(g\inv a)) = 1.
\]
Hence $g.b$ is a Busemann function with focus $g.\xi$.

Now suppose that $G \le \Aut(T)$ fixes $\xi$ and let $g \in G$.  By (i), we have $b = b_n$ and $g.b = b_m$ for some $m,n \in \Zb$; thus $g.b(v)-b(v) = m-n$, independent of the choice of $v$.  Clearly adding a constant to $b$ commutes with the action of $g$ on the set of Busemann functions, so $g.b(v)-b(v)$ is also independent of the choice of $b$.  To see that $\beta:= \beta_{\xi}$ is a homomorphism, we see that
\[
\beta(gh) = gh.b(v) - b(v) = g.(h.b)(v) - (h.b)(v) + h.b(v) - b(v) = \beta(g) + \beta(h),
\]
using the fact that $h.b$ is also a Busemann function.

(iii)
Let $G \le \Aut(T)$ fix $\xi$ and let $g \in G$.  Then $(g\inv .v_0,g\inv.v_1,\dots)$ is a ray representing $\xi$, so there is some $n \in \Zb$ such that for all but finitely many $t \ge 0$, we have $g\inv .v_t = v_{t-n}$.  We then have
\[
\beta(g) = b(v_{t-n}) - b(v_t) = n.
\]
Moreover, we see that $\xi$ is an attracting end of $g$ if and only if $n >0$ and a repelling end if $n < 0$.  If $n=0$, then $g$ fixes pointwise the ray $(v_t,v_{t+1},\dots)$ for $t$ sufficiently large.
\end{proof}

We now see why, in the taxonomy of Theorem~\ref{thm:types}, the name {\it horocyclic} is used. Suppose $(G, T)$ is horocyclic with fixed end $\xi$, and for some $n \in \Zb$ and some Busemann function $b$ with focus $\xi$, let $H = \{v \in VT \mid b(v)=n\}$ be a horosphere. By Lemma~\ref{lem:Busemann}(\ref{lem:Busemann:item:III}), we have $\beta_\xi(g) = 0$ for any $g \in G$. Therefore $gH = \{gv \in VT \mid b(v)=n\} = \{w \in VT \mid g.b(w) = n\} = H$. Thus for a horocyclic group $G$, each Busemann function, each horoball and each horocycle is invariant under $G$.

\begin{lem}\label{lem:Busemann:bis}
Let $T$ be a tree and let $G \le \Aut(T)$, equipped with the subspace topology in $\Aut(T)$.
\begin{enumerate}[(i)]
\item If $G$ is lineal or focal, then $G$ decomposes as a semidirect product $E \rtimes \grp{s}$, where $E$ is an open subgroup consisting of all elliptic elements of $G$ and $s$ is a translation.  If $G$ is lineal then $E$ fixes pointwise the axis of $\xi$, whereas if $G$ is focal then $E$ is horocyclic.
\item Suppose $G$ is topologically perfect (i.e., the commutator $[G, G]$ subgroup is dense in $G$). Then $G$ is of fixed vertex, horocyclic or general type, and if $G$ is of general type then $G$ does not preserve any line of $T$.
\end{enumerate}
\end{lem}

\begin{proof}
(i)
Let $\xi$ be an end of $T$ fixed by $G$ and take vertices $v_t$ for $t \ge 0$ and the homomorphism $\beta := \beta_{\xi}$ as in Lemma~\ref{lem:Busemann}.  Then given $g \in G$, we see that $\beta(g) = 0$ if and only if $g$ is elliptic.  Thus $E: = \ker\beta$ is a subgroup consisting of all the elliptic elements; in particular $E$ contains all vertex stabilisers, so it is open.  Taking a translation $s$ in $G$ such that $|\beta(s)|$ is minimised, we then see that for each $g \in G$ the value $|\beta(s)|$ must divide $|\beta(g)|$. Consequently, $\beta(\grp{s}) = \beta(G)$ and hence $G = E \rtimes \grp{s}$.  Now $E$ acts without translation and fixes an end; by considering the six types of action on a tree, we conclude that $E$ is of fixed vertex or horocyclic type.  To finish the proof of this part, we argue that $E$ fixes a vertex if and only if $G$ is lineal.

Suppose $E$ fixes a vertex $v$.  Then, being normal, $E$ also fixes $s^k.v$ for all $k \in \Zb$, and hence fixes pointwise the smallest subtree $T'$ containing $\{s^k.v \mid k \in \Zb\}$.  Since $T'$ is an $\grp{s}$-invariant subtree, it contains the axis of $s$.  Thus $E$ fixes pointwise the axis of $s$, and hence $G$ preserves the axis of $s$, showing that $G$ is lineal.

Conversely, suppose $G$ is lineal and let $\xi$ and $-\xi$ be the two fixed ends of $G$, in other words the ends of the axis $L$ of $s$.  Then given $g \in E$, by Lemma~\ref{lem:Busemann}, $g$ fixes pointwise a ray $R$ representing $\xi$ and $g$ also fixes pointwise a ray $-R$ representing $-\xi$.  The union of these two rays contains all but finitely many vertices of $L$, while the remaining vertices of $L$ lie on the shortest path from $-R$ to $R$, so in fact $g$ fixes $L$ pointwise.  Thus $E$ fixes $L$ pointwise.

(ii)
Since $G$ is topologically perfect, it has no proper open normal subgroup $N$ such that the quotient $G/N$ is abelian. Hence, by (i), $G$ cannot be lineal or focal.  If $G$ is of inversion type, or $G$ is of general type and preserves a line $L$, then it has an open subgroup $H$ of index $2$: in the former case we take $H$ to be the parity-preserving subgroup of $G$, and in the latter case we take $H$ to be the subgroup that preserves the orientation of $L$.  In either case, $G/H$ is a discrete abelian group of order $2$, which contradicts the hypothesis that $G$ is topologically perfect.  Thus $G$ must be of one of the remaining types, that is, fixed vertex, horocyclic or general type, and in the general type case, $G$ cannot preserve any line.
\end{proof}

The \defbold{smallest invariant subtree} of an action $(T, G)$  is a proper nonempty subtree $T'$ of $T$ such that $gT' = T'$ and such that $T' \subseteq T''$ for every $G$-invariant subtree $T''$ of $T$. Note that if $(T, G)$ has a smallest invariant subtree then the smallest invariant subtree is necessarily unique.

\begin{lem}\label{lem:minimal_invariant}
Let $(T,G)$ be an action on a tree.  Then $(T,G)$ has a smallest invariant subtree $T'$ exactly when one of the following holds:
\begin{enumerate}[(i)]
\item $G$ fixes a unique vertex, in which case $T'$ is that vertex;
\item $G$ is of inversion type, in which case $T'$ is the invariant undirected edge;
\item \label{lem:minimal_invariant:translation}
$G$ acts with translation, in which case $T'$ is the union of the axes of translation of $G$, and $T'$ is either a line or an infinitely ended tree. In particular, if $G$ is of general type then $G$ acts geometrically densely on $T'$.
\end{enumerate}
\end{lem}

\begin{proof}
In cases (i) and (ii), the smallest invariant subtree is given by \cite[Corollaire~3.5]{Tits70} and its proof. If $G$ acts with translation, then there is a smallest invariant subtree $T'$ by \cite[Corollaire~3.5]{Tits70}, and moreover $T'$ is infinite.  By \cite[Lemma~2.1(iii)]{MollerVonk}, $T'$ is the union of the axes of translation of $G$ on $T$.  If $T'$ is not a line, then there are at least two distinct axes; since the powers of a translation $g$ do not fix any end other than the ends of the axes of $g$, it follows that $G$ has an infinite orbit on the ends of $T'$, and in particular, $T'$ is infinitely ended.  If $G$ is of general type, then $G$ fixes no end, so clearly the action of $G$ on $T'$ is geometrically dense.

By Theorem~\ref{thm:types}, the remaining cases are that $G$ fixes more than one vertex, or that $G$ is horocyclic.  In the former case, distinct fixed vertices give disjoint invariant subtrees, so there is no smallest invariant subtree. If $G$ is horocyclic with fixed end $\xi$, we see from Lemma~\ref{lem:Busemann} that there is a $G$-invariant Busemann function with focus $\xi$, so each of the horoballs centred at $\xi$ spans a $G$-invariant subtree; the horoballs centred at $\xi$ have empty intersection, and hence there is no smallest invariant subtree.
\end{proof}

Table~\ref{fig:types} summarises the properties implied by the six types of action on a tree.  The last column describes the smallest invariant subtree.

\begin{table}[h]
\caption{The types of action on a tree}\label{fig:types}%
\begin{tabular}{@{}llllll@{}}
\toprule
Type & Bounded & Inversions & Translations & fixed ends  & s. i. s. \\ 
\midrule
Fixed vertex & Yes & No & No & any & vertex or $\not\exists$ \\
Inversion & Yes & Yes & No & 0 & edge \\
Lineal & No & No & Yes  & 2 & line \\
Horocyclic & No & No & No & 1 & $\not\exists$ \\
Focal & No & No & Yes & 1 & $\infty$-ended \\
General & No & Maybe & Yes & 0 & line or $\infty$-ended \\
\botrule
\end{tabular}
\end{table}

General type action implies a useful dichotomy in the normal subgroups of $G$.

\begin{lem}\label{lem:general_type}
Let $T$ be a tree, let $G \le \Aut(T)$ act with translation, let $T'$ be the smallest invariant subtree of $G$ and let $N$ be a normal subgroup of $G$.  Suppose that $N$ acts nontrivially on $T'$.  Then $N$ is unbounded.  If $G$ is of general type and $T'$ has more than two ends, then $N$ is of general type and the smallest invariant subtree of $N$ is the same as that of $G$.
\end{lem}

\begin{proof}
Suppose that $N$ fixes a vertex.  Then there is a unique smallest subtree $T_0$ containing all the fixed vertices of $N$, and we see that in fact, $N$ fixes $T_0$ pointwise.  In particular, since $N \unlhd G$, $T_0$ is $G$-invariant, so $T' \subseteq T_0$.  But then $N$ acts trivially on $T'$, a contradiction.  Thus $N$ has no fixed vertices.  A similar argument shows that $N$ does not preserve any undirected edge.  So $N$ is unbounded.

Now suppose that $G$ is of general type and $T'$ has more than two ends, and consider the set $E$ of ends fixed by $N$.  Then $E$ is $G$-invariant, and since $N$ has unbounded action, $|E| \le 2$.  However, $G$ cannot fix any end or preserve a set of two ends, since $T'$ has more than two ends.  So $E = \emptyset$.  We deduce that $N$ has general type, with minimal subtree $T''$ spanned by the axes of translation of $N$.  Since $N \le G$ we see that $T'' \subseteq T'$; but since $N \unlhd G$, the tree $T''$ is $G$-invariant, so $T' \subseteq T''$.  Thus $T' = T''$.
\end{proof}

We also note a sufficient condition for $G$ to have finitely many orbits on its smallest invariant subtree.

\begin{lem}[{See also \cite[Lemma~2.4]{CapDeM}}]\label{lem:finite_type}
Let $(T,G)$ be an action with translation on a tree, and let $T'$ be the smallest invariant subtree of $T$.  Suppose that $G$ is generated by finitely many cosets of $G_v$, for some $v \in VT'$.  Then $G$ has finitely many orbits on $VT' \sqcup AT'$.
\end{lem}

\begin{proof}
We can take a symmetric generating set for $G$ of the form $S = F \cup G_v$, where $F$ is finite.  Then the set $\{sv \mid s \in S\}$ is finite and for each $s \in S$, the path $[v,sv]$ from $v$ to $sv$ is contained in $T'$.  Let $T''$ be the subtree spanned by the paths $[v,sv]$ as $s$ ranges over $S$.  Then $T''$ is finite and for each $s \in S$, the graph $T'' \cup sT''$ is connected: specifically, both $T''$ and $sT''$ are connected and contain $sv$.  From here, we see that the graph $\bigcup_{g \in G} gT''$ is also connected, and hence equal to $T'$.  This shows that $G$ has finitely many orbits on $VT' \sqcup AT'$.
\end{proof}

\subsection{The $\propP{}$-closure, property $\propP{}$ and group topology} 
\label{subsec:P-closure}

Recall from the introduction that $G \leq \Aut(T)$ is $\propP{}$-closed if and only if it is closed as a subgroup of $\Aut(T)$ and has Tits' independence property $\propP{}$ (see (Definition~\ref{Def:TitsP})). In this subsection we describe ideas related to property $\propP{}$.
We focus first on a natural notion called the $\propP{k}$-closure of $G$, denoted $G^{\propP{k}}$, originally due to Banks, Elder  and Willis, and we carefully extend this notion to non-locally finite trees in Definition~\ref{PropPk}.
As a consequence of Proposition~\ref{prop:kclosure} and Theorem~\ref{propertyP_oneclosure} we see that for all trees (not just those that are locally-finite) being $\propP{1}$-closed in the sense of Definition~\ref{PropPk} is equivalent to being $\propP{}$-closed. This allows us to define the $\propP{}$-closure of $G$ in Definition~\ref{def:PClosed}.

We define important subgroups $G^+$ and $G^{++}$ in Definition~\ref{def:Gplus_Gplusplus} then state Tits' famous simplicity theorem  for $G^+$ (Theorem~\ref{thm:Tits}) and M\"{o}ller and Vonk's simplicity theorem for $G^{++}$ (Theorem~\ref{thm:MollerVonk}). Following this we explore the concept of a group having closed local actions. We conclude the subsection with elementary results relating the permutational and topological properties of groups.

Given a set $X$, we equip $\Sym(X)$ with the permutation topology (see \S\ref{sec:groups}) and note that this is the coarsest group topology such that the stabiliser of every $x \in X$ is open.  Given a tree $T$, we give $\Aut(T)$ the subspace topology, regarding $\Aut(T)$ as a subgroup of $\Sym(VT)$.  Observe that in fact $\Aut(T)$ corresponds to a closed subgroup of  $\Sym(VT)$; if $VT$ is countable, this ensures that $\Aut(T)$ is Polish (that is, separable and completely metrisable) and also totally disconnected, but $\Aut(T)$ is not necessarily locally compact.  Assuming $T$ has no leaves, one could equivalently define the topology of $\Aut(T)$ with respect to the permutation topology on arcs or undirected edges, or the permutation topology on one part of the natural bipartition of the vertices of $T$: this can be seen by noting that two undirected edges suffice to specify a vertex, two vertices to specify an arc, and two vertices in one part of the bipartition to specify a vertex in the other part.

\begin{defn} \label{PropPk}
Given 
$G \le \Aut(T)$ and $k \ge 1$, the \defbold{$\propP{k}$-closure} of $G$, denoted by $G^{\propP{k}}$, is the set of automorphisms $g \in \Aut(T)$ such that for all $v \in VT$, and every finite set of vertices $X$ all of which are at distance at most $k$ from $v$, there exists $g_X \in G$ such that $gw = g_Xw$ for every vertex $w \in X$.  We say $G$ is \defbold{$\propP{k}$-closed} if $G = G^{\propP{k}}$.
 
This definition 
appears independently in 
\cite{BanksElderWillis} in the context of groups acting on trees with independence (where it is called the $k$-closure\footnote{The term $k$-closure has a well-established meaning in permutation group theory due to Wielandt.}), and in  \cite{SamShep} in the context of generalising Leighton's Theorem (where it is called the $k$-symmetry-restricted closure), but in both cases this property is defined only for locally finite trees. We caution the reader that care must be taken when extending the definition from these papers to trees that are not locally finite in order to avoid having $\propP{k}$-closed groups that are not closed. Note that in our definition we specify the condition on finite subsets of $B_k(v)$ rather than $B_k(v)$ itself in order to ensure that the $\propP{k}$-closure is always closed in the permutation topology. For locally finite trees our definition is equivalent to that given by Banks, Elder and Willis in \cite{BanksElderWillis}.
\end{defn}

We will use some basic properties of the $\propP{k}$-closure.  We include proofs here because the approach of \cite{BanksElderWillis} implicitly assumes that trees are locally finite.

For the rest of this article, we define $G_{(X)} := \{g \in G \mid \forall x \in X: gx = x\}$, where $X$ is a set of vertices of $T$.

\begin{prop}[See \cite{BanksElderWillis} Proposition~3.4]\label{prop:kclosure}Let $T$ be a tree, let $G \le \Aut(T)$ and let $k \in \Nb$.
\begin{enumerate}[(i)]
\item $G^{\propP{k}}$ is a closed subgroup of $\Aut(T)$.
\item $G^{\propP{l}} = (G^{\propP{k}})^{\propP{l}}$ whenever $l \le k$.  In particular, $(G^{\propP{k}})^{\propP{k}} = G^{\propP{k}}$, so $G^{\propP{k}}$ is $\propP{k}$-closed.
\end{enumerate}
\end{prop}

\begin{proof}  
(i)
Write $A := \Aut(T)$.  Let $g,h \in G^{\propP{k}}$, let $v \in VT$ and let $X$ be a finite set of vertices all of which are at distance at most $k$ from $v$.  Then there exists $h_X \in G$ such that $h_Xw = hw$ for all $w \in X$.  In turn, $hX := \{hw \mid w \in X\}$ is a finite set of vertices, all of which are at distance at most $k$ from $hv$, so there exists $g_{hX} \in G$ such that $g_{hX}w = gw$ for all $w \in hX$.  Thus $g_{hX}h_X$ is an element of $G$ such that $g_{hX}h_Xw = ghw$ for all $w \in X$.  We conclude that $gh \in G^{\propP{k}}$.  Similarly, it is clear that $G^{\propP{k}}$ is closed under inverses.  Thus $G^{\propP{k}}$ is a subgroup of $A$.

Let $\mathcal{X}_k$ be the set of all finite sets $X$ of vertices in $T$, such that there is a vertex $v$ at distance at most $k$ from every vertex in $X$.  Then $A_{(X)}$ is an open subgroup of $A$ for every $X \in \mathcal{X}_k$.  Observe that given $g \in A \smallsetminus G^{\propP{k}}$, then there exists $X_g \in \mathcal{X}_k$ such that no element of $G$ agrees with $g$ on $X_g$, and hence no element of $G^{\propP{k}}$ agrees with $g$ on $X_g$, that is, $G^{\propP{k}} \cap gA_{(X_g)} = \emptyset$.  We can therefore express the complement of $G^{\propP{k}}$ as the following union of open sets:

\[
A \smallsetminus G^{\propP{k}} = \bigcup_{g \in A \smallsetminus G^{\propP{k}}} gA_{(X_g)}.
\]

Hence $G^{\propP{k}}$ is closed in $\Aut(T)$.

(ii)
Since $G \le G^{\propP{k}}$ then $G^{\propP{l}} \le  (G^{\propP{k}})^{\propP{l}}$.  Let $g \in (G^{\propP{k}})^{\propP{l}}$ and let $X$ be a finite set of vertices of $T$, all of which are at distance at most $l$ from some vertex $v$.  Then there exists $g_X \in G^{\propP{k}}$ such that $g_Xw = gw$ for all $w \in X$.  But then all the vertices in $X$ are at distance at most $k$ from $v$, so there exists $g'_X \in G$ such that $g'_Xw = g_Xw = gw$ for all $w \in X$.  Hence $G^{\propP{l}} = (G^{\propP{k}})^{\propP{l}}$.  The remaining conclusions are clear.
\end{proof}

It is useful to note that the property of being $\propP{k}$-closed is inherited by fixators of vertices and is stable under taking intersections of subgroups.

\begin{lem}\label{kclosed_intersection}
Let $T$ be a tree, let $k$ be a positive integer and let $\mc{G}$ be a family of $\propP{k}$-closed subgroups of $\Aut(T)$.  Then $H = \bigcap_{G \in \mc{G}}G$ is $\propP{k}$-closed.
\end{lem}

\begin{proof}
Let $h \in H^{\propP{k}}$ and let $G \in \mc{G}$.  Then for each $v \in VT$ and each finite set $X$ of vertices in $B_k(v)$, there is some $g_X \in H$ such that $g_Xw = hw$ for all $w \in X$.  In particular, $g_X \in G$.  Since $G$ is $\propP{k}$-closed, it follows that $h \in G$; since $G \in \mc{G}$ was arbitrary, in fact $h \in H$.  Thus $H = H^{\propP{k}}$, so $H$ is $\propP{k}$-closed.
\end{proof}

\begin{lem}\label{oneclosed_stabiliser}Let $T$ be a tree, let $G \le \Aut(T)$, let $k \ge 1$ and let $X$ be a set of vertices of $T$.  If $G$ is $\propP{k}$-closed, then so is $G_{(X)}$.\end{lem}

\begin{proof}Suppose $G$ is $\propP{k}$-closed, and let $H = (G_{(X)})^{\propP{k}}$.  Given $v \in X$ and $g \in H$, we see from the definition of the $\propP{k}$-closure that $g$ must fix $v$.  So in fact $H$ is a subgroup of $(G^{\propP{k}})_{(X)}$, which is just $G_{(X)}$.  Hence $H = G_{(X)}$.\end{proof}

We also recall Tits' property $\propP{}$, introduced in \cite{Tits70}.

\begin{defn} \label{Def:TitsP}
Let $T$ be a tree and let $\theta: G \rightarrow \Aut(T)$ be a group homomorphism.  Given a nonempty (finite or infinite)  simple  path $L$ in $T$, let $\pi_L:VT \rightarrow VL$ be the closest point projection of the vertices of $T$ onto $L$; observe that for each $x \in L$, the set $\pi^{-1}_L(x)$ is a nonempty subtree of $T$.  Write $\theta(G)_{(L)}$ for the pointwise stabiliser of $L$ (so $\theta(G)_{(L)}$ preserves setwise each of the fibres $\pi^{-1}_L(x)$ of $\pi_L$).  Then for each vertex $x \in L$, there is a natural homomorphism $\phi_x: \theta(G)_{(L)} \rightarrow \Sym(\pi^{-1}_L(x))$ induced by the action of $\theta(G)_{(L)}$ on $\pi^{-1}_L(x)$.  We can combine the homomorphisms $\phi_x$ in the obvious way to obtain a homomorphism
\[
\phi_L: \theta(G)_{(L)} \rightarrow \prod_{x \in VL} \phi_x(\theta(G)_{(L)}).
\]
In general, $\phi_L$ is injective but not necessarily surjective.  We say $G$ (or more precisely, the action of $G$ on $T$) has \defbold{Tits' independence property $\propP{}$} or briefly \defbold{property $\propP{}$} (with respect to a collection $\mathcal{L}$ of simple  paths) 
if $\phi_L$ is surjective for every possible choice of $L$ (such that $L \in \mathcal{L}$).
\end{defn}

A major motivation of \cite{BanksElderWillis} was to generalise Tits' property $\propP{}$, and indeed property $\propP{}$ has a natural interpretation in terms of the $\propP{1}$-closure.

\begin{thm}[See \cite{BanksElderWillis} Theorem 5.4 and Corollary 6.4]\label{propertyP_oneclosure}Let $T$ be a tree and let $G$ be a closed subgroup of $\Aut(T)$.  Then $G = G^{\propP{1}}$ if and only if $G$ satisfies Tits' property $\propP{}$.  Furthermore, if $G$ has property $\propP{}$ with respect to the edges of $T$, then $G = G^{\propP{1}}$, so $G$ has property $\propP{}$ with respect to all  simple  paths.\end{thm}

\begin{proof}Let $L$ be a nonempty  simple  path in $T$, and let $g \in \Aut(T)_{(L)}$ such that
\[
{\phi_L(g)\in \prod_{v \in VL}\phi_v(G_{(L)})};
\]
say $\phi_L(g) = (s_v)_{v \in VL}$.  We now claim that $g \in G^{\propP{1}}$ (indeed, $g \in (G_{(L)})^{\propP{1}}$).  Let $X$ be a finite set of vertices, all adjacent to some vertex $w$ of $T$.  
We will show that there exists $g_X \in G_{(L)}$ such that $g$ agrees with $g_X$ on $X$.
We may assume that $X \cap VL = \emptyset$, since the vertices of $L$ all are all fixed by both $g$ and $G_{(L)}$.  Let $x = \pi_L(w)$.  We observe that since $X$ is disjoint from $VL$ and the set $X \cup \{w\}$ spans a subtree, any path from $X$ to $L$ must pass through $x$, in other words $X \subseteq \pi^{-1}_L(x)$.  There is then $g_X \in G_{(L)}$ such that $\phi_v(g_X) = s_x$, so that $g$ agrees with $g_X$ on $\pi^{-1}_L(x)$ and in particular on $X$.  Given the freedom of choice of $X$, we conclude that $g \in G^{\propP{1}}$ as claimed.  Thus if $G = G^{\propP{1}}$, then $G$ has property $\propP{}$.

Conversely, suppose that $G$ is closed and satisfies property $\propP{}$ with respect to the edges of $T$.  Suppose that $G \neq G^{\propP{1}}$ and let $g \in G^{\propP{1}} \smallsetminus G$.  Since $G$ is closed, the set $G^{\propP{1}} \smallsetminus G$ is a neighbourhood of $g$ in $G^{\propP{1}}$, so there is a finite set $X$ of vertices such that $g(G^{\propP{1}})_{(X)} \cap G = \emptyset$.  Let $S$ be the smallest subtree of $T$ containing $X$; note that $\Aut(T)_{(X)} = \Aut(T)_{(S)}$, since every vertex of $S$ lies on the  shortest  path between a pair of vertices in $X$.  Let us suppose that $X$ has been chosen so that $|S|$ is minimised.

By the definition of $G^{\propP{1}}$, we see that $S$ is not a star, so for every $x \in S$, there is a vertex in $S$ at distance $2$ from $x$. Hence there exist adjacent vertices $x$ and $y$ of $S$ such that neither $x$ nor $y$ is a leaf of $S$.  Let $L$ be the path formed by  the single arc $(x,y)$.  By the minimality of $|S|$, there is some $h \in G$ such that $gx = hx$ and $gy = hy$, so that $h^{-1}g$ fixes $L$ pointwise.  Let
\[
S_1 = (S \cap \pi^{-1}_L(x)) \cup \{y\} \text{ and } S_2 = (S \cap \pi^{-1}_L(y)) \cup \{x\}.
\]
Note that for $i=1,2$, then $S_i$ is the set of vertices of a subtree of $S$ that contains $L$.  The condition that neither $x$ nor $y$ is a leaf of $S$ ensures that there is some neighbour of $x$ in $S$ that is not contained in $S_2$, and similarly there is some neighbour of $y$ in $S$ that is not contained in $S_1$.  Hence $S_1$ and $S_2$ are both proper subtrees of $S$, so by the minimality of $|S|$, there exists $h_1,h_2 \in G$ such that $h_iw_i = h^{-1}gw_i$ for all $w_i \in S_i$ ($i=1,2$).  Indeed, $h_1$ and $h_2$ are elements of $G_{(L)}$, since $h_1$ and $h_2$ both agree with $h^{-1}g$ on $L$.  In particular, we see that the action of $h^{-1}g$ induces an element of $\phi_x(G_{(L)}) \times \phi_y(G_{(L)})$.  But then by (the restricted) property $\propP{}$, we have $h^{-1}g \in G_{(L)}$ and hence $g \in G$, a contradiction.
\end{proof}

If $G \leq \Aut(T)$ is closed with Tits' property $\propP{}$ then $G = G^{\propP{1}}$ by Theorem~\ref{propertyP_oneclosure}. On the other hand, if $G \leq \Aut(T)$ satisfies $G = G^{\propP{1}}$ then  Proposition~\ref{prop:kclosure} guarantees that $G$ is closed and so by Theorem~\ref{propertyP_oneclosure} we have that $G$ satisfies Tits' property $\propP{}$.
Thus, $G$ is $\propP{1}$-closed if and only if $G$ is closed in $\Aut(T)$ and has Tits' property $\propP{}$. This observation motivates the following definition.

\begin{defn} \label{def:PClosed}
Recall from the introduction that an action on a tree is $\propP{}$-closed if it is closed and has Tits' independence property $\propP{}$. From now on we can refer to the \defbold{$\propP{}$-closure} of an action $G \leq \Aut(T)$, written $G^{\propP{}}$, meaning the smallest $\propP{}$-closed subgroup of $\Aut(T)$ that contains $G$. 
We will use repeatedly without comment the fact that being $\propP{}$-closed is equivalent to being $\propP{1}$-closed.
\end{defn}

\begin{lem}\label{lem:Pclosed_subtree}
Let $T$ be a tree, let $G \le \Aut(T)$ and let $T'$ be a $G$-invariant subtree of $T$.  Suppose that $(T,G)$ 
is $\propP{}$-closed and that the action of $G$ on $T'$ is closed.  Then $(T',G)$ 
is also $\propP{}$-closed.
\end{lem}

\begin{proof}
Let $L$ be a  simple  path in $T'$.  Since $(T,G)$ has property~$\propP{}$, the natural homomorphism
\[
\phi_L: G_{(L)} \rightarrow \prod_{x \in VL} \phi_x(G_{(L)})
\]
is surjective, where $\phi_L$ and $\phi_x$ are defined with respect to $T$.  Now consider what happens if we replace $\phi_L$ and $\phi_x$ with $\phi'_L$ and $\phi'_x$ respectively, which are now defined with respect to $T'$.  We also have closest point projections $\pi_L$ for $L$ as a  simple path in $T$ and $\pi'_L$ for $L$ as a  simple path in $T'$, but in fact $\pi_L$ and $\pi'_L$ agree on $VT'$.  If we choose $g_x \in \phi'_x(G_{(L)})$ for each $x \in VL$, then there is some $h_x \in G_{(L)}$ such that $\phi'_x(h_x) = g_x$, and then by the surjectivity of $\phi_L$, there is $g \in G_{(L)}$ such that $\phi_L(g) = (\phi_x(h_x))_{x \in VL}$.  But then since
\[
(\pi')\inv_L(x) = \pi\inv_L(x) \cap VT' \subseteq \pi\inv_L(x),
\]
we immediately see that
\[
\phi'_L(g) = (\phi'_x(h_x))_{x \in VL} = (g_x)_{x \in VL}.
\]
Thus $\phi'_L$ is surjective, so $G$ has property~$\propP{}$ on $T'$. 
\end{proof}

\begin{lem}\label{lem:StronglyClosedFromCompactArcs} Let $T$ be a tree, let $G \le \Aut(T)$ be closed and let $T'$ be an invariant subtree for $(T,G)$.  If  there is an arc $a \in AT'$ such that the arc stabiliser $G_a$ is compact in $(T, G)$, then  the action of $G$ on $T'$  is closed.
\end{lem}
\begin{proof} Suppose $T'$ is an invariant subtree of $(T, G)$, and $a \in AT'$ is such that $G_a$ is compact in $\Aut(T)$. The action of $G$ on $T'$ is continuous because $(T, G)$ is Hausdorff and vertex stabilisers $G_v$ for $v \in VT'$ are open in $(T, G)$. The continuous image of the compact group $G_a$ is compact in $(T', G)$, and therefore closed in the Hausdorff group $\Aut(T')$. In other words, the stabiliser in $(T', G)$ of $a \in T'$  is a closed subgroup of $\Aut(T')$, and it follows then that $(T', G)$ is a closed subgroup of $\Aut(T')$.
\end{proof}

Property~$\propP{}$ was originally introduced as a sufficient condition to obtain a simple group.  We recall two relevant results from the literature.

\begin{defn} \label{def:Gplus_Gplusplus}
Given $G \le \Aut(T)$, write $G^+$ for the subgroup of $G$ generated by arc stabilisers in $G$.  Given $a \in AT$, write $\rist_G(T_a)$ for the rigid stabiliser of $T_a$, in other words, the pointwise stabiliser of $T_{\ol{a}}$; write $G^{++}$ for the closure of the subgroup of $G$ generated by the rigid stabilisers of half-trees in $G$.
\end{defn}

\begin{thm}[\cite{Tits70} Th\'{e}or\`{e}me~4.5]\label{thm:Tits} Let $T$ be a tree and let $G$ be a geometrically dense subgroup of $\Aut(T)$ with property $\propP{}$.  Then every nontrivial subgroup of $G$ normalised by $G^+$ contains $G^+$.  In particular, $G^+$ is trivial or abstractly simple.\end{thm}

\begin{thm}[\cite{MollerVonk} Theorem~2.4]\label{thm:MollerVonk}Let $T$ be a tree and let $G$ be a closed geometrically dense subgroup of $\Aut(T)$.  Then every nontrivial closed subgroup of $G$ normalised by $G^{++}$ contains $G^{++}$.  In particular, $G^{++}$ is trivial or topologically simple.\end{thm}

\begin{rem} \label{rem:Gpp_vs_Gp}
Note that in general $G^{++} \le G^+$; if $G$ is $\propP{}$-closed, then equality holds.  Since every proper subtree is contained in a half-tree, another way of expressing the condition that $G^{++}$ is nontrivial is the following: there exists $g \in G$, such that the convex hull of $\{v \in VT \mid gv \neq v\}$ is not the whole of $T$.

In the situations of Theorems~\ref{thm:Tits} and~\ref{thm:MollerVonk}, we note that as soon as $G^+$ or $G^{++}$ is nontrivial, then it is geometrically dense by Lemma~\ref{lem:general_type}.  (Since $G$ is geometrically dense, the tree has no leaves; the case that $T$ has fewer than three ends is ruled out by the existence of nontrivial arc stabilisers.)
\end{rem}

An additional complication that occurs when working with locally infinite trees is that the local action of a given group need not be closed.  We say a subgroup $G \le \Aut(T)$ has \defbold{closed local actions} if for every $v \in VT$, the permutation group induced by $G_v$ on $o\inv(v)$ is closed.  In practice, most of the groups we are interested in will have closed local actions; the next lemma justifies this assumption by providing some sufficient conditions.

\begin{lem}\label{lem:closed_local_actions}
Let $T$ be a tree and let $G \le \Aut(T)$.  If one or more of the following conditions holds, then $G$ has closed local actions:
\begin{enumerate}[(i)]
\item $G$ is $\propP{}$-closed;
\item $G_a$ is compact for every $a \in AT$;
\item There is some group $L \le \Aut(T)$, such that $L^+ \le G \le L$ and $L$ has closed local actions.
\end{enumerate}
\end{lem}

\begin{proof}
Fix $v \in VT$ and let $\theta: \Aut(T)_v \rightarrow \Sym(o\inv(v))$ be the action of $\Aut(T)_v$ on $o\inv(v)$.  We aim to show in each case that $\theta(G_v)$ is closed in $\Sym(o\inv(v))$.

For (i), we suppose that $G$ is $\propP{}$-closed.  Suppose there is $h \in \Sym(o\inv(v))$ and a net $(g_i)_{i \in I}$ of elements of $G_v$ such that $\theta(g_i) \rightarrow h$.  Then for all finite subsets $F$ of $o\inv(v)$, there is $i \in I$ such that for all $j > i$ and $a \in F$ we have $g_ja = ha$.  We build an automorphism $g$ of $T$ as follows: set $gv = v$ and if $a \in o\inv(v)$, set $ga = ha$ and $g\ol{a} = \ol{ha}$.  For each $a \in o\inv(v)$, we choose $j_a \in I$ large enough that $g_{j_a}a = ha$ (the choices of $j_a$ can be made independently of one another).  We then set $gr = g_{j_a}r$, for all vertices and arcs $r$ of $T_a$.  It is then easy to check that $g$ is an automorphism of $T$ such that for all $w \in VT$ and all finite subsets $F \subseteq o\inv(w)$, there is some $g_F \in G$ such that $g_F$ agrees with $g$ on $F$.  Hence
\[
g \in G^{\propP{1}} = G^{\propP{}} = G.
\]
In particular, we have obtained an element of $G_v$ that acts as $h$ on $o\inv(v)$; given the choice of $h$, we conclude that $\theta(G_v)$ is closed, so $G$ has closed local actions.

Case (ii) follows from Lemma~\ref{lem:StronglyClosedFromCompactArcs}, by considering the action of $G_v$ on the $G_v$-invariant subtree spanned by $v$ and its neighbours.

In case (iii), we see that $\theta(G_v)$ contains all the point stabilisers of $\theta(L_v)$, so $\theta(G_v)$ is open in $\theta(L_v)$.  Since $\theta(L_v)$ is closed in $\Sym(o\inv(v))$, we deduce that $\theta(G_v)$ is closed in $\Sym(o\inv(v))$.
\end{proof}

For some parts of this article we will be particularly interested in the case that the closed subgroup $G \le \Aut(T)$ is locally compact.  Since $\Aut(T)$ is totally disconnected, the same is true of any subgroup with the induced topology.  Here we recall some basic properties of totally disconnected, locally compact (\tdlc) groups that will be used without comment later.

\begin{thm}[{Van Dantzig, \cite[TG. 39]{vD}}]
Let $G$ be a \tdlc group.  Then $G$ has a base of neighbourhoods of the identity consisting of compact open subgroups.
\end{thm}

\begin{lem}\label{lem:lc_permutation}
Let $X$ be a set and let $G \le \Sym(X)$ be closed.  Then $G$ is compact if and only if it has only finite orbits; $G$ is locally compact if and only if there is a finite subset $Y$ of $X$ such that the pointwise fixator of $Y$ in $G$ has only finite orbits on $X$.
\end{lem}

\begin{proof}
Write $G_{(Y)}$ for the pointwise fixator of $Y$ in $G$.  By the definition of the permutation topology, the subgroups $G_{(Y)}$ for $Y \subseteq X$ finite form a base of open neighbourhoods of the identity in $G$.  If some $G_{(Y)}$ is compact, then clearly $G$ is locally compact.  Conversely if $G$ is locally compact, then $G$ is a \tdlc group, so by Van Dantzig's theorem, we have $G_{(Y)} \subseteq U$ for some compact open subgroup $U$ of $G$ and some finite subset $Y$ of $X$.  In that case $G_{(Y)}$ is an open subgroup of $U$, hence also closed, and so $G_{(Y)}$ is compact.

It now suffices to characterise when $G_{(Y)}$ is compact; without loss of generality we can assume $G$ fixes $Y$ pointwise and determine whether $G$ is compact.  If $G$ is compact, then since point stabilisers are open, it must have only finite orbits by the orbit-stabiliser theorem. 
Conversely if $G$ has only finite orbits, say $\{X_i\}_{i \in I}$ is its set of orbits, then $G$ is embedded as a closed subgroup of $\prod_{i \in I}\Sym(X_i)$, which is in turn embedded as a closed subgroup of $\Sym(X)$.  Note that as a subgroup of $\Sym(X)$, the group $\prod_{i \in I}\Sym(X_i)$ carries the product topology, with each of the finite groups $\Sym(X_i)$ carrying the discrete topology.  The group $\prod_{i \in I}\Sym(X_i)$ is thus compact by Tychonoff's theorem, so $G$ is also compact.
\end{proof}

Given the above two results we have a ``Cayley--Van Dantzig theorem'' for \tdlc groups: Every \tdlc group is isomorphic as a topological group to a closed permutation group (not necessarily transitive) in which every point stabiliser has finite orbits.  For instance, one can take the natural left translation action of $G$ on the set
\[
\bigsqcup \{G/U \mid U \text{ compact open subgroup of } G\}. 
\]
We will not need this fact, but it is useful to keep in mind with respect to the universality of permutational constructions in the theory of \tdlc groups.

A topological group is \defbold{compactly generated} if it can be generated as an abstract group by some compact subset.  Compactly generated \tdlc groups are related to general \tdlc groups as follows.

\begin{lem}\label{lem:tdlc_union}
Let $G$ be a \tdlc group.  Then $G$ is the directed union of a family of compactly generated open subgroups.
\end{lem}

\begin{proof}
Let $U$ be a compact open subgroup of $G$, and let $I$ be the set of all finite subsets of $G$.  For each $S \in I$ write $G_S = \grp{S \cup U}$.  Then it is clear that $G = \bigcup_{S \in I}G_S$; the groups $G_S$ form a directed family, since
\[
G_{S_1},\dots, G_{S_n} \le G_{S_1 \cup \dots \cup S_n}.
\]
Each of the groups $G_S$ is a union of cosets of $U$, so $G_S$ is an open subgroup of $G$.
\end{proof}

\subsection{Bass--Serre theory} \label{BassSerreTheory}

Here we recall some standard results in Bass--Serre theory for groups acting on trees.  In this article we will not be using Bass--Serre theory to construct the groups but we will use it occasionally to analyse them.  Note that conventional Bass--Serre theory considers only actions on trees without inversion, whereas we allow reversal of edges; we keep track of these edge reversals in the quotient graph by allowing a loop to be its own inverse.  This added generality has no deep significance, since an action with inversion can always be converted to an action without inversion by subdividing edges, but it necessitates some adjustments to the statements.

Given a group $G$ acting on a tree $T$, we define the \defbold{inversion-free subdivision} $T^i$ by subdividing in two parts those edges $a$ of $T$ such that $\ol{a} \in Ga$.  Analogously, in the quotient graph $\Gamma = G \backslash T$, we define the  \defbold{inversion-free subdivision} (or \defbold{orientable subdivision}) of $\Gamma^i$ of $\Gamma$ by taking each loop $a$ such that $a = \ol{a}$ (with $o(a) = t(a) = v$, say), adding a new vertex $v_a$, and replacing $a$ with the geometric edge $\{a',\ol{a'}\}$ where $o(a') = v$ and $t(a') = v_a$, so that $\Gamma^i$ is an orientable graph.  The action of $G$ on $T^i$ is then without inversion, and the quotient map from $T$ to $\Gamma$ naturally gives rise to a quotient map from $T^i$ to $\Gamma^i$.

\begin{lem}[{\cite[I.3.1, Proposition~14]{Serre:trees}}]\label{lem:subtree_lift}
Let $G$ be a group acting without inversion on a tree $T$.  Then every subtree of $G \backslash T$ lifts to a subtree of $T$.
\end{lem}

\begin{thm}[{\cite[I.5.4, Corollary 1 \& Exercise 2]{Serre:trees}}]\label{thm:BassSerreSimplicity}
Suppose that $T$ is a tree, and let $G$ be a group acting on $T$ without inversion. Let $R$ be the group generated by all vertex stabilisers $G_v$, $v \in VT$. Then $R$ is a normal subgroup of $G$, and $G/R$ is isomorphic to the fundamental group of $G \backslash T$. Moreover, $G=R$ if and only if $G \backslash T$ is a tree.
\end{thm}

\begin{thm}[{\cite[I.5.4, Theorem~13]{Serre:trees}}]\label{thm:BassSerre}(Bass--Serre structure theorem for groups acting on trees)
Let $G$ be a group acting on a tree $T$.  For the inversion-free subdivision $T^i$, let $\pi: T^i \rightarrow G \backslash T^i$ be the quotient map of $(T^i,G)$.  Choose a subtree $T'$ of $T^i$ that is a lift of a maximal subtree of $G \backslash T^i$.  Choose a subset $E^+ \subseteq AT^i$ such that $o(a) \in VT'$ for all $a \in E^+$, such that $\pi$ is injective on $E^+$ and $\pi(E)$ is an orientation of $G \backslash T^i$ and set $E = E^+ \cup \ol{E^+}$.  For each $a \in AT^i$ let $\tau_a$ be the inclusion of $G_a$ into $G_{t(a)}$.  For each $a \in E^+$ choose $s_a \in G$ so that $s\inv_at(a) \in VT'$, with $s_a=1$ if $a \in AT'$, and set $s_{\ol{a}} = s^{-1}_a$.  Write $F(E)$ for the free group over $\{s_a \mid a \in E\}$.  Then $G$ has the form
\[
\frac{F(E) \ast \Asterisk_{v \in VT'}G_v}{\langle \langle s_a\tau_a(g)s_{\ol{a}}\tau_{\overline{a}}(g)^{-1} \; (a \in E, g \in G_a), \; s_as_{\ol{a}} \; (a \in E), \; s_a \; (a \in AT') \rangle \rangle}.
\]
\end{thm}

\begin{defn}
Retain the hypotheses and notation of the previous theorem. Let $c$ be a directed path in $G \backslash T^i$ with vertex sequence $(v_0,v_1,\dots,v_n)$ and arc sequence $(a_1,a_2,\dots,a_n)$.  A \defbold{word of type $c$} is then a word $w = g_0s_{e_1}g_1 \dots g_{n-1}s_{e_n}g_n$ over $E \sqcup \bigsqcup_{v \in VT'}G_v$ such that $g_i \in G_{v'_i}$ where $\pi(v'_i) = v_i$ and $\pi(e_i) = a_i$.  Say that $w$ is \defbold{reduced} if it is of type $c$ for some directed path $c$ in $G \backslash T^i$, and satisfies the following conditions:

If $n=0$ then $g_0 \neq 1$; if $n \ge 1$, then for each index $i$ such that $a_{i+1} = \ol{a_i}$, then $g_i \not\in G_{e_i}$.
\end{defn}

\begin{thm}[{\cite[I.5.2, Theorem~11]{Serre:trees}}]\label{thm:NormalForm}(Normal form theorem of Bass--Serre theory)
Under the hypotheses of Theorem~\ref{thm:BassSerre}, every reduced word evaluates to a nontrivial element of $G$.
\end{thm}

The following corollary is valid without assuming that $G$ acts without inversion, since on the one hand, every group generated by vertex stabilisers acts without inversion (since it is parity-preserving) and on the other, if $\Gamma = G \backslash T$ is a tree then $\Gamma$ has no loops, so certainly $G$ acts on $T$ without inversion.

\begin{cor}[{\cite[I.5.4, Exercise~2]{Serre:trees}}]\label{cor:vertex_group}
Let $G$ be a group acting on a tree $T$.  Then $G \backslash T$ is a tree if and only if $G$ is generated by vertex stabilisers.  Moreover, if $G$ is generated by vertex stabilisers, then $G = \grp{G_v \mid v \in VT'}$ where $T'$ is a lift of a maximal subtree in $G \backslash T$.
\end{cor}

\section{A parametrisation of $\propP{}$-closed groups}\label{sec:correspondence}

This section contains the main ideas underpinning our theory of local action diagrams. We define local action diagrams (Definition~\ref{def:LADs}) and isomorphisms of local action diagrams (Definition~\ref{def:IsoOfLADs}). For a given local action diagram $\Delta$ we define a $\Delta$-tree (Definition~\ref{def:deltaTree}), construct it and show that it is unique up to isomorphism (Lemma~\ref{lem:tree_isomorphism}). Each $\Delta$-tree $\mathbf{T}$ consists of an underlying tree $T$ with some additional extra structure consisting of arc colours and a projection map.

For $G \leq \Aut(T)$ we define an associated local action diagram $\Delta$ (Definition~\ref{def:AssocLocalActionDiagram}) and observe that $T$ can be equipped with the additional structure of a $\Delta$-tree $\mathbf{T}$ so that the underlying tree for $\mathbf{T}$ is $T$. Thus for any $G \leq \Aut(T)$ we have an associated pair $(\Delta,\mathbf{T})$ where $\Delta$ is a local action diagram and $\mathbf{T}$ is a $\Delta$-tree. This associated pair is unique up to isomorphism (Lemma~\ref{lem:sameT}).

In a series of definitions and results (Definition~\ref{UniversalGpOfLocalActions} to Definition~\ref{def:THEUniversalGp}) we show the following (all statements are up to isomorphism).
\begin{enumerate}[(i)]
\item
	We define the universal group $\Univ(\Delta)$ of a local action diagram $\Delta$ as a group of automorphisms of its $\Delta$-tree $\mathbf{T}$. In particular we have $\Univ(\Delta) \leq \Aut(T)$ where $T$ is the underlying tree for $\mathbf{T}$.
\item
	We show that this universal group exists and is unique.
\item
	We show that the local action diagram associated to $\Univ(\Delta)$ acting on $T$ is $\Delta$.
\item
	We show that if $\Delta$ is the local action diagram associated to an action $G \leq \Aut(T)$ then $\Univ(\Delta)$ is the $\propP{}$-closure of $G$. 
\end{enumerate}
Together these observations form the basis of the correspondence theorem (Theorem~\ref{thm:correspondence}).

We conclude the section with two examples of local action diagrams (Examples~\ref{Ex:AutT} and \ref{Ex:BM}).\\

\begin{defn} \label{def:LADs}
A \defbold{local action diagram $\Delta = (\Gamma,(X_a),(G(v)))$} consists of the following information:
\begin{itemize}
\item A connected graph $\Gamma$.
\item For each arc $a$ of $\Gamma$, a nonempty set $X_a$ (called the \defbold{colour set} of $a$). We take the colour sets of distinct arcs to be disjoint from one another, and the elements of $\bigsqcup_{a \in A\Gamma}X_a$ are the \defbold{colours} of $\Delta$.
\item
	For each vertex $v$ of $\Gamma$, a group $G(v)$ (called the \defbold{local action} at $v$) with the following properties: write $X_v$ to denote the disjoint union $\bigsqcup_{a \in o\inv(v)}X_a$, then the group $G(v)$ is a closed subgroup of $\Sym(X_v)$ and the sets $X_a$ are the orbits of $G(v)$ on $X_v$.
\end{itemize}
\end{defn}
 
Examples~\ref{Ex:AutT} and \ref{Ex:BM} can be found starting on page~\pageref{Ex:AutT}.

\begin{defn} \label{def:IsoOfLADs}
Let $\Delta = (\Gamma,(X_a),(G(v)))$ and $\Delta' = (\Gamma',(X'_a),(G'(v)))$ be local action diagrams.

An \defbold{isomorphism} $\boldsymbol{\theta} = (\theta,(\theta_v))$ from $\Delta$ to $\Delta'$ is an isomorphism $\theta: \Gamma \rightarrow \Gamma'$ of graphs, together with a bijection $\theta_v: X_v \rightarrow X'_{\theta(v)}$ for each $v \in V\Gamma$ that restricts to a bijection from $X_a$ to $X_{\theta(a)}$ for each $a \in o\inv(v)$, and such that $\theta_vG(v)\theta\inv_v = G'(v')$.
\end{defn}

Local action diagrams have the advantage of having a simple description from a combinatorial perspective.  In terms of the permutation groups $G(v)$, there are no interactions between them or compatibility conditions to check, except that $G(v)$ should have the specified orbit structure.  However, we will see that they provide a parametrisation of all $\propP{}$-closed groups of tree automorphisms, taken up to isomorphisms of the tree.  Our aim in this section is to prove the following:

\begin{thm}\label{thm:correspondence}
There is a natural one-to-one correspondence between isomorphism classes of local action diagrams, and isomorphism classes of pairs $(T,G)$ where $T$ is a tree and $G$ is a $\propP{}$-closed subgroup of $\Aut(T)$.
\end{thm}

\begin{defn} \label{def:deltaTree}
Given a local action diagram $\Delta$, a \defbold{$\Delta$-tree} $\mathbf{T}$ is a tree $T$ together with a surjective graph homomorphism $\pi: T \rightarrow \Gamma$ and a \defbold{$\Delta$-colouring}, that is, a map $\mc{L}: AT \rightarrow \bigsqcup_{a \in A\Gamma}X_a$, such that for each vertex $v \in VT$, and each arc $a$ in $o\inv(\pi(v))$, the map $\mc{L}$ restricts to a bijection $\mc{L}_{v,a}$ from $\{b \in o\inv(v) \mid \pi(b) = a\}$ to $X_a$.  Write $\mathbf{T} = (T, \mc{L}, \pi)$, and given $v \in VT$, write $\mc{L}_v$ for the restriction of $\mc{L}$ to a bijection from $o\inv(v)$ to $X_{\pi(v)}$.

Note that the groups $G(v)$ play no role in the definition of $\mathbf{T}$. 
\end{defn}

\begin{lem}\label{lem:tree_isomorphism}
Let $\Delta$ be a local action diagram.  Then there exists a $\Delta$-tree.  Moreover, given any two $\Delta$-trees $(T,\pi,\mc{L})$ and $(T',\pi',\mc{L}')$, there is a graph isomorphism $\alpha: T \rightarrow T'$ such that $\pi' \circ \alpha = \pi$.
\end{lem}

\begin{proof}
Choose a base vertex $v_0 \in V\Gamma$. We construct a $\Delta$-tree $\mathbf{T}$ as follows.

Given $v \in V\Gamma$ and $c \in X_v$, the \defbold{type} $p(c)$ of $c$ is the unique $a \in A\Gamma$ such that $c \in X_a$.  A \defbold{coloured path (of length $n$)} in $\Gamma$ is a finite sequence $(c_1,c_2,\dots,c_n)$ of colours, where for each $1 \le i < n$, we have $o(p(c_{i+1})) = t(p(c_i))$.  The \defbold{origin} of the coloured path is $o(p(c_1))$.

Vertices $v \in VT$ will be labelled by coloured paths with origin $v_0$.  For vertex labels $v = (c_1,c_2,\dots,c_n)$ and $w = (c_1, c_2, \dots, c_n, c_{n+1}, \dots, c_m)$ we say that $v$ is a \defbold{prefix} of $w$.
For each vertex label $v = (c_1,c_2,\dots,c_n)$, there will also be a reverse label $\overline{v} = (d_1,d_2,\dots,d_n)$ of the same length, where $d_i$ is a colour such that $p(d_i) = \overline{p(c_i)}$, and such that if $v$ is a prefix of $w$, then $\overline{v}$ is the corresponding prefix of $\overline{w}$.  We produce the vertices of $VT$ inductively starting at a root vertex $()$.

Suppose we have defined a vertex $v = (c_1,c_2,\dots,c_n)$ with reverse label $\overline{v} = (d_1,d_2,\dots,d_n)$.  Then we define vertices $v_{+c_{n+1}} = (c_1,\dots,c_n,c_{n+1})$, for all $c_{n+1}$ such that $o(p(c_{n+1})) = t(p(c_n))$ and $c_{n+1} \neq d_n$.  We then set $\overline{v_{+c_{n+1}}} =  (d_1,d_2,\dots,d_{n+1})$, where $d_{n+1}$ is some element of $X_{\overline{p(c_{n+1})}}$ (chosen arbitrarily).

The set $AT_+$ of forward arcs of $T$ consists of ordered pairs $(v,w)$, where $v$ is a prefix of $w$ of length one less than $w$; then $AT_- =: \{(w,v) \mid (v,w) \in AT_+\}$ and $AT := AT_- \sqcup AT_+$.  Origin and terminal vertices and edge reversal are defined in the obvious way, and it is clear that we obtain a tree.  The colouring $\mc{L}$ is defined as follows: given $(v,w) \in AT_+$, then $\mc{L}(v,w)$ is the last entry of $w$ and $\mc{L}(w,v)$ is the last entry of $\overline{w}$.

The graph homomorphism $\pi: T \rightarrow \Gamma$ is given by $\pi(()) = v_0$ for the base vertex; $\pi(v) = t(p(c_n))$ for any vertex $v = (c_1,\dots,c_n)$ in $VT$; and $\pi(a) = p(\mc{L}(a))$ for $a \in AT$.  Given the way in which the entries $c_i$ and $d_i$ were chosen and used to define $\mc{L}$, one sees that $\pi$ is a surjective graph homomorphism.

Given a vertex $v = (c_1,\dots,c_n)$ with reverse label $\overline{v} = (d_1,d_2,\dots,d_n)$, then $v$ has one parent vertex $(c_1,\dots,c_{n-1})$ and a set of child vertices of the form $(c_1,\dots,c_n,c')$, where $c'$ ranges over the set $X_{t(p(c_n))} \smallsetminus  \{d_n\} = X_{\pi(v)} \smallsetminus \{d_n\}$.  The set $o\inv(v)$ is thus in a natural bijection with $X_{\pi(v)}$ in a manner that respects the partition into sets $X_a$ for $a \in o\inv(\pi(v))$, and the colouring produces the same bijection.  In particular, for each arc $a \in o\inv(\pi(v))$, we see that  $\mc{L}$ restricts to a bijection from $\{b \in o\inv(v) \mid \pi(b) = a\}$ to $X_a$.  Thus the object $\mathbf{T}$ we have constructed is a $\Delta$-tree.

Now suppose that we have two $\Delta$-trees $(T,\pi,\mc{L})$ and $(T',\pi',\mc{L}')$.  We construct a graph isomorphism $\alpha: T \rightarrow T'$ compatible with $(\pi,\pi')$ inductively as follows.

Choose $v_1 \in VT$ and $w_1 \in VT'$ such that $\pi(v_1) = \pi'(w_1) = v_0$, and set $\alpha(v_1) = w_1$.  Suppose we have defined $\alpha$ for vertices and arcs in $B_n(v_1)$ ($n \ge 0$), let $v$ be a vertex in $T$ at distance $n$ from $v_1$ and let $w = \alpha(v)$.  Then $\pi(v) = \pi'(w)$ by the induction hypothesis.  Given $a \in o\inv(\pi(v))$, we have bijections $\mc{L}_{v,a}: \{b \in o\inv(v) \mid \pi(b) = a\} \rightarrow X_a$ and $\mc{L}'_{w,a} := \{b \in o\inv(w) \mid \pi'(b) = a\} \rightarrow X_a$.  In particular, the sets $\{b \in o\inv(v) \mid \pi(b) = a\}$ and $\{b \in o\inv(w) \mid \pi'(b) = a\}$ have the same size, so we can extend $\alpha$ to include $o\inv(v)$ in its domain, in such a way that it restricts to a bijection from $\{b \in o\inv(v) \mid \pi(b) = a\}$ to $\{b \in o\inv(w) \mid \pi'(b) = a\}$.  The choice of bijection is unimportant here, except in the case that $\pi\inv(a)$ contains an arc $a'$ starting at $v$ in the direction of $v_1$: in this case, $\alpha(a')$ has already been chosen, so we choose a bijection from $\{b \in o\inv(v) \mid \pi(b) = a, b \neq a'\}$ to $\{b \in o\inv(w) \mid \pi'(b) = a, b \neq \alpha(a')\}$.  For $b \in t\inv(v)$ and $v' = o(b)$, we set $\alpha(b) = \overline{\alpha(\overline{b})}$ and $\alpha(v') = o(\alpha(b))$ respectively.  This extends the definition of $\alpha$ to a ball of radius $n+1$ about $v_1$; notice that $\alpha$ still produces a graph isomorphism from $B_{n+1}(v_1)$ to $B_{n+1}(w_1)$, completing the inductive step.  We can thus extend $\alpha$ to a graph isomorphism from $T$ to $T'$ such that $\pi' \circ \alpha = \pi$.
\end{proof}

Note: we do not claim that $\alpha$ can be chosen to map $\mc{L}$ to $\mc{L}'$.

\begin{defn} \label{def:AssocLocalActionDiagram}
Let $G$ be a group of automorphisms of a tree $T$.  We define an \defbold{associated local action diagram} $\Delta$ and equip $T$ with the structure of a $\Delta$-tree as follows.
\begin{itemize}
\item
	$\Gamma$ is the quotient graph $G \backslash T$, and $\pi$ is the natural quotient map.
\item
	For each $v \in V\Gamma$, choose a vertex $v^* \in \pi\inv(v)$; write $V^*$ for the set of vertices so obtained.  Given $a \in A\Gamma$ such that $v = o(a)$, let $X_a = \{b \in o\inv(v^*) \mid \pi(b) = a\}$.  The set $X_v := o\inv(v^*)$ is then naturally partitioned as required.  Define the group $G(v)$ to be the closure of the permutation group induced on $X_v$ by $G_{v^*}$.
\item
	For each $w \in VT$, choose $g_w \in G$ such that $g_ww \in V^*$.  Then $g_w$ also induces a bijection from $o\inv(w)$ to $X_v$.  Given $b \in o\inv(w)$, set $\mc{L}(b) = g_wb$.
\end{itemize}
\end{defn}

The definition is such that given $v,w \in VT$ such that $\pi(v) = \pi(w)$, the restrictions $\mc{L}_v$ and $\mc{L}_w$ of $\mc{L}$ to $o\inv(v)$ and $o\inv(w)$ respectively are bijections that form two sides of a commuting triangle: 
if $r_{v,w}$ is the map from $o\inv(v)$ to $o\inv(w)$ induced by $g\inv_w g_v$, then
\[
\mc{L}_v = \mc{L}_w r_{v,w}.
\]

There are many choices for the associated local action diagram, but they are all isomorphic, as we see in the following lemma.

\begin{lem}\label{lem:sameT}
Let $T$ be a tree and let $G$ be a group of automorphisms of $T$.  Then any two local action diagrams $\Delta = (\Gamma,(X_a),(G(v)))$ and $\Delta' = (\Gamma,(X'_a),(G'(v)))$ associated to $G$ are isomorphic, via an isomorphism $\boldsymbol{\theta}$ that is the identity map on the graph $\Gamma$.
\end{lem}

\begin{proof}
Without loss of generality we can assume $G$ is closed in $\Aut(T)$.  From the definition, we see that $\Delta$ and $\Delta'$ have the same associated graph $\Gamma = G \backslash T$; let $\theta$ be the trivial graph automorphism of $\Gamma$.  Given $v \in V\Gamma$, say the chosen element of $\pi\inv(v)$ is $v^*$ in the construction of $\Delta$, and $v^{**}$ in the construction of $\Delta'$.  Then $v^{**} = g_vv^*$ for some $g_{v} \in G$, since $\pi(v^*) = Gv^*$.  We can thus define a bijection $\theta_v$ from $X_v:= o\inv(v^*)$ to $X'_v := o\inv(v^{**})$ by setting $\theta_v(a) = g_va$.  Given that $g_vG_{v^*}g\inv_v = G_{v^{**}}$, and $G(v)$ and $G'(v)$ are determined by the actions of the vertex stabilisers $G_{v^*}$ and $G_{v^{**}}$ respectively, we see that $\theta_vG(v)\theta\inv_v = G'(v)$.  In particular, $\theta_v$ sends orbits of $G(v)$ to orbits of $G'(v)$, so it restricts to a bijection from $X_a$ to $X'_a$ for each $a \in o\inv(v)$.  Thus $(\theta,\theta_v)$ is an isomorphism of local action diagrams.
\end{proof}

Thus from now on, we can talk about {\it the} local action diagram $\Delta(T,G)$ associated to $(T,G)$ without ambiguity.

\begin{defn} \label{UniversalGpOfLocalActions}
An \defbold{automorphism} of the $\Delta$-tree $\mathbf{T}$ is a graph automorphism $\theta$ of $T$ such that $\pi \circ \theta = \pi$.  Write $\Aut_{\pi}(T)$ for the group of all such automorphisms.  Given $g \in \Aut_{\pi}(T)$, a vertex $v \in VT$, and $\mc{L}$ the colouring associated to $\mathbf{T}$, we define the \defbold{$\mc{L}$-local action} of $g$ at $v$ as follows:
\[
\sigma_{\mc{L},v}(g): X_{\pi(v)} \rightarrow X_{\pi(v)} \quad \quad\sigma_{\mc{L},v}(g)(c) := \mc{L}g\mc{L}|^{-1}_{o\inv(v)}(c).
\]
We see that $\sigma_{\mc{L},v}(g)$ is a permutation of $X_{\pi(v)}$, and for a vertex $v \in VT$ the map $g \mapsto \sigma_{\mc{L},v}(g)$ is continuous.
Finally, we define the \defbold{universal group of $\mathbf{T}$ with respect to local actions $(G(v))_{v \in V\Gamma}$} to be the set $\Univ(\mathbf{T},(G(v)))$ of all elements $g$ of $\Aut_{\pi}(T)$ such that for every $v \in VT$, the permutation $\sigma_{\mc{L},v}(g)$ belongs to $G(\pi(v))$.
\end{defn}

\begin{thm}\label{thm:diag_to_U}
Let $\Delta$ be a local action diagram, let $\mathbf{T}$ be a $\Delta$-tree, and let $H = \Univ(\mathbf{T},(G(v)))$. Then $H$ is a $\propP{}$-closed subgroup of $\Aut(T)$; $\Delta$ is isomorphic to a local action diagram associated to $H$; and for every vertex $v \in VT$ and $g \in G(\pi(v))$, there is $h \in H_v$ such that $\sigma_{\mc{L},v}(h) = g$.
\end{thm}

\begin{proof}
Let $g,h \in \Univ(\mathbf{T},(G(v)))$ and let $v \in VT$.  It is clear that $g\inv, gh \in \Aut_{\pi}(T)$, so there is a fixed vertex $w \in V\Gamma$ such that $w = \pi(v) = \pi(hv) = \pi(gv)=\pi(g^{-1}v)$.  It is easily seen that $\sigma_{\mc{L},v}(g\inv)$ and $\sigma_{\mc{L},v}(gh)$ are given by the following formulae:
\[
\sigma_{\mc{L},v}(gh) = \sigma_{\mc{L},hv}(g)\sigma_{\mc{L},v}(h)
\]
\[
\sigma_{\mc{L},v}(g\inv) = \left ( \sigma_{\mc{L},g^{-1} v}(g) \right )\inv
\]
Since $\sigma_{\mc{L},hv}(g)$, $\sigma_{\mc{L},v}(h)$ and 
$\sigma_{\mc{L},g^{-1}v}(g)$
are all in the group $G(w)$, we see that $\sigma_{\mc{L},v}(gh)$ and $\sigma_{\mc{L},v}(g\inv)$ are also elements of $G(w)$.  This proves $\Univ(\mathbf{T},(G(v)))$ is closed under products and inverses.  Since $\Univ(\mathbf{T},(G(v)))$ clearly also contains the trivial automorphism of $T$, we conclude that $H:= \Univ(\mathbf{T},(G(v)))$ is a subgroup of $\Aut_{\pi}(T)$.

Since $H \le \Aut_{\pi}(T)$, certainly every orbit of $H$ is contained in a fibre of $\pi$.  We claim that  if $e$ is any vertex or arc of $\Gamma$ then $H$ acts transitively on $\pi\inv(e)$. It is enough to show that $H$ is transitive when $e$ is an arc of $\Gamma$, as the vertex case will then follow by considering origin vertices of arcs.  So fix arcs $a,b \in \pi\inv(e)$; we aim to construct $g \in H$ such that $ga = b$.  We define $g$ in stages on balls of radius $n$ about $v_0:=o(a)$.

Let $v'_0 = o(b)$ and let $w = \pi(v_0)$.  Choose an element $h_0 \in G(w)$ such that $h_0\mc{L}(a) = \mc{L}(b)$; this is possible since by definition, $\mc{L}(a)$ and $\mc{L}(b)$ must lie in the same $G(w)$-orbit.  Then there is a unique graph isomorphism $g_1$ from $B_1(v_0)$ to $B_1(v'_0)$ such that $\mc{L}g_1\mc{L}|_{o\inv(v_0)}(c) = h_0(c)$ for all $c \in X_w$.

Let us also pause to note that by varying $b$, we can obtain every element of $G(w)$ as a suitable $h_0$ whilst also fixing $v_0$: specifically, given $h \in G(w)$, then $h\mc{L}(a) = \mc{L}(b)$ for some unique $b \in o\inv(v_0)$, and hence in this case we can take $h_0 = h$.  Thus, provided we can extend $g_1$ to an element of $H$, we will have shown that $H_{v_0}$ achieves all possible values of $\sigma_{\mc{L},v_0}$ at the vertex $v_0$.  By varying $a$, the vertex $v_0$ can also be made an arbitrary vertex of $T$.

Suppose we have defined a graph isomorphism $g_n$ from $B_n(v_0)$ to $B_n(v'_0)$, such that 
\[
\mc{L}g_n\mc{L}|_{o\inv(v)} \in G(\pi(v)) \text{ for all }v \in VB_{n-1}(v_0).
\]
Let $S_n(v_0)$ be the sphere of radius $n$ about $v_0$ and let $v \in S_n(v_0)$.  We have already defined $g_nv$ and also $g_nr$ for the unique arc $r \in o\inv(v)$ in the direction of $v_0$.  Similar to before, we see that $\mc{L}(r)$ and $\mc{L}(g_nr)$ lie in the same $G(\pi(v))$-orbit, so there is $h_n \in G(\pi(v))$ such that $h_n\mc{L}(r) = \mc{L}(g_nr)$.  There is then a unique graph isomorphism $h'(v)$ from $B_1(v)$ to $B_1(g_nv)$ such that $\mc{L}h'\mc{L}|_{o\inv(v)}(c) = h_n(c)$ for all $c \in X_{\pi(v)}$.  Note the domains of the maps $\{h'(v) \mid v \in S_n(v_0)\}$ are pairwise disjoint; for each $v \in S_n(v_0)$, the domains of $g_n$ and $h'(v)$ overlap only on a single edge and its endpoints, and for this overlap, $g_n$ and $h'(v)$ agree.  We can thus combine $g_n$ with the set of maps $\{h'(v) \mid v \in S_n(v_0)\}$ to produce a graph isomorphism $g_{n+1}$ from $B_{n+1}(v_0)$ to $B_{n+1}(v'_0)$.  By construction, we see that $\mc{L}g_{n+1}\mc{L}|_{o\inv(v)} \in G(\pi(v))$  for all $v \in VB_{n}(v_0)$, completing the inductive step.

By combining the sequence $(g_n)$ of graph isomorphisms, we thus obtain $g \in \Aut_{\pi}(T)$ such that $ga = b$ and such that $\sigma_{\mc{L},v}(g) \in G(\pi(v))$ for all $v \in VT$, so $g \in H$.

Now let $\Delta' = (\Gamma',(X'_a),(H(v)))$ be a local action diagram associated to $H$.  We aim to construct an isomorphism $\boldsymbol{\theta} = (\theta,(\theta_v))$ of local action diagrams from $\Delta'$ to $\Delta$.  We have shown that $H$ acts transitively on each fibre of $\pi$, so the quotient graph $\Gamma' = H \backslash T$ can be naturally identified with $\Gamma$ and $\theta$ can be taken to be the trivial graph isomorphism on $\Gamma$.  Given $v \in V\Gamma$, let $v^*$ be the chosen vertex of $VT$ in the construction of $\Delta'$.  Then for each $a \in A\Gamma$ such that $v = o(a)$, by definition $X'_a = \{b \in o\inv(v^*) \mid \pi(b) = a\}$.  The definition of $\mathbf{T}$ then provides a bijection $\mc{L}_{v^*,a}$ from $X'_a$ to $X_a$.  We can thus construct a bijection $\theta_v$ from $X'_v$ to $X_v$ by setting $\theta_v(c) = \mc{L}_{v^*,a}(c)$.  By definition, $H(v)$ is just the closure of the group of permutations induced by $H_{v^*}$ on $o\inv(v^*)$.  As previously observed, we obtain in this way the set of all permutations $h$ of $o\inv(v^*)$ such that the permutation induced by $\mc{L}h\mc{L}|_{o\inv(v^*)}$ on $X_v$ is an element of $G(v)$.  Given $c \in X_v$, we see from the definition of $\theta_v$ that
\[
\mc{L}h\mc{L}|^{-1}_{o\inv(v^*)}(c) = \theta_vh\theta\inv_v(c).
\]
Thus $\theta_vH(v)\theta\inv_v = G(v)$, completing the proof that $(\theta,(\theta_v))$ is an isomorphism of local action diagrams from $\Delta'$ to $\Delta$.

Finally, let $g$ be an element of the $\propP{}$-closure of $H$.  Then for every finite subset $Y$ of $o\inv(v)$, there is some $h \in H$ such that $gv = hv$ and $gy=hy$ for all $y \in Y$.  In other words, $h$ is such that $\sigma_{\mc{L},v}(h)(c) = \sigma_{\mc{L},v}(g)(c)$ for all $c \in \mc{L}(Y)$; note that by definition, $\sigma_{\mc{L},v}(h) \in G(\pi(v))$.  Since $Y$ can be any finite subset of $o\inv(v)$ and $G(\pi(v))$ is closed, it follows that $\sigma_{\mc{L},v}(g) \in G(\pi(v))$ for all $v \in VT$, and hence $g \in H$.  Thus $H$ is $\propP{}$-closed.
\end{proof}

\begin{thm}\label{thm:G_to_U}
Let $T$ be a tree, let $G \le \Aut(T)$, let $\Delta = \Delta(T,G)$ and let $\mathbf{T}$ be an associated $\Delta$-tree structure on $T$.  Then $\Univ(\mathbf{T},(G(v)))$ is the $\propP{}$-closure of $G$.
\end{thm}

\begin{proof}
By Theorem~\ref{thm:diag_to_U}, the set $H = \Univ(\mathbf{T},(G(v)))$ is a $\propP{}$-closed subgroup of $\Aut(T)$.  Let $V^*$ be the set of vertices of $T$ used to define the $\Delta$-tree structure, and let $v \in VT$.  Then there is some $w \in V^*$ such that $\pi(v) = \pi(w)$.  In particular, $v$ and $w$ lie in the same $G$-orbit, and given the way in which $\mathbf{T}$ is constructed, there is some $g_v \in G$ such that $g_vv = w$ and the colouring for $b \in o\inv(v)$ is given by $\mc{L}(b) = g_vb$.

Let $g \in G$.  Then $\pi(v) = \pi(gv) = \pi(w)$, so $g_{gv}(gv) = w$ where $g_{gv}$ is as in the definition of the colouring.  Let $b \in o\inv(v)$ and let $c = \mc{L}(b)$.  Then $c = g_vb$ and $\mc{L}(gb) = g_{gv}gb = (g_{gv}gg\inv_v)c$, so
\[
\sigma_{\mc{L},v}(g)(c) = (g_{gv}gg\inv_v)c.
\]
We see that $(g_{gv}gg\inv_v)w = w$, so $g_{gv}gg\inv_v \in G_w$.  Thus by definition, the permutation induced by $g_{gv}gg\inv_v$ on $w$ is an element of $G(\pi(w))=G(\pi(v))$, that is, $\sigma_{\mc{L},v}(g) \in G(\pi(v))$.  Since this holds for all $v \in VT$, we see that $g \in H$.  Thus $G \le H$; since $H$ is $\propP{}$-closed, in fact $G^{\propP{}} \le H$.

Conversely, let $h \in H$.  Then by the definition of $H$,
\[
\sigma_{\mc{L},v}(h) \in G(\pi(v)).
\]
Given the definition of $G(\pi(v))$, for every finite subset $X$ of $o\inv(w)$ there is $g'_X \in G_w$ such that for all $c \in X$,
\[
g'_Xc = \sigma_{\mc{L},v}(h)(c).
\]
Since $h \in \Aut_{\pi}(T)$ we have $\pi(hv) =\pi(v)$, so we can write $g_{hv}(hv) = w$.  
Given a finite subset $X$ of $o\inv(v)$ and $b \in X$, we have $\mc{L}(hb) = \sigma_{\mc{L},v}(h)(g_vb)$ and also $\mc{L}(hb) = g_{hv}hb$.  Thus
\[
hb = g\inv_{hv}\sigma_{\mc{L},v}(h)(g_vb) = g\inv_{hv}g'_{g_vX}g_vb.
\] 
Thus on the set $X$, we see that $h$ agrees with the element $g\inv_{hv}g'_{g_vX}g_v$ of $G$.  Since this can be achieved at every vertex $v \in VT$ and for every finite subset $X$ of $o\inv(v)$, we conclude that $h \in G^{\propP{}}$.  This proves that $G^{\propP{}} = H$ as required.
\end{proof}

Since the local action diagram can be recovered from the group $\Univ(\mathbf{T},(G(v)))$, we have the following corollary.

\begin{cor}
Let $T$ be a tree, let $G \le \Aut(T)$ and let $\Delta$ be an associated local action diagram.  Then $\Delta$ is also an associated local action diagram for the action of $G^{\propP{}}$ on $T$.
\end{cor}

We have now shown that the $\propP{}$-closed subgroups of $\Aut(T)$ are exactly the groups realisable as a group $\Univ(\mathbf{T},(G(v)))$ definable from some local action diagram.  It remains to show that each local action diagram $\Delta$ gives rise to only one group $\Univ(\mathbf{T},(G(v)))$ up to tree isomorphisms, in other words, the choices made in defining the $\Delta$-tree $\mathbf{T}$ are not significant.

\begin{thm}\label{thm:unique_U}
Let $\Delta = (\Gamma,(X_a),(G(v)))$ be a local action diagram and let $\mathbf{T} = (T,\pi,\mc{L})$ and $\mathbf{T'}=(T',\pi',\mc{L}')$ be $\Delta$-trees.  Then there is a graph isomorphism $\phi: T \rightarrow T'$ such that $\phi \Univ(\mathbf{T},(G(v))) \phi\inv = \Univ(\mathbf{T}',(G(v)))$.
\end{thm}

\begin{proof}
Let $G = \Univ(\mathbf{T},(G(v)))$ and $H = \Univ(\mathbf{T}',(G(v)))$.  By Lemma~\ref{lem:tree_isomorphism}, there is a graph isomorphism $\alpha: T \rightarrow T'$ such that $\pi' \circ \alpha = \pi$.  Thus by applying $\alpha$ to $\mathbf{T}$ and replacing $G$ with $\alpha G \alpha\inv$, we may assume that $T = T'$ and $\pi = \pi'$, in other words, $G$ and $H$ act on the same tree with the same orbits, in such a way that the quotient graph can be naturally identified with $\Gamma$.

By Theorem~\ref{thm:diag_to_U}, for every pair of vertices $v,w \in VT$ such that $\pi(v) = \pi(w)$, we have
\[
\sigma_{\mc{L},v}(G_v) = G(\pi(v)) = \sigma_{\mc{L}',w}(H_w).
\]
Let $G^*_v$ and $H^*_w$ be the permutation groups induced by $G_v$ on $o\inv(v)$ and $H_w$ on $o\inv(w)$ respectively.  We see that $G^*_v$ and $H^*_w$ are both isomorphic to $G(\pi(v))$ as permutation groups; moreover, since $\pi = \pi'$, given any element $g \in G$ such that $gv = w$, then the groups $gG^*_vg\inv$ and $H^*_w$ also have the same orbits on $o\inv(w)$, with the same correspondence between orbits on $o\inv(w)$ and elements of $o\inv(\pi(v))$.  There is thus a bijection $\rho_{v,w}$ from $o\inv(v)$ to $o\inv(w)$ such that $\pi(\rho_{v,w}(a)) = \pi(a)$ for all $a \in o\inv(v)$ and such that $\rho_{v,w}G^*_v\rho\inv_{v,w} = H^*_w$.  Since $G^*_v$ is transitive on each $\pi$-fibre in $o\inv(v)$, for a single given $a \in o\inv(v)$, we are free to choose $\rho_{v,w}(a)$ from the set $\{b \in o\inv(w) \mid \pi(b) = \pi(a)\}$.

We now aim to construct $\phi \in \Aut_{\pi}(T)$ such that
\[
\phi \Univ(\mathbf{T},(G(v))) \phi\inv = \Univ(\mathbf{T}',(G(v))).
\]
We construct $\phi$ successively on balls of radius $n$ centred on some vertex $v_0 \in VT$, starting with $\phi(v_0) = v_0$, such that at each stage $\phi$ is a graph automorphism on $B_n(v_0)$ that commutes with $\pi$.  To define $\phi$ on $B_1(v_0)$, we let it act as $\rho_{v_0,v_0}$ on $o\inv(v_0)$, and then extend to the remaining arcs and vertices in $B_1(v_0)$ via the equations $\phi(\overline{a}) = \overline{\phi(a)}$ and $\phi(t(a)) = t(\phi(a))$.  Now suppose we have defined $\phi$ on $B_n(v_0)$ for some $n \ge 1$, and let $v \in S_n(v_0)$; write $\phi_n$ for the automorphism of $B_n(v_0)$.  Let $a$ be the unique arc in $o\inv(v)$ in the direction of $v_0$.  Then we have already specified $\phi(a)$, and it has been chosen in such a way that $\pi(\overline{\phi(a)}) = \pi(\overline{a})$; hence it is also the case that $\pi(\phi(a)) = \pi(a)$.  We can thus choose $\rho_{v,\phi(v)}$ so that $\rho_{v,\phi(v)}(a) = \phi(a)$.  There is then a unique isomorphism $\phi_v$ from $B_1(v)$ to $B_1(\phi(v))$ that is compatible with both $\phi_n$ and with $\rho_{v,v_0}$.  We then define $\phi(e)$ for $e$ a vertex or arc of $B_{n+1}(v_0)$ to be $\phi_n(e)$ or $\phi_v(e)$ as applicable, and observe that we have produced a graph automorphism of $B_{n+1}(v_0)$ that commutes with $\phi$.  By induction, we produce $\phi \in \Aut_{\pi}(T)$.

Now let $v \in VT$, write $w =\phi(v)$ and consider $\phi G_v \phi\inv$.  The construction of $\phi$ was such that the permutation group induced by $\phi G_v \phi\inv$ on $o\inv(w)$ is $\rho_{v,w}G_v\rho\inv_{v,w}$, or in other words, it is the same permutation group as the one induced by $H_w$ on $o\inv(w)$. Thus $\sigma_{\mc{L}',w}(\phi G_v \phi\inv) = G(\pi(w))$. By varying $v$ so that $w$ ranges over $VT$, we see (from the definition of $H = \Univ(\mathbf{T}',(G(v)))$) that $\phi G \phi\inv \le H$.  On the other hand, we see by a similar argument that $\phi\inv H \phi \le G$, so $\phi G \phi \inv  \ge H$.  So in fact $\phi G \phi\inv = H$, as required.
\end{proof}

\begin{defn} \label{def:THEUniversalGp}
We can thus define {\it the} universal group of a local action diagram: $\Univ(\Delta)$ is defined as $\Univ(\mathbf{T},(G(v)))$ where $\mathbf{T}$ is some $\Delta$-tree.  Then $\Univ(\Delta)$ is defined up to isomorphisms of the tree on which it acts; we write $\Univ_{\mathbf{T}}(\Delta)$ if we want to impose a specific action on a specific tree.
\end{defn}

To conclude the section, we now prove the correspondence theorem.

\begin{proof}[Proof of Theorem~\ref{thm:correspondence}]
Given a local action diagram $\Delta$, we have an associated pair $(T,\Univ(\Delta))$, and by Theorem~\ref{thm:unique_U} the pair $(T,\Univ(\Delta))$ is specified uniquely up to isomorphisms; on the other hand, it is clear from the construction that if $\Delta$ and $\Delta'$ are isomorphic as local action diagrams, then they will produce isomorphic pairs $(T,\Univ(\Delta))$ and $(T',\Univ(\Delta'))$.  
Thus we have a well-defined mapping $\beta$ from isomorphism classes of local action diagram to isomorphism classes of pairs $(T,G)$ where $T$ is a tree and $G \le \Aut(T)$ is $\propP{}$-closed on $T$. 
Theorem~\ref{thm:G_to_U} shows that $\beta$ is surjective and Theorem~\ref{thm:diag_to_U} shows that $\beta$ is injective.  Thus we have a natural one-to-one correspondence as claimed.
\end{proof}

\begin{ex} \label{Ex:AutT}
Let $T$ be the $(2,4)$-biregular tree; that is the tree in which all vertices in one part of the natural bipartition have valency $2$, and vertices in the other part have valency $4$. Suppose $G = \Aut(T)$. Then $G$ is $\propP{}$-closed, has two orbits on $VT$, two orbits on $AT$ and its two local actions are the symmetric groups $\Sym(2)$ and $\Sym(4)$. In Figure~\ref{Fig:AutTAssocLAD} on page \pageref{Fig:AutTAssocLAD} we see $T$ (top) with the associated local action diagram $\Delta$ (bottom). The two orbit representatives  $v^*, w^* \in VT$  used in the construction of the associated local action diagram are indicated; so too are labels $\{1,2\}$ for $o\inv(v^*)$ and $\{3,4,5,6\}$ for  $o\inv(w^*)$.  In this example we have $G = \Univ(\Delta)$. 

\begin{figure}
\begin{center}
\scalebox{0.7}{ 
\begin{tikzpicture}
%
\begin{scope}
    \foreach \x in {-5,-3,...,5} {
        \node[Vertex] (VCen\x) at (\x,5) {};
    };
    \foreach \x in {-4,-2,...,4} {
        \node[OpenVertex] (VCen\x) at (\x,5) {};
        \foreach \y in {6} {
            \node[Vertex] (VMA\x) at (\x,\y) {};
            \node[OpenVertex] (VMU\x) at (\x,\y+0.5) {};
            \node[Vertex] (VTop\x) at (\x,\y+0.5+0.5) {};
            \node[Vertex] (VMUL\x) at (\x-0.5,\y+0.5) {};
            \node[Vertex] (VMUR\x) at (\x+0.5,\y+0.5) {};
        };
        \foreach \y in {4} {
            \node[Vertex] (VMB\x) at (\x,\y) {};
            \node[OpenVertex] (VML\x) at (\x,\y-0.5) {};
            \node[Vertex] (VBot\x) at (\x,\y-0.5-0.5) {};
            \node[Vertex] (VMLL\x) at (\x-0.5,\y-0.5) {};
            \node[Vertex] (VMLR\x) at (\x+0.5,\y-0.5) {};
        };
    };
    \foreach \x in {-4,0,4} {
        \draw [color=black, line width=0.8] (VMA\x) to (VTop\x);
        \draw [color=black, line width=0.8 ] (VMUL\x) to (VMUR\x);
        \node[OpenVertex] at (VMU\x) {};

        \draw [color=black, line width=0.8] (VMB\x) to (VBot\x);
        \draw [color=black, line width=0.8 ] (VMLL\x) to (VMLR\x);
        \node[OpenVertex] at (VML\x) {};

        \draw [color=black, line width=0.8 ] (VMA\x) to (VMB\x);
        \node[OpenVertex] at (VCen\x) {};
    };
    \foreach \x in {-2,2} {
        \draw [color=black, line width=0.8 ] (VMA\x) to (VTop\x);
        \draw [color=black, line width=0.8] (VMUL\x) to (VMUR\x);
        \node[OpenVertex] at (VMU\x) {};

        \draw [color=black, line width=0.8 ] (VMB\x) to (VBot\x);
        \draw [color=black, line width=0.8] (VMLL\x) to (VMLR\x);
        \node[OpenVertex] at (VML\x) {};

        \draw [color=black, line width=0.8] (VMA\x) to (VMB\x);
        \node[OpenVertex] at (VCen\x) {};
    };
    
    \draw [color=black, line width=0.8] (VCen-5)--(VCen-3);
    \draw [color=black, line width=0.8 ] (VCen-3)--(VCen-1);
    \draw [color=black, line width=0.8] (VCen-1)--(VCen1);
    \draw [color=black, line width=0.8 ] (VCen1)--(VCen3);
    \draw [color=black, line width=0.8] (VCen3)--(VCen5);
    \foreach \x in {-4,-2,...,4} {
        \node[OpenVertex] (VCen\x) at (\x,5) {};
    };
    
    \node[TextNode] at ([shift=(80:0.3)]VCen-3) {$v^*$};
    \node[TextNode] at ([shift=(40:0.4)]VCen4) {$w^*$};
    
    \node[TextNode] at ([shift=(205:0.5)]VCen-3) {1};
    \node[TextNode] at ([shift=(-25:0.5)]VCen-3) {2};    
    \node[TextNode] at ([shift=(205:0.5)]VCen4) {3};
    \node[TextNode] at ([shift=(-25:0.5)]VCen4) {4};
    \node[TextNode] at ([shift=(115:0.5)]VCen4) {5};
    \node[TextNode] at ([shift=(250:0.6)]VCen4) {6};
\end{scope}
%
\draw[rounded corners] (-5.7, -1.5) rectangle (5.7, 2) {};
%
\node[Vertex] (V) at (-4,0) {};
    \node[TextNode] at ([shift=(180:0.3)]V) {$v$};
    \node[TextNode] at ([shift=(90:1.6)]V) {$G(v) = \Sym(\{1,2\})$};
%
\node[OpenVertex] (W) at (4,0) {};
	\node[TextNode] at ([shift=(0:0.3)]W) {$w$};
    \node[TextNode] at ([shift=(230:1.3)]W) {$G(w) = \Sym(\{3,4,5,6\})$};
%
\draw [color=black, ->-, line width=0.8] (V) to [bend left=30] node[above] (colour_a) {$a$} (W);
    \node[TextNode,align=left] at ([shift=(0:2)]colour_a) {$X_a = \{1, 2\}$};
\draw [color=black, ->-, line width=0.8] (W) to [bend right=15]  node[below] (colour_aInv) {$\overline{a}$} (V);
    \node[TextNode,align=left] at ([shift=(-15:2.4)]colour_aInv) {$X_{\overline{a}} = \{3,4, 5, 6\}$};    
\draw[rounded corners] (-5.9, -1.8) rectangle (5.9, 7.4) {};
\end{tikzpicture}
}
\caption{A figure for Example~\ref{Ex:AutT} showing part of the $(2,4)$-biregular tree $T$ with labels (top), and (bottom) the associated local action diagram for the action $(T, G)$ when $G = \Aut(T)$.}
\label{Fig:AutTAssocLAD}
\end{center}
\end{figure}
\end{ex}

\begin{ex} \label{Ex:BM} Let $T$ be the $6$-regular tree and suppose $H := \langle (1 \ 2 \ 3), (4 \ 5 \ 6) \rangle \leq \Sym(6)$. If $G$ is the Burger--Mozes group $U(H)$ (see \cite{BurgerMozes}) then $G$ is $\propP{}$-closed, has one orbit on $VT$ and $2$ orbits on $AT$. Notice that any arc and its reverse lie in a common orbit. In Figure~\ref{Fig:AssocLADBMGp} on page \pageref{Fig:AssocLADBMGp} we see the local action diagram $\Delta$ for the action $(T, U(H))$. In this example we have that the Burger--Mozes group $U(H)$ is the universal group $\Univ(\Delta)$ of the local action diagram $\Delta$.

\begin{figure}
\begin{center}
\scalebox{0.85}{ 
\begin{tikzpicture}
\draw[rounded corners] (-5.7, -2) rectangle (5.7, 2) {};
%
\node[Vertex] (V) at (0,0) {};
    \node[TextNode] at ([shift=(180:0.3)]V) {$v$};
    \node[TextNode] at ([shift=(180:2.8)]V) {$G(v) = \langle (1 \ 2 \ 3), (4 \ 5 \ 6) \rangle$};
    \draw[color=black, line width=0.8] (V) to  [out=40,in=120,distance=15mm] node[above] (colour_a) {$a = \overline{a}$} (V);
    \node[TextNode,align=left] at ([shift=(-10:2.5)]colour_a) {$X_a = X_{\overline{a}} = \{1,2,3\}$};
    \draw[color=black, line width=0.8] (V) to [out=-40,in=-120,distance=15mm] node[below] (colour_b) {$b = \overline{b}$} (V);
    \node[TextNode,align=right] at ([shift=(10:2.5)]colour_b) {$X_b = X_{\overline{b}} = \{4,5,6\}$};    
\end{tikzpicture}
}
\caption{A figure for Example~\ref{Ex:BM}, showing the local action diagram for the Burger--Mozes group $U(\langle (1 \ 2 \ 3), (4 \ 5 \ 6) \rangle)$.}
\label{Fig:AssocLADBMGp}
\end{center}
\end{figure}
\end{ex}

\section{$\propP{}$-closed subgroups of $\propP{}$-closed groups}\label{sec:tree}

Let $G$ be a $\propP{}$-closed group of automorphisms of the tree $T$.   
In this section, we examine subgroups of $\propP{}$-closed groups and look for those subgroups that are themselves $\propP{}$-closed. In \S\ref{LocalSubactionDiagrams}
we introduce local subaction diagrams. In \S\ref{BeingPClosed},
we present sufficient conditions for a subgroup of $G$ to be $\propP{}$-closed. Following this, we look at specific types of subgroups: 
vertex stabilisers in \S\ref{VertexStabsAndPClosure}; $G^+$ in \S\ref{sec:G+}; 
open subgroups containing a translation in \S\ref{sec:open_hyperbolic}
and stabilisers of locally invariant ends in \S\ref{sec:local_ends}. In \S\ref{sec:local_ends} we also prove Theorem~\ref{thm:open_primitive} from the introduction.

\subsection{Local subaction diagrams} \label{LocalSubactionDiagrams}

In this subsection we give a general description of $\propP{}$-closed subgroups of a $\propP{}$-closed group based on the respective local action diagrams.

\begin{defn} \label{defn:local_subaction_diagram}
Suppose $\Delta = (\Gamma, (X_a), (G(v)))$ is a local action diagram. A \defbold{local subaction diagram} of $\Delta$ is a local action diagram $\Lambda = (\hat{\Gamma}, (\hat{X}_{\hat{a}}), (H(\hat{v})))$ admitting a locally surjective graph homomorphism $p: \hat{\Gamma} \rightarrow \Gamma$, and for each $\hat{v} \in \hat{\Gamma}$ a bijection $\theta_{\hat{v}}: \hat{X}_{\hat{v}} \rightarrow X_{p(\hat{v})}$, satisfying the following conditions:
\begin{enumerate}[(i)]
\item
	For each vertex $\hat{v}$ of $\hat{\Gamma}$, then $\theta_{\hat{v}}H(\hat{v})\theta\inv_{\hat{v}}$ is a closed subgroup of $G(p(\hat{v}))$.
\item
	For each arc $a$ of $\Gamma$ and each vertex $\hat{v}$ of $\hat{\Gamma}$ such that $p(\hat{v}) = o(a)$, then $\theta_{\hat{v}}$ restricts to a bijection
	\[
	\theta_{\hat{v},a}: \bigsqcup_{\hat{a} \in o\inv(\hat{v}), p(\hat{a}) = a} \hat{X}_{\hat{a}} \rightarrow X_a. 
	\]
\end{enumerate}

If $\Lambda$ is a local subaction diagram of $\Delta$ with associated graph homomorphism $p: \hat{\Gamma} \rightarrow \Gamma$, then we write $\Lambda \leq_p \Delta$.  We combine the bijections $\theta_{\hat{v}}$ into a map
\[
\theta_p: \bigsqcup_{\hat{v} \in V\hat{\Gamma}}\hat{X}_{\hat{v}} \rightarrow \bigsqcup_{v \in V\Gamma}X_v; \; \text{ for } x \in \hat{X}_{\hat{v}}, \; \theta_p(x) := \theta_{\hat{v}}(x).
\]
In general $\theta_p$ is surjective but not necessarily injective.
\end{defn}

\begin{prop}\label{Prop:Subgroups_of_local_groups} 
Suppose that $\Delta = (\Gamma, (X_a), (G(v)))$ is a local action diagram and $\Lambda = (\hat{\Gamma}, (X_{\hat{a}}), (H(\hat{v})))$ is a local subaction diagram of $\Delta$, with $\Lambda \leq_p \Delta$.  If $\mathbf{\hat{T}} = (T, \mc{L}, \hat{\pi})$ is a $\Lambda$-tree, then $\mathbf{T}  = (T, \theta_p \circ \mc{L}, p \circ \hat{\pi})$ is a $\Delta$-tree and $\Aut_{\hat{\pi}}(T) \leq \Aut_{p \circ \hat{\pi}}(T)$. Moreover, $\Univ(\mathbf{\hat{T}}, (H(\hat{v}))) \leq \Univ(\mathbf{T},(G(v)))$.
\end{prop}
\begin{proof}
The $\Lambda$-tree $\mathbf{\hat{T}}$ is a tree $T$ together with a surjective graph homomorphism $\hat{\pi} : T \rightarrow \hat{\Gamma}$ and a $\Lambda$-colouring $\mc{L}: AT \rightarrow \bigsqcup_{\hat{a} \in A\hat{\Gamma}}X_{\hat{a}}$. For each vertex $v \in VT$ and each arc $\hat{a} \in o\inv(\hat{\pi}(v))$, the map $\mc{L}$ restricts to a bijection $\mc{L}_{v, \hat{a}}$ from $\{b \in o\inv (v) \mid \hat{\pi}(b) = \hat{a}\}$ to $\hat{X}_{\hat{a}}$ (where here $o\inv$ has codomain $AT$).

Define $\pi := p \circ \hat{\pi}$, and note that $\pi : T \rightarrow \Gamma$ is a surjective graph homomorphism.  Now we see that $\theta_p \circ \mc{L}$ is a map from $AT$ to $\bigsqcup_{a \in A\Gamma}X_{a}$.  Let us now consider the restriction of $\mc{L}' := \theta_p \circ \mc{L}$ to $o\inv(v)$ for some $v \in VT$.  Then $\hat{\pi}(v)$ is some vertex $\hat{v} \in V\hat{\Gamma}$ and $\mc{L}$ restricts to a bijection from $o\inv(v)$ to $\hat{X}_{\hat{v}}$; in turn, $\theta_p$ restricts to $\theta_{\hat{v}}$ on $\hat{X}_{\hat{v}}$.  Thus $\mc{L}'$ restricts to a bijection from $o\inv(v)$ to $X_{p(\hat{v})} = X_{\pi(v)}$.  Now fix $a \in o\inv(\pi(v))$; given $b \in o\inv(v)$, consider when $\mc{L}'(b) \in X_a$.  By the second condition in Definition~\ref{defn:local_subaction_diagram}, this occurs if and only if $\mc{L}(b) \in \hat{X}_{\hat{b}}$ for some $\hat{b} \in o\inv(\hat{v})$ such that $p(\hat{b}) = a$. By the definition of a $\Lambda$-colouring, this is equivalent to the condition that $p\circ \hat\pi(b) = a$, in other words, $\pi(b)=a$. Thus $\mc{L}'$ restricts to a bijection
\[
\mc{L}'_{v,a} : \{b \in o\inv(v) \mid \pi(b) = a\} \rightarrow X_a.
\]
Thus $\mc{L}'$ is a $\Delta$-colouring and hence $\mathbf{T}$ is a $\Delta$-tree.  For $\theta \in \Aut_{\hat{\pi}}(T)$ we have $\theta \in \Aut(T)$ and $\hat{\pi} \circ \theta = \hat{\pi}$. Hence $(p \circ \hat{\pi}) \circ \theta = p \circ \hat{\pi}$, so $\theta \in \Aut_{\pi}(T)$.

Fix $g \in \Univ(\mathbf{\hat{T}}, (H(v)))$; note that $g \in \Aut_{\hat{\pi}}(T) \leq \Aut_{\pi}(T)$.  Given $v \in VT$, then $g$ has local actions with respect to $\mc{L}$ and $\mc{L}'$ at $v$, namely 
$\sigma_{\mc{L},v}(g) \in \Sym(\hat{X}_{\hat{\pi}(v)})$ and $\sigma_{\mc{L}',v}(g) \in \Sym(X_{\pi(v)})$; moreover $\sigma_{\mc{L},v}(g) \in H(\hat{\pi}(v))$.  From the construction of $\mc{L}'$ from $\mc{L}$, we see that in fact 
\[
\sigma_{\mc{L}',v}(g)\theta_{\hat{\pi}(v)} = \theta_{\hat{\pi}(v)} \sigma_{\mc{L},v}(g).
\]
The first condition of Definition~\ref{defn:local_subaction_diagram} now ensures that $\sigma_{\mc{L}',v}(g) \in G(\pi(v))$.  Hence $g \in \Univ(\mathbf{T}, (G(v)))$, showing that $\Univ(\mathbf{\hat{T}}, (H(v))) \le \Univ(\mathbf{T}, (G(v)))$.
\end{proof}

In light of Theorem~\ref{thm:unique_U} and the definition of $ \Univ(\Delta)$, we have the following immediate corollary.

\begin{cor} \label{cor:U_of_subdiagram_is_subgroup} If $\Delta$ is a local action diagram with local subaction diagram $\Lambda$, then $\Univ(\Lambda)$ occurs as a subgroup of $\Univ(\Delta)$. \qed
\end{cor}

Conversely, every $\propP{}$-closed subgroup of a $\propP{}$-closed group has an associated local subaction diagram.

\begin{prop}\label{prop:P-closed_subgroup_diagram}
Let $T$ be a tree, let $G$ be a $\propP{}$-closed subgroup of $\Aut(T)$, and let $H \le G$ such that $H = H^{\propP{}}$.  Let $\Delta = \Delta(T,G)$ and $\Lambda = \Delta(T,H)$.  Then $\Lambda$ is a local subaction diagram of $\Delta$.
\end{prop}

\begin{proof}
Note that $G$ and $H$ both have closed local actions by Lemma~\ref{lem:closed_local_actions}.  We write $\Delta = (\Gamma, (X_a), (G(v)))$ and $\Lambda = (\hat{\Gamma}, (\hat{X}_{\hat{a}}), (H(\hat{v})))$, with associated quotient maps $\pi: T \rightarrow \Gamma$ and $\hat{\pi}: T \rightarrow \hat{\Gamma}$ respectively.  Suppose we have defined $\Delta$ using the representative vertices $V^* \subseteq VT$ (one vertex for each $G$-orbit on $VT$), and we have defined $\Lambda$ using the representative vertices  $V^{**} \subseteq VT$ (one vertex for each $H$-orbit on $VT$).  For each $v \in VT$ we have chosen $g_v \in G$ and $h_v \in H$ such that $g_vv \in V^*$ and $h_vv \in V^{**}$.  By Lemma~\ref{lem:sameT}, the choices we make in the constructions do not change $\Delta$ and $\Lambda$ up to isomorphism of local action diagrams.  Since every $G$-orbit is partitioned into $H$-orbits, we may assume $V^* \subseteq V^{**}$ without loss of generality.

If $e$ is a vertex or arc of $T$, then by definition $\pi(e) = Ge$ and $\hat{\pi}(e) = He$.  Since $H \le G$, we can simply set $p(He) = GHe = Ge$ without ambiguity.  This defines a  
locally surjective graph homomorphism from $\hat{\Gamma}$ to $\Gamma$, with the property that $p \circ \hat{\pi} = \pi$.

Given $\hat{v} \in V\hat{\Gamma}$, there is a unique $w \in V^{**}$ such that $\hat{\pi}(w)=\hat{v}$, and then a unique $w' \in V^*$ in the same $G$-orbit as $w$; in fact $w' = g_ww$, and we note that $\pi(w') = \pi(w)$.  The construction of $\Delta$ and $\Lambda$ is such that $X_{\pi(w')} = o\inv(w')$ and $\hat{X}_{\hat{v}} = o\inv(w)$.  We can thus define the bijection $\theta_{\hat{v}}: \hat{X}_{\hat{v}} \rightarrow X_{p(v)}$ by setting $\theta_{\hat{v}}(b) = g_wb$.  
 
Let $K$ be the subgroup of $\Sym(o\inv(w))$ induced by the action of $H_w$ on $o\inv(w) = \hat{X}_{\hat{v}}$. Since $\Lambda$ is an associated local action diagram, by definition $H(\hat{v})$ is the closure of $K$. However, $K$ is closed because $H$ has closed local actions, so in fact $H(\hat{v}) = K$. Thus $\theta_{\hat{v}}H(\hat{v})\theta\inv_{\hat{v}}$ is the action of $g_wH_wg\inv_w$ on $o\inv(w')$; the latter is clearly a closed subgroup of $G_{w'}$, so we see that 
$\theta_{\hat{v}}H(\hat{v})\theta\inv_{\hat{v}}$ 
is a closed subgroup of $G(p(\hat{v}))$.  Thus condition (i) of Definition~\ref{defn:local_subaction_diagram} is satisfied.

For condition (ii), let $\hat{v}$, $w$ and $w'$ be as before and let $a \in A\Gamma$ be such that $p(\hat{v}) = o(a)$.  Then $a = Gb$ for some $b \in o\inv(w)$.  Given $c \in \hat{X}_{\hat{v}}$, then $c \in \hat{X}_{\hat{a}}$ for some $\hat{a} \in o\inv(\hat{v}) \subseteq A\hat{\Gamma}$; we can write $\hat{X}_{\hat{a}} = H_wb'$  for some $b' \in o\inv(w)$.  If $p(\hat{a}) = a$ then $Gb' = Gb$, so $g_wc \in Gb \cap o\inv(w') = X_a$.  If $p(\hat{a}) \neq a$ then $Gb'$ is disjoint from $Gb$ and hence $g_wc$ is not in $X_a$.  Thus $\theta_{\hat{v}}$ restricts to a bijection $\theta_{\hat{v},a}$ as required for condition (ii) of Definition~\ref{defn:local_subaction_diagram}. We conclude that $\Lambda$ is a local subaction diagram of $\Delta$ as required.
\end{proof}

\begin{rem} \label{rem:limitations_of_subaction}
As a description of $\propP{}$-closed subgroups of $\propP{}$-closed groups, the results we have just obtained fall short of a classification in two respects.
\begin{enumerate}[(i)]
\item There is no obvious recipe to construct all the local subaction diagrams $\Lambda = (\hat{\Gamma}, (X_{\hat{a}}), (H(\hat{v})))$ of a given local action diagram $\Delta = (\Gamma, (X_a), (G(v)))$.  The case $\Gamma = \hat{\Gamma}$ is relatively straightforward: $\Lambda$ is then determined by choosing the tuple $(H(v))_{v \in V\Gamma}$, where $H(v)$ is a closed subgroup of $G(v)$ acting transitively on each $G(v)$-orbit, and the entries of $(H(v))_{v \in V\Gamma}$ can be chosen independently of one another.  However as soon as we move away from this case, there are dependencies between the choices of local actions $H(\hat{v})$.  For example, even if we assume $V\Gamma = V\hat{\Gamma}$, then for each arc $a \in A\Gamma$, the number of orbits of $H(o(a))$ on $X_a$ must be the same as the number of orbits of $H(t(a))$ on $X_{\ol{a}}$.
\item Generally speaking, when classifying subgroups of a group of a special kind, one seeks a classification up to conjugacy.  In other words, given a $\propP{}$-closed action $(T,G)$ on a tree $T$, we would want a classification of $\propP{}$-closed subgroups of $G$ up to conjugacy in $G$.  Local subaction diagrams of $\Delta(T,G)$ only yield a classification of $\propP{}$-closed subgroups of $G$ up to conjugacy in $\Aut(T)$, which is a coarser classification in general.
\end{enumerate}
\end{rem}

In the rest of this section, we will focus on the relationship between $\propP{}$-closed subgroups of $G$ and open subgroups of $G$, where $G$ is a $\propP{}$-closed group.  This will lead to strong restrictions on the open subgroups, and also to the proof of Theorem~\ref{thm:open_primitive} from the introduction.

\subsection{Sufficient conditions for a subgroup to be $\propP{}$-closed} \label{BeingPClosed}

\begin{lem}\label{lem:arc_stabiliser:1closed}
Let $T$ be a tree, let $G$ be a $\propP{}$-closed subgroup of $\Aut(T)$, and let $H \le G$.  Suppose that $H$ contains an arc stabiliser of $G$.  Then $H$ is $\propP{}$-closed.
\end{lem}

\begin{proof}
Since $T$ is a simple graph, we associate each arc $a \in AT$ with the vertex pair $(o(a), t(a))$ and write $(o(a), t(a)) \in AT$. Now suppose
that $H$ contains an arc stabiliser $G_{(x,y)}$ of $G$ for some arc $(x,y) \in AT$.  
Note that $G$ must be a closed subgroup of $\Aut(T)$ by Proposition~\ref{prop:kclosure}. Since $H$ is an open subgroup of $G$, it is also a closed subgroup of $\Aut(T)$. Thus, by Theorem~\ref{propertyP_oneclosure}, it suffices for us to show that $H$ has property $\propP{}$ with respect to the edges of $T$.

Let $(v,w) \in AT$; we claim that $H$ has $\propP{}$ with respect to $(v,w)$.  If $(v,w)$ is equal to $(x,y)$ or its reverse, then $H_{(v,w)} = G_{(x,y)}$ and the claim is clear, so we may assume $(v,w) \not\in \{(x,y),(y,x)\}$.  Then there is a simple path that passes through $x$ and $y$ in some order, and then later through $v$ and $w$ in some order; without loss of generality, let us say that the path passes through these points in the order $x,y,v,w$.

Since $G$ has property $\propP{}$, we have $G_{(v,w)} =  \rist_G(T_{(v,w)}) \times \rist_G(T_{(w,v)})$.  Note that $(x,y)$ is outside $T_{(v,w)}$, and hence
\[
\rist_G(T_{(v,w)}) \le G_{(x,y)} \le H.
\]
For $h \in H_{(v,w)}$ we can write $h = k_1 k_2$ with $k_1 \in \rist_G(T_{(v,w)})$ and $k_2  \in \rist_G(T_{(w,v)})$.  Then $k_1 \in H$, so also $k_2 \in H$.  The arc stabiliser $H_{(v,w)}$ therefore decomposes as
\[
H_{(v,w)} = \rist_H(T_{(v,w)}) \times \rist_H(T_{(w,v)}).
\]
Thus $H$ satisfies property $\propP{}$ with respect to all edges in $T$.
\end{proof}

Every vertex stabiliser contains the stabilisers of the arcs incident with that vertex, so the following is a special case of Lemma~\ref{lem:arc_stabiliser:1closed}.

\begin{cor}\label{cor:vertex_stabiliser:1closed}
Let $G$ be a $\propP{}$-closed subgroup of $\Aut(T)$, and let $H \le G$.  Suppose that $H$ contains a vertex stabiliser of $G$.  Then $H$ is $\propP{}$-closed.
\end{cor}

\subsection{Vertex stabilisers and $\propP{}$-closure} \label{VertexStabsAndPClosure}

Given a $\propP{}$-closed action of a group $G$ on a tree $T$, then every vertex stabiliser of $G$ itself has $\propP{}$-closed action, by Corollary~\ref{cor:vertex_stabiliser:1closed}.  In turn, $\propP{}$-closed subgroups of $\Aut(T)$ that fix a vertex have a special structure.  This generalises the observation \cite[Proposition 15]{SmithDuke} that the box product of two permutation groups $M$ and $N$ contains isomorphic copies of $M$ and $N$ as subgroups.

\begin{prop}\label{prop:semidirect}
Let $G$ be a $\propP{}$-closed subgroup of $\Aut(T)$, 
let $\epsilon \in VT$ and let $B_r$ be the closed ball of radius $r$ around $\epsilon$.  Let $G_r$ be the pointwise stabiliser of $B_r$ in $G$.  Then there is an increasing sequence 
$(H_r)_{r \ge 1}$
of subgroups of $G_0$, each closed in $\Aut(T)$, such that $G_0 = G_r \rtimes H_r$ as topological groups for all $r \ge 1$.  In particular, there is a permutational isomorphism between the action of $H_1$ on $o\inv(\epsilon)$ and the corresponding vertex group $G(\pi(\epsilon))$ of the local action diagram of $G$.
\end{prop}

\begin{proof}
Let $\Gamma = G \backslash T$, let $\pi = \pi_{(T,G)}$ and let $v_0 = \pi(\epsilon)$.  Form the local action diagram $\Delta = (\Gamma,(X_a),(G(v)))$ for $(T,G)$.  If we replace $G$ with $G_0 = G_{\epsilon}$, then it will not affect the permutational isomorphism type of $G(v_0)$, nor the structure of $G_0$, so let us assume that $G = G_0$.
Since $G$ fixes $\epsilon$, any local action of $G$ at a vertex $w \in VT\smallsetminus \{\epsilon\}$ must fix pointwise the unique edge incident with $w$ contained in $[v_0,w]$. Thus in any colouring $\mc{L}$ such that $\mathbf{T} = (T,\pi,\mc{L})$ is a $\Delta$-tree,  if $a \in AT$ such that $d(t(a),\epsilon) < d(o(a),\epsilon)$, then $|X_{\pi(a)}|=1$, so $\mc{L}(a)$ must be the unique element of $X_{\pi(a)}$.  Fix such a colouring $\mc{L}$ and for $r \geq 1$, set 
\[
H_r = \{g \in G_0 \mid \forall w \in VT: d(w,\epsilon) \ge r \Rightarrow \sigma_{\mc{L},w}(g) = 1\}.
\]
Given that $G_0$ preserves distance from $\epsilon$, it is easy to see that $H_r$ is a subgroup of $G_0$; $H_r \cap G_r = \triv$; and $H_{r} \le H_{r'}$ whenever $r \le r'$.  It is also clear that $H_r$ is determined as a subgroup of $G_0$ by its orbits on arcs, so $H_r$ is closed in $G_0$; since $G_0$ is closed in $\Aut(T)$, it follows that $H_r$ is closed in $\Aut(T)$.  To see that $G_0 = G_rH_r$, consider some $r \ge 1$, an element $h \in G_0$, and a vertex $w \in VT$ with $d(\epsilon,w) = r$.  
Then $h$ maps the arc in $o\inv(w)$ directed towards $\epsilon$ to the arc in $o\inv(hw)$ directed towards $\epsilon$. This corresponds to a fixed point for the permutation $\sigma_{\mc{L},w}(h) \in G(\pi(w))$, so the action of $h$ on $B_r$ is compatible with having trivial local action at $r$.  Hence there is $g \in \Aut(T)$ that has the same action as $h$ on $B_r$, but has trivial local action for every vertex $w \in VT$ such that $d(w,\epsilon) \ge r$.  
We then see that in fact $g \in H_r$ and $h \in G_rg$, showing that $h \in G_rH_r$.  In particular, it is now clear that $\sigma_{\mc{L},\epsilon}$ restricts to a permutational isomorphism from $H_1$ acting on $o\inv(\epsilon)$ to $G(v_0)$ acting on $X_{v_0}$.

Finally, we show that as a topological group, $G_0$ splits as the semidirect product $G_r \rtimes H_r$.  Certainly $G_0$ splits this way abstractly, since $G_r$ is a closed normal subgroup that has trivial intersection with $H_r$.  To show that the permutation topology of $G_0$ is the product of the permutation topologies on $G_r$ and $H_r$, it suffices to show for every vertex $w \in VT$ that the vertex stabiliser $K = G_w$ satisfies $K = (G_r \cap K)(H_r \cap K)$.  Let $(a_1,\dots,a_t)$ be the directed path of arcs from $\epsilon$ to $w$.  Given $g \in G$, by assumption we have $g \in G_0$, and thus we observe that $g \in K$ if and only if $g$ fixes $a_i$ for $1 \le i \le t$, which occurs if and only if $\sigma_{\mc{L},o(a_i)}(g)$ fixes $\mc{L}(a_i)$ for $1 \le i \le t$.  On the other hand, if $g = g_rh_r$ for $g_r \in G_r$ and $h_r \in H_r$, then
\[
\sigma_{\mc{L},o(a_i)}(g) =
\begin{cases}
\sigma_{\mc{L},o(a_i)}(h_r) &\mbox{if} \;  i \le r\\
\sigma_{\mc{L},h_ro(a_i)}(g_r) &\mbox{if} \;  i > r.
\end{cases}
\]
From this, we see that $g$ fixes $w$ if and only if both $g_r$ and $h_r$ fix $w$.  Thus $K = (G_r \cap K)(H_r \cap K)$ as desired.  
\end{proof}
We note that for an arbitrary group action $(T,G)$ on a tree, the operation of taking the $\propP{}$-closure behaves well with respect to the subgroup  $G^+$  generated by the arc stabilisers.

\begin{prop}\label{propP_plusgroup}
Let $G$ be a group acting on a tree $T$.  Then
\[
(G^{\propP{}})^+ = (G^+)^{\propP{}}.
\]
\end{prop}

\begin{proof}
We see that $(G^+)^{\propP{}}$ has the same orbits on arcs as $G^+$, so $(G^+)^{\propP{}} = G^+ (G^+)^{\propP{}}_a$ for any $a \in AT$.  In turn, it is clear that $G^+ \le (G^{\propP{}})^+$ and $(G^+)^{\propP{}}_a \le (G^{\propP{}})_a \le (G^{\propP{}})^+$,  Thus $(G^{\propP{}})^+ \ge (G^+)^{\propP{}}$.

It remains to show that $(G^{\propP{}})^+ \le (G^+)^{\propP{}}$.  In fact it suffices to show that $(G^{\propP{}})_a \le (G^+)^{\propP{}}$ for $a \in AT$.  In other words we wish to show, for all $a \in AT$, $v \in VT$, $F$ finite subsets of $o\inv(v)$ and $g \in (G^{\propP{}})_a$, there exists $h \in G^+$ such that $hg$ fixes $F$ pointwise.  For this discussion we fix $a$ and $g$; we will proceed by induction on $d = d(v,o(a))$.

In the base case, $d = 0$, in other words $v = o(a)$.  Since $g \in G^{\propP{}}$, for any finite $F \subseteq o\inv(v)$ there is $h \in G$ such that $hg$ fixes $F$ pointwise; since $g$ fixes $a$, in fact $h \in G_a \le G^+$. Thus, $g \in (G^+)^{\propP{}}$.

From now on, we may assume $d > 0$.  Let $b$ be the arc in $o\inv(v)$ pointing towards $o(a)$.  Then by induction there is $h' \in G^+$ such that $h'g$ fixes $\ol{b} \in o\inv(t(b))$, and hence fixes $b$.  In turn for any finite $F \subseteq o\inv(v)$ there is $h \in G$ such that $h(h'g)$ fixes $F \cup \{b\}$ pointwise.  Since $h'g$ fixes $b$, in fact $h \in G_b \le G^+$, and hence $hh' \in G^+$.  Thus, $g \in (G^+)^{\propP{}}$. This completes the inductive step and hence completes the proof.
\end{proof}

\subsection{The local action of $G^+$} \label{sec:G+}

Given a tree $T$ and a $\propP{}$-closed group of automorphisms $G$, we can describe both the action of $G^+$ on $T$ and the action of $G$ on $G^+ \backslash T$ using local action diagrams derived from the local action diagram of $G$.
This description is given in Theorem~\ref{thm:plus_subaction_diagram}. The theorem has two interesting corollaries (\ref{cor:G+_vertex} and \ref{cor:quotient_tree:G+}) that apply to all groups acting on trees with closed local actions, not just to those that have Tits' property $\propP{}$.

\begin{defn}\label{defn:plus_subaction_diagram}
Given a local action diagram $\Delta = (\Gamma,(X_a),(G(v)))$ and $v \in V\Gamma$, write $G(v)^+$ for the subgroup of $G(v)$ generated by the point stabilisers of its action on $X_v$.

We use the relationship between the groups $G(v)$ and $G(v)^+$ to define two more local action diagrams, which will reappear later.

The \defbold{reduced local action diagram} of $\Delta$ is $\Delta^* = (\Gamma^*,(X^*_a),(G^*(v)))$, where $\Gamma^* = \Gamma$; $G^*(v) = G(v)/G(v)^+$; and $X^*_a$ is the $G^*(v)$-set formed by the set of $G(v)^+$-orbits on $X_a$.
Note that since $G(v)^+$ is normal and contains every point stabiliser of the action on $X_a$, actually $X^*_a$ is a regular $G^*(v)$-set. In particular, each group $G^*(v)$ is closed in the permutation topology.

The \defbold{plus local subaction diagram} of $\Delta$ is $\Delta^+ = (\Gamma^+,(X^+_a),(G^+(v)))$, constructed as follows.  We take $\Gamma^+$ to be a $\Delta^*$-tree, which is then equipped with a 
locally surjective graph quotient map $p: \Gamma^+ \rightarrow \Gamma$ and a colouring $\mc{L}^*$.  Given $v \in V\Gamma^+$, we can restrict $\mc{L}^*$ to a bijection
\[
\mc{L}^*_v: o\inv(v) \rightarrow \bigsqcup_{b \in o\inv(p(v))}X^*_b.
\]
An arc $a \in A\Gamma^+$ has a colour (when viewed as an arc in the $\Delta^*$-tree $\Gamma^+$), and this colour is an element of $X^*_{p(a)}$, which is in particular an orbit $Y_a$ of $G(p(v))^+$ in its original action on $X_{p(v)}$.  We can thus take the colour set $X^+_a$ of $a$ in $\Delta^+$ to be a set in bijection with $Y_a$, extending to a bijection $\theta_v$ from $X^+_v:= \bigsqcup_{a \in o\inv(v)}X^+_a$ to $X_{p(v)}$, and then let $G^+(v)$ be the group of permutations of $X^+_v$ such that $\theta_vG^+(v)\theta\inv_v = G(p(v))^+$.  As in Definition~\ref{defn:local_subaction_diagram}, we combine the bijections $\theta_v$ into a surjection
\[
\theta_p: \bigsqcup_{v \in V\Gamma^+}X^+_v \rightarrow \bigsqcup_{w \in V\Gamma}X_w.
\]
\end{defn}

It is now straightforward to verify the following.

\begin{lem}\label{lem:plus_subaction_diagram}
Let $\Delta$ be a local action diagram.  Then $\Delta^*$ is also a local action diagram, while $\Delta^+$ is a local subaction diagram of $\Delta$ as witnessed by the homomorphism $p$ and bijections $\theta_v$ defined in Definition~\ref{defn:plus_subaction_diagram}.
\end{lem}

We now give the interpretations of the reduced local action diagram and the plus local subaction diagram for the action of $\Univ(\Delta)$ on its $\Delta$-tree.

\begin{thm}\label{thm:plus_subaction_diagram}
Let $\Delta =  (\Gamma,(X_a),(G(v)))$ be a local action diagram.  Then there is a tree $T$, carrying the structures of both a $\Delta$-tree $\mathbf{T}$ and a $\Delta^+$-tree $\mathbf{T}^+$, such that $\Univ_{\mathbf{T}^+}(\Delta^+) = G^+$ for $G = \Univ_{\mathbf{T}}(\Delta)$, while $\Delta^*$ is the local action diagram of the action of $G/G^+$ on the tree $G^+ \backslash T$.  Moreover, the action of $G/G^+$ on $G^+ \backslash T$ has $\propP{}$; indeed, it acts freely on the arcs of $G^+ \backslash T$.
\end{thm}

\begin{proof}
Let $\mathbf{T}^+ = (T,\mc{L}^+,\pi^+)$ be a $\Delta^+$-tree.  Since $\Delta^+$ is a local subaction diagram of $\Delta$, we can apply Proposition~\ref{Prop:Subgroups_of_local_groups} to equip $T$ with the structure of a $\Delta$-tree $\mathbf{T} = (T,\mc{L},\pi)$, where $\mc{L} = \theta_p \circ \mc{L}^+$ and $\pi = p \circ \pi^+$, so that $H:= \Univ_{\mathbf{T}^+}(\Delta^+)$ is a subgroup of $G:= \Univ_{\mathbf{T}}(\Delta)$.

We now need to show that $H = G^+$.  It is enough to prove the following three statements:
\begin{enumerate}[(a)]
\item $H$ is generated by vertex stabilisers.
\item The vertex stabilisers of $H$ are generated by arc stabilisers.  Specifically, given $v \in VT$,
\[
H_v =  \grp{H_a \mid a \in o\inv(v)}.
\]
\item For all $a \in AT$ we have $G_a \le H$.
\end{enumerate}

Statement (a) follows from Corollary~\ref{cor:vertex_group}, since $\Gamma^+$ is a tree by construction.  Taking $v \in VT$, then the action of $H_v$ on $o\inv(v)$ is isomorphic as a permutation group to $G^+(\pi^+(v)) \cong G(\pi(v))^+$, which is generated by point stabilisers.  The point stabilisers of the action of $H_v$ on $o\inv(v)$ are the arc stabilisers for arcs $a \in o\inv(v)$.  Statement (b) now follows, since $H_v$ and $\grp{H_a \mid a \in o\inv(v)}$ have the same action on $o\inv(v)$ and both groups contain $H_a$ for $a \in o\inv(v)$. 

Now fix $a \in AT$ and consider $G_a$.  Write $v_0 = o(a)$; let $B_r$ be the ball of the radius $r$ around $v_0$ and let $G_r$, respectively $H_r$, be the pointwise stabiliser of $B_r$ in $G$, respectively $H$.  Since the action of $G_a$ on $o\inv(v_0)$ fixes a point, we see that it acts as a subgroup of $G(\pi(v))^+$; we deduce that $\sigma_{\mc{L},v_0}(H_a) = \sigma_{\mc{L},v_0}(G_a)$, so that $G_a = H_a G_1$.  Suppose that $G_a = H_aG_r$ for some $r \ge 1$, let $g \in G_r$ and let $v$ be a vertex at distance $r$ from $v_0$. There is a unique arc $a' \in o\inv(v) \cap B_r$. Since $\sigma_{\mc{L},v}(H_{a'}) = \sigma_{\mc{L},v}(G_{a'})$, there is an element $h_v \in H_{a'}$ such that $h_v$ and $g$ induce the same action on $o\inv(v)$.
By property $\propP{}$ we can take $h_v$ to lie in the pointwise stabiliser of $T_{a'}$. By taking a product of such elements as $v$ ranges over the sphere of radius $r$ around $v_0$, we obtain $h \in H_r \leq H_a$ such that $g \in hG_{r+1}$. Thus, $G_a = H_aG_{r+1}$. By induction we have $G_a = H_aG_r$ for all $r \ge 1$.  Since $H_a$ is closed, we conclude that $G_a = H_a$, proving (c).

We now consider the tree $G^+ \backslash T$; since $G^+ = \Univ_{\mathbf{T}^+}(\Delta^+)$, we can naturally identify $G^+ \backslash T$ with the graph $\Gamma^+$ of $\Delta^+$; since $G^+$ is normal in $G$ this identification yields an action $\alpha$ of $G$ on $\Gamma^+$.  By construction $\Gamma^+$ has the structure of a $\Delta^*$-tree $\mathbf{T}^*$ equipped with a quotient map $p: \Gamma^+ \rightarrow \Gamma$ and a colouring $\mc{L}^*$: we claim that $\alpha(G) \leq \Univ_{\mathbf{T}^*}(\Delta^*)$. 
Since $\pi = p \circ \pi^+$, we see that for each vertex or arc $e$ of $\Gamma^+$, then $Ge = p(e)$, so that $G$ acts as a subgroup of $\Aut_p(\Gamma^+)$ with the correct orbits on vertices and arcs.  For the local actions, via our identification of $G^+\backslash T$ with $\Gamma^+$ and $\mathbf{T}^*$, we see that for $v \in VT$, the colouring at the vertex $G^+v$ restricts to a bijection
\[
\mc{L}^*_{G^+v}: o\inv(G^+v) \rightarrow  \bigsqcup_{b \in o\inv(Gv))}X^*_b.
\]
Given $a \in o\inv(v)$, then $\mc{L}^*(G^+a) = X^+_{\pi^+(a)}$, which is the $G(\pi(v))^+$-orbit on $X_{\pi(a)}$ that contains $\mc{L}(a)$.  Consequently, given $g \in G$, then $\sigma_{\mc{L}^*,G^+v}(\alpha(g))$ is obtained by taking the permutation $\sigma_{\mc{L},v}(g)$, which acts on $X_{\pi(v)}$, and reading off its action on the set of $G(\pi(v))^+$-orbits on $X_{\pi(v)}$.  From this description we see that $\sigma_{\mc{L}^*,G^+v}(\alpha(G)) = G^*(\pi(v))$, and hence $\alpha(G) \leq \Univ_{\mathbf{T}^*}(\Delta^*)$ with $\Delta^*$ the local action diagram of $\alpha(G)$ on $\Gamma^+$. 

Finally, given an arc $a$ of $\Gamma^+$, we can write $a = G^+b$ for $b \in AT$, and then
\[
\alpha(g)a = a \Leftrightarrow gb \in G^+b \Leftrightarrow g \in G^+G_b = G^+;
\]
thus $\alpha$ has kernel $G^+$ and we have a free action of $G/G^+$ on $A\Gamma^+$. Moreover, $\alpha(G)$ can now be identified with the action of $G/G^+$ on the tree $G^+\backslash T$, so this latter action has local action diagram $\Delta^*$, and a free action on arcs is easily seen to be closed with property $\propP{}$.
\end{proof}

We note two results arising from Theorem~\ref{thm:plus_subaction_diagram} and its proof that apply to groups acting on trees with closed local actions, without assuming property $\propP{}$.  (We remind the reader that for a group $G$ acting on a tree, $G^+$ remains the subgroup of $G$ generated by arc stabilisers.)

\begin{cor}\label{cor:G+_vertex}
Let $T$ be a tree and let $G$ be a subgroup of $\Aut(T)$ with closed local actions.  Then for all $v \in VT$,
\[
(G^+)_v = \grp{G_a \mid a \in o\inv(v)}.
\]
\end{cor}

\begin{proof}
Clearly $(G^+)_v$ contains $\grp{G_a \mid a \in o\inv(v)}$.
If the action of $(G^+)_v$ on $o\inv(v)$ is generated by point stabilisers (that is, $\grp{G_a \mid a \in o\inv(v)}$ and $(G^+)_v$ induce the same permutation group on $o\inv(v)$), then for any $g \in (G^+)_v$ and any $b \in o\inv(v)$ there is an $h \in \grp{G_a \mid a \in o\inv(v)}$ such that $h^{-1}gb = b$, so $h^{-1}g \in G_b \leq \grp{G_a \mid a \in o\inv(v)}$, and from this it follows that $(G^+)_v$ and $\grp{G_a \mid a \in o\inv(v)}$ are equal.  Thus it suffices to show that the action of $(G^+)_v$ on $o\inv(v)$ is generated by point stabilisers.

Since $G$ has closed local actions, the same is true of $G^+$ by Lemma~\ref{lem:closed_local_actions}(iii).  Consequently $(G^+)^{\propP{}}$ has the same local actions as $G^+$, so it is enough to show that the action of $((G^+)^{\propP{}})_v$ on $o\inv(v)$ is generated by point stabilisers, for all vertices $v$.  Since $(G^{\propP{}})^+ = (G^+)^{\propP{}}$ by Proposition~\ref{propP_plusgroup}, without loss of generality we can now replace $G$ with $G^{\propP{}}$, and assume that $G$ is $\propP{}$-closed.  In other words, by Theorem~\ref{thm:G_to_U}, we can identify $G$ with $\Univ(\Delta)$ for some local action diagram $\Delta$.   The conclusion is now statement (b) from the proof of Theorem~\ref{thm:plus_subaction_diagram}, given that we showed $H = G^+$ in that context.
\end{proof}

\begin{cor}\label{cor:quotient_tree:G+}
Let $(T,G)$ be an action on a tree, with local action diagram $\Delta = (\Gamma,(X_a),(G(v)))$.  Then $(G^+ \backslash T,G/G^+)$ is a $\propP{}$-closed action (indeed, a free action on arcs) which depends only on $(\Gamma,(X_a),(G(v)))$; in particular it is unchanged if we replace $G$ with $G^{\propP{}}$.  The action $(G^+ \backslash T,G/G^+)$ admits the following two equivalent descriptions:
\begin{enumerate}[(i)]
\item It is the $\propP{}$-closed action admitting the reduced local action diagram $\Delta^*$ of $\Delta$ as given in Definition~\ref{defn:plus_subaction_diagram}.
\item It is the fundamental group of a graph of groups over the graph $\Gamma^i$, arising as the quotient by the action of $G$ on the inversion-free subdivision $T^i$ of $T$, with the following data.  Given $v \in V\Gamma^i$, if $v$ is the image of a vertex of $T$, then the vertex group is $G^*(v)$ as in (i) for the corresponding vertex of $\Gamma$; if instead $v$ is the centre of an edge of $T$ that is reversed by $G$, then the vertex group is $C_2$ in its natural action on two points.  The edge groups and associated embeddings for the graph of groups are all trivial.
\end{enumerate}
\end{cor}

\begin{proof}
Let $L = G^{\propP{}}$.  By Proposition~\ref{propP_plusgroup} we have $L^+ = (G^+)^{\propP{}}$.  In particular, $L^+$ and $G^+$ have the same orbits on arcs, so $L^+ \backslash T = G^+ \backslash T$.  Since $G$ and $L$ have the same orbits on arcs, we have $L = GL_a$ and in particular $L = GL^+$, so for every element of $L$, there is an element of $G$ with the same action on $L^+ \backslash T$. Recall from the definition of an associated local action diagram (Definition~\ref{def:AssocLocalActionDiagram}) that the groups $G(v)$ are closed. Thus we can also take $(\Gamma,(X_a),(\ol{G(v)})) = \Delta$ as the local action diagram of $L$ and we see, for all $v \in V\Gamma$, that
\[
\ol{G(v)}/\ol{G(v)}^+ = G(v)/G(v)^+
\]  
as permutation groups acting on the set of orbits of $G(v)^+$ on $X_v$, that is on $X_v/G(v)^+ = X_v/\ol{G(v)}^+$. So for the rest of the proof, it makes no difference if we replace $G$ with $L$, so we may assume that $G$ is $\propP{}$-closed.

The description (i) now follows from Theorem~\ref{thm:plus_subaction_diagram}.  In particular, in the action $(G^+ \backslash T,G/G^+)$, all arc stabilisers are trivial.  In this context, we can easily deduce the graph of groups description given in (ii).
\end{proof}

\subsection{Open subgroups containing a translation}\label{sec:open_hyperbolic}

We now show that given a nondiscrete $\propP{}$-closed group $G$ of automorphisms of a tree, if $g \in G$ is a translation, then there is a unique smallest open subgroup $\Res_G(g)$ of $G$ normalised by that element 
(the group $\Res_G(g)$ is an example of a discrete residual; see \cite{ReidFlat}, \cite{ReidDistal} for example).
Given two translations $g$ and $h$, whether or not $\Res_G(g) = \Res_G(h)$ depends on a certain class of subtrees determined by the local actions of $G$ (see Proposition~\ref{prop:quasiblock_subtree}). We introduce the following.

\begin{defn}
Given a permutation group $H$ acting on a set $X$, define a \defbold{quasiblock} of the action to be a nonempty subset $Y$ of $X$ with the property that the setwise stabiliser $H_Y$ of $Y$ contains the point stabiliser $H_y$ for every $y \in Y$.

Given a group $G$ acting on a tree $T$, say that $T'$ is a \defbold{quasiblock subtree} for the action of $G$ on $T$ if $T'$ is a subtree such that $AT'$ is a quasiblock for the action of $G$ on $AT$.
\end{defn}

\begin{lem}\label{lem:prim_quasiblock} If $H \leq \Sym(X)$ is primitive but not regular, and $Y \subseteq X$ is a quasiblock, then either $Y = X$ or $|Y| = 1$.
\end{lem}
\begin{proof} Suppose $H \leq \Sym(X)$ is primitive, with quasiblock $Y$ satisfying $|Y|>1$. For all $y \in Y$ we have $H_y \leq H_Y \leq H$, with $H_y$ maximal. Therefore either $H_y = H_Y$ for all $y \in Y$ or $H_Y = H$. The latter implies $Y$ is an orbit of $H$, and since $H$ is transitive on $X$ we have $Y = X$. Suppose then that $H_y = H_Y$ for all $y \in Y$. There are thus distinct $y,y'\in Y$ such that $H_y = H_Y = H_{y'}$. Pairs of elements of $X$ with equal point stabilisers thus form a nontrivial $H$-invariant equivalence relation on $X$. Since $H$ is primitive, any stabiliser $H_y$ must fix $X$ pointwise and is therefore trivial; that is, $H$ is regular on $X$.
\end{proof}

Notice that quasiblock subtrees are naturally constrained by the local actions of $G$: namely,  if $T'$ is a quasiblock subtree, then $o\inv(v) \cap AT'$ is a quasiblock of $G_v$ for all $v \in VT'$. For $g \in G$ we say that a subgroup $H \leq G$ is $\grp{g}$-invariant if $g^{-1}Hg = H$.

\begin{prop}\label{prop:quasiblock_subtree}
Let $T$ be a tree, let $G$ be a nondiscrete $\propP{}$-closed subgroup of $\Aut(T)$, let $g \in G$ be a translation and let $\Res_G(g)$ be the intersection of all open $\grp{g}$-invariant subgroups of $G$.
\begin{enumerate}[(i)]
\item There is a unique smallest $G$-quasiblock subtree $T'$ of $T$ containing the axis of $g$.
\item We have $\Res_G(g) = \grp{G_e \mid e \in AT'}$.  In particular, $\Res_G(g)$ is open and determined as a subgroup of $G$ by $T'$.
\item In the action of $O = \Res_G(g)\grp{g}$ on $T$, the smallest invariant subtree is $T'$, and $T'$ is spanned by the axes of translation of $O$.
\end{enumerate}
\end{prop}

For the proof, we use a lemma extracted from the proof of \cite[Proposition~2.7]{MollerVonk}; we include the proof for clarity.

\begin{lem}\label{lem:open_subgroup_with_translation}
Let $T$ be a tree, let $G$ be a closed subgroup of $\Aut(T)$, let $H$ be an open subgroup of $G$ acting with translation, and let $T'$ be the smallest invariant subtree for the action of $H$.  Then $\rist_G(T_e) \times \rist_G(T_{\ol{e}}) \le H$ for all $e \in AT'$.
\end{lem}

\begin{proof}
By Lemma~\ref{lem:minimal_invariant}, $e$ lies on the axis of some $h \in H$.  Then $T_e$ is comparable with $T_{he}$; without loss of generality $T_{he} \subseteq T_e$.  We then see that $\bigcap_{n \ge 0}T_{h^ne} = \emptyset$.  Since $H$ is open, there are $v_1,\dots,v_k \in VT$ such that $\bigcap^k_{i=1}G_{v_i} \le H$.  Now take $n$ large enough that $VT_{h^ne}$ is disjoint from $\{v_1,\dots,v_k\}$.  Then
\[
h^n\rist_G(T_e)h^{-n} = \rist_G(T_{h^ne}) \le \bigcap^k_{i=1}G_{v_i} \le H,
\]
so $\rist_G(T_e) \le H$.  Similarly, $\rist_G(T_{\ol{e}}) \le H$.
\end{proof}

\begin{proof}[Proof of Proposition~\ref{prop:quasiblock_subtree}]
Given a subtree $T^*$ of $T$, write $H(T^*) = \grp{G_a \mid a \in AT^*}$.  We can construct $T'$ recursively as follows: start with the axis $T_0$ of $g$, and thereafter, for $\alpha$ an ordinal we set $T_{\alpha+1}$ to be the union of all $H(T_{\alpha})$-images of $T_{\alpha}$, and for $\lambda$ a limit ordinal we set $T_{\lambda} = \bigcup_{\alpha < \lambda}T_{\alpha}$.  By transfinite recursion this process eventually terminates, 
yielding a tree $T_{\beta} = T_{\beta+1}$ that is a $G$-quasiblock subtree of $T$.  It is clear from the construction that $T' := T_{\beta}$ is then the unique smallest $G$-quasiblock subtree containing the axis of $g$.  This proves (i).

Let $O = \Res_G(g)\grp{g}$, let $P$ be an open $\grp{g}$-invariant subgroup of $G$ and let $e$ be an arc on the axis of $g$.   By Lemma~\ref{lem:open_subgroup_with_translation}, we have $G_e \le P\grp{g^n}$ for all $n > 0$, so 
\[
G_e \le \bigcap_{n > 0} P\grp{g^n} = P.
\]
Thus $G_e \le \Res_G(g)$; in particular, $\Res_G(g)$ is open, so $O$ is open.  In addition, we see that $\Res_G(g)$ is actually the smallest open normal subgroup of $O$, that is, $\Res_G(g) = \Res(O)$.

Since $O$ acts with translation, by Lemma~\ref{lem:minimal_invariant}, the union of the axes of translation of $O$ is a subtree $T''$ that is also the smallest invariant subtree for $O$.  Let $R = \grp{G_e \mid e \in AT''}$.  From the previous paragraph, we see that $R \le \Res(O)$.  On the other hand, $R$ is clearly open and normalised by $O$, and $\Res(O)$ is the smallest open normal subgroup of $O$, in fact $R = \Res(O)$ as claimed.

To finish the proof, we argue that $T'' = T'$.  The stabiliser of $T'$ in $G$ is an open subgroup (since it contains arc stabilisers) that contains $g$ (by the uniqueness of its construction from the axis of $g$).  In particular, it follows that $O$ stabilises $T'$, ensuring that $T'' \subseteq T'$.  On the other hand, since $R \le O$, we see that $T''$ is stabilised by each of its arc stabilisers in $G$, that is, $T''$ is a $G$-quasiblock subtree; since $T''$ also contains the axis of $g$, it follows that $T' \subseteq T''$, and hence $T' = T''$.  This completes the proof of (ii) and (iii).
\end{proof}

As an application of Proposition~\ref{prop:quasiblock_subtree}, we obtain a restriction on the closed subgroups of a $\propP{}$-closed group that can constructed from open subgroups.  To make this result precise, we recall a class of subgroups introduced in \cite{ReidDistal}.

\begin{defn}
Let $G$ be a topological group and let $H$ be a closed subgroup of $G$.  Then $H$ is a \defbold{RIO subgroup} if $H$ is a directed union of a family of subgroups $\mc{O}$, such that each $O \in \mc{O}$ is an open subgroup of $H$ and an intersection of open subgroups of $G$.
\end{defn}

When $G$ is a totally disconnected locally compact group, the class of RIO subgroups has several closure properties.  For example, every closed subnormal subgroup is RIO; an intersection of RIO subgroups is RIO; if $K$ is RIO in $H$ and $H$ is RIO in $G$, then $K$ is RIO in $G$; and if $G$ acts distally (for example, by isometries) on a Hausdorff topological space $X$ such that the map $g \mapsto gx$ is continuous for $x \in X$, then $G_x$ is a RIO subgroup of $G$.  (See \cite{ReidDistal} for details.)  Given the scope of the class of RIO subgroups, it is then striking that when $G$ is a $\propP{}$-closed group with compact stabilisers, every RIO subgroup of $G$ is either `small' (a directed union of compact subgroups, or contained in a vertex stabiliser), or it contains one of the `large' open subgroups $O = \Res_G(g)\grp{g}$ described by Proposition~\ref{prop:quasiblock_subtree}.

\begin{thm}\label{thm:RIO}
Let $T$ be a tree and let $G$ be a $\propP{}$-closed subgroup of $\Aut(T)$ with compact arc stabilisers.  Let $H$ be a RIO subgroup of $G$.  Then at least one of the following holds.
\begin{enumerate}[(i)]
\item Every compactly generated closed subgroup of $H$ is compact.
\item $H$ fixes exactly one vertex of $T$.
\item $H$ acts with translation and contains an arc stabiliser of $G$.  In particular, $H$ is open in $G$ and $\propP{}$-closed.
\end{enumerate}
\end{thm}

\begin{proof}
Note that since arc stabilisers are open, $G$ must be locally compact.
By \cite[Proposition~4.9]{ReidDistal}, it follows that every compactly generated open subgroup of $H$ is an intersection of open subgroups of $G$.

Let us first consider the case that $H$ is compactly generated.  If $H$ preserves an undirected edge (which includes the case when $H$ fixes two or more vertices), then $H$ is compact.  If $H$ fixes exactly one vertex, then (ii) holds.  Otherwise, it follows from Theorem~\ref{thm:types} and Corollary~\ref{cor:horocyclic} that $H$ contains a translation $g$ say.  Then $\Res_G(g) \le H$, since $H$ is an intersection of open subgroups; by Proposition~\ref{prop:quasiblock_subtree}, it follows that $H$ contains an arc stabiliser of $G$. Hence $H$ is open in $G$, and by Lemma~\ref{lem:arc_stabiliser:1closed}, $H$ is $\propP{}$-closed, so (iii) holds.

Now consider the general case.  If $h \in H$ is a translation, then $L = \grp{h,K}$ acts with translation, where $K$ is a compact open subgroup of $H$.  Now $L$ is a compactly generated open subgroup of $H$; by the previous paragraph, $L$ contains an arc stabiliser of $G$, hence also $H$ contains that arc stabiliser. Thus again, $H$ is open in $G$ and is $\propP{}$-closed so (iii) holds.  From now on we may assume $H$ acts without translation.  Then by Corollary~\ref{cor:horocyclic}, every compactly generated subgroup of $H$ has bounded action.  

Suppose (i) fails. Then $H$ has a compactly generated, noncompact, closed subgroup $M$. 
Let $C$ be a compact generating set for $M$ and let be $U$ a compact open subgroup of $H$. Then $\grp{C, U}$ is an open, compactly generated and noncompact subgroup of $H$. A slight adjustment to the proof of Lemma~\ref{lem:tdlc_union} now gives that $H$ is the union of a directed family $(K_i)_{i \in I}$ of noncompact compactly generated open subgroups. 
We have already established the theorem in the compactly generated case, and so we can apply it to each group $K_i$. Since each $K_i$ cannot satisfy (i) or (iii), the group $K_i$ fixes a unique vertex $v_i$; the uniqueness of $v_i$ ensures that $v_i = v_j$ whenever $K_i \le K_j$, so in fact $H$ fixes a unique vertex, proving (ii).
\end{proof}

\subsection{End stabilisers}\label{sec:local_ends}

In this subsection we define locally invariant ends (Definition~\ref{def:Locally_Invariant_End})
and develop tools for recognising them in local action diagrams (Proposition~\ref{prop:SpottingLocallyInvEnds} and its corollary). Using these, we prove  Theorem~\ref{thm:open_primitive} from the introduction.

Note that property $\propP{}$ is inherited by end stabilisers.  Moreover, the end stabiliser is open if and only if it contains arc stabilisers.

\begin{lem}
Let $T$ be a tree, let $G$ be a $\propP{}$-closed subgroup of $\Aut(T)$ and let $\xi$ be an end of $T$.  Then the end stabiliser $G_{\xi}$ is $\propP{}$-closed.
\end{lem}

\begin{proof}
Let $H = G_{\xi}$.  It is clear that $H$ is a closed subgroup of $G$, and hence of $\Aut(T)$.  Let $a \in AT$; by property $\propP{}$, we can write $G_a = \rist_G(T_a) \times \rist_G(T_{\ol{a}})$.  Given $a \in AT$, then $\xi$ is an end of $T_a$ or of $T_{\ol{a}}$, but not both; say $\xi$ is an end of $T_a$.  Then $\rist_G(T_{\ol{a}})$ fixes $a$ and $\xi$, that is, $ \rist_G(T_{\ol{a}}) \le H_a$, so 
\[
H_a = (\rist_G(T_a) \cap H_a) \times  \rist_G(T_{\ol{a}}) = \rist_H(T_a) \times \rist_H(T_{\ol{a}}).
\]
From this decomposition of the arc stabiliser, we conclude via Theorem~\ref{propertyP_oneclosure} that $H$ is $\propP{}$-closed.
\end{proof} 

\begin{prop}\label{prop:horocyclic_arc_stabiliser}
Let $T$ be a tree, let $G$ be a $\propP{}$-closed subgroup of $\Aut(T)$, and let $\xi$ be an end of $T$.  Then $G_{\xi}$ is open in $G$ if and only if $G_a \le G_{\xi}$ for some $a \in AT$. Moreover, if 
 $G_{\xi}$ is open, the arcs $a \in AT$ such that $G_a \le G_{\xi}$ form the arcs of a subtree $T'$, such that $\xi$ is an end of $T'$.
\end{prop}

\begin{proof}
If $G_{\xi}$ contains an arc stabiliser, then certainly $G_{\xi}$ is open.  So we may suppose for the rest of the proof that $G_{\xi}$ is open in $G$.  In other words, there exist $v_1,\dots,v_n \in VT$ such that $\bigcap^n_{i=1}G_{v_i} \le G_{\xi}$.

The set $X$ of arcs $a \in AT$ such that $G_a \le G_{\xi}$ is closed under edge reversal, so $X$ is the set of arcs of some subgraph $T'$ of $T$.  We claim that $T'$ is nonempty; more precisely we claim that, given a ray $r$ representing $\xi$, then the intersection of $T'$ with $r$ is a subray of $r$.  So let $r$ be a ray representing $\xi$, given by the sequence of arcs $(a_1,a_2,\dots)$, with all arcs pointing towards $\xi$.  We see that there exists $k$ such that the half-tree $T_{a_{k}}$ is disjoint from $\{v_1,\dots,v_n\}$, so $\rist_G(T_{a_{k}})$ fixes $v_1,\dots,v_n$.  By our hypothesis about $G_{\xi}$, it follows that $\rist_G(T_{a_{k}})$ fixes $\xi$.  Now let $k'$ be the least $k' \ge 1$ such that $\rist_G(T_{a_{k'}})$ fixes $\xi$.  Then for $k'' < k'$, $G_{a_{k''}}$ does not fix $\xi$, but for $k'' \ge k'$ we can see that $G_{a_{k''}}$ does fix $\xi$.  Indeed, $G_{a_{k''}} = \rist_G(T_{a_{k''}}) \times \rist_G(T_{\ol{a_{k''}}})$ by property $\propP{}$; we see that $\rist_G(T_{\ol{a_{k''}}})$ fixes $\xi$, since $\xi$ is not an end of $T_{\ol{a_{k''}}}$; and $T_{a_{k''}}$ is contained in $T_{a_{k'}}$ so $\rist_G(T_{a_{k''}}) \le \rist_G(T_{a_{k'}}) \le G_{\xi}$.

From the description of the intersection of $T'$ with every ray representing $\xi$, we conclude that $T'$ is a subtree of $T$ such that $\xi$ is an end of $T'$.
\end{proof}

\begin{defn} \label{def:Locally_Invariant_End} Let $T$ be a tree, let $G$ be a $\propP{}$-closed subgroup of $\Aut(T)$, and let $\xi$ be an end of $T$. We say that $\xi$ is \defbold{locally invariant} if $G_\xi$ is open.
\end{defn}

Using the fact that the end stabiliser contains arc stabilisers for arcs on a ray representing that end, we can describe locally invariant ends in the local action diagram.

\begin{defn}
Let $\Delta = (\Gamma,(X_a),(G(v)))$ be a local action diagram.  A \defbold{locally invariant ray quotient} consists of a sequence of arcs $(b_1,b_2,\dots)$, forming a path in $\Gamma$ (that is, $t(b_i) = o(b_{i+1})$), and points $p_i \in X_{b_i}$ and $q_i \in X_{\ol{b_i}}$ such that
\[
\forall i \ge 2: p_i \neq q_{i-1} \text{ and } G(o(b_i))_{p_i} \ge G(o(b_i))_{q_{i-1}}.
\]
Suppose we have locally invariant ray quotients $r = ((b_i),(p_i),(q_i))$ and $r' = ((b'_i),(p'_i),(q'_i))$.  We say $r$ and $r'$ are \defbold{ray-equivalent} if $b_i = b'_i$, and for each $i \ge 2$, there is $g_i \in G(o(b_i))$ such that $(g_ip_i,g_iq_{i-1}) = (p'_i,q'_{i-1})$; and \defbold{end-equivalent} if $r$ and $r'$ can be made ray-equivalent by shifting the indices (where we delete all entries with negative index).  A \defbold{locally invariant end quotient} is an end-equivalence class of locally invariant ray quotients.
\end{defn}

\begin{prop}\label{prop:SpottingLocallyInvEnds}
Let $T$ be a tree and let $G$ be a $\propP{}$-closed subgroup of $\Aut(T)$.  Then the $G$-orbits of locally invariant ends for $G$ are in one-to-one correspondence with the locally invariant end quotients of the local action diagram of $G$.
\end{prop}

\begin{proof}
Let $\Delta = (\Gamma,(X_a),(G(v)))$ be the local action diagram of $G$, with quotient map $\pi: T \rightarrow \Gamma$, and suppose that $G$ is the universal group given by a $\Delta$-colouring $\mc{L}$.

Let $\xi$ be a locally invariant end of $T$ for the action of $G$.  By Proposition~\ref{prop:horocyclic_arc_stabiliser}, there is some ray $(a_1,a_2,\dots)$ in $T$ representing $\xi$ such that $G_{a_i}$ fixes $\xi$ for all $i \ge 1$.  Letting $b_i = \pi(a_i)$, we have a sequence of arcs $(b_1,b_2,\dots)$, forming a path in $\Gamma$, and also points $p_i = \mc{L}(a_i) \in X_{b_i}$ and $q_i = \mc{L}(\ol{a_i}) \in X_{\ol{b_i}}$.  Given $i \ge 2$, the definition of the colouring ensures that $p_i \neq q_{i-1}$; the fact that $G_{a_{i-1}} = G_{\ol{a_{i-1}}}$ fixes $\xi$ ensures that $G_{\ol{a_{i-1}}}$ fixes the arc $a_i$, so in the local action, $G(o(b_i))_{q_{i-1}}$ fixes $p_i$.  Thus $r = ((b_i),(p_i),(q_i))$ is a locally invariant ray quotient.  Consider now how $r$ depends on the choices we have made (keeping the colouring fixed) and what happens if we move $\xi$ using an element of $G$.  Given $g \in G$, the end $g\xi$ is represented by the ray $(a'_1,a'_2,\dots)$ where $a'_i = ga_i$; we can then perform the same procedure as before to obtain a locally invariant ray quotient $r' = ((b'_i),(p'_i),(q'_i))$.  We have
\[
b'_i = \pi(ga_i) = \pi(a_i) = b_i,
\]
and the sequences $(p_i)$ and $(q_i)$ are changed according to the local action of $g$, that is, 
\[
p'_i = \mc{L}(ga_i) = \sigma_{\mc{L},o(a_i)}(g)p_i; \; \quad q'_i = \mc{L}(g\ol{a_i}) = \sigma_{\mc{L},o(\ol{a_i})}(g)q_i = \sigma_{\mc{L},o(a_{i+1})}(g)q_i.
\]
In particular, note that for $i \ge 2$, $p_i$ and $q_{i-1}$ are moved by the same element $g_i := \sigma_{\mc{L},o(a_i)}(g)$ of $G(o(b(i)))$.  So $r$ and $r'$ are ray-equivalent.  The choice of where to start the ray $(a_1,a_2,\dots)$ was arbitrary, but any two choices would be shift-tail equivalent, so regardless of the choices made, the end-equivalence class of $r'$ is uniquely determined by the orbit $G\xi$.  Thus each $G$-orbit of locally invariant ends produces a unique locally invariant end quotient; write $[r(\xi)]$ for the end-equivalence class of locally invariant ray quotients obtained from $\xi$.

Conversely, suppose that $r = ((b_i),(p_i),(q_i))$ is a locally invariant ray quotient of $\Delta$; we aim to produce a ray $(a_1,a_2,\dots)$ representing an end $\xi$ such that $\pi(a_i) = b_i$ for all $i$ and $r \in [r(\xi)]$.  Start with an arc $a_1$ such that $\mc{L}(a_1) = p_1$.  Once we have chosen the arc $a_i$, let $q'_i = \mc{L}(\ol{a_i})$.  Then $q'_i$ belongs to the same $G(o(b_{i+1}))$-orbit as $q_i$, that is, $q'_i = g_{i+1}q_i$ for some $g_{i+1} \in G(o(b_{i+1}))$.  There is then an arc $a_{i+1}$ with $o(a_{i+1}) = t(a_i)$ such that $\mc{L}(a_{i+1}) = g_{i+1}p_{i+1}$. Since $q_i \neq p_{i+1}$, the arcs $a_{i+1}$ and $\ol{a_i}$ have different colours and are therefore distinct, ensuring that $(a_1,a_2,\dots)$ does not backtrack.  Moreover, $g_{i+1}p_{i+1}$ is in the same $G(o(b_{i+1}))$-orbit as $p_{i+1}$, so $\pi(a_{i+1}) = b_{i+1}$.  Now let $i \ge 2$ and consider the action of the arc stabiliser $G_{a_{i-1}} = G_{\ol{a_{i-1}}}$ on $o\inv(o(a_{i}))$.  We see that $G_{a_{i-1}}$ has local action at $o(a_{i})$ given by $G(o(b_i)_{g_{i}q_{i-1}}) = g_iG(o(b_i))_{q_{i-1}}g\inv_i$.  From the definition of a locally invariant ray quotient, we see that $g_iG(o(b_i))_{q_{i-1}}g\inv_i$ also fixes $g_ip_i$, so $G_{a_{i-1}}$ fixes $a_i$.  Thus we have an ascending sequence of open subgroups
\[
G_{a_1} \le G_{a_2} \le \dots
\]
of $G$, from which we see that $G_{a_1}$ fixes the ray spanned by the arcs $(a_1,a_2,\dots)$.  Thus $G$ has a locally invariant end $\xi$ represented by this ray.  From the colours of the arcs $a_i$ and $\ol{a_i}$, we see that the locally invariant ray quotient obtained by applying the procedure in the previous paragraph is ray-equivalent to $r$; in particular, $r \in [r(\xi)]$.

So we have a surjective map map $G\xi \mapsto [r(\xi)]$ from $G$-orbits of locally invariant ends of the action on $T$ to locally invariant end quotients in $\Delta$.  It remains to check that $G\xi \mapsto [r(\xi)]$ is injective.  Consider two rays $(a_1,a_2,\dots)$ and $(a'_1,a'_2,\dots)$, representing ends fixed respectively by $G_{a_i}$ and by $G_{a'_i}$ for all $i$, giving rise to locally invariant ray quotients $r = ((b_i),(p_i),(q_i))$ and $r' = ((b'_i),(p'_i),(q'_i))$ respectively, such that $p_i = \mc{L}(a_i)$, $q_i = \mc{L}(\ol{a_i})$, and similarly for $(a'_1,a'_2,\dots)$ and $r'$.  Suppose that $r$ and $r'$ are end-equivalent.  After shifting indices on the rays, we may assume that $r$ and $r'$ are ray-equivalent; say $h_i \in G(o(b_i))$ is such that $(p'_i,q'_{i-1}) = (h_ip_i,h_iq_{i-1})$.  Note in particular that $b_i = \pi(a_i)$ and $b'_i = \pi(a'_i)$, so the fact that $r$ and $r'$ are equivalent implies that $\pi(a_i) = \pi(a'_i)$.  We use $r$ and $r'$ to produce rays representing locally invariant ends of $T$ as in the previous paragraph.  Notice that we can use the initial arc $a_1$ for the ray obtained from $r$ and $a'_1$ for the ray obtained from $r'$, and from there we can choose the permutations $g_i \in G(o(b_i))$ along the way to be trivial, with the result that the ray we produce from $r$ is just $(a_1,a_2,\dots)$, and the ray we produce from $r'$ is $(a'_1,a'_2,\dots)$.  In particular, we see that for $i \ge 2$,
\[
\mc{L}(\ol{a'_{i-1}}) = q'_{i-1} = h_iq_{i-1}; \quad \mc{L}(a'_i) = p'_i = h_ip_i.
\]
Given how the colouring is defined, we see that there is $k_i \in G$, with local action $h_i$ at $o(a_i)$, such that $k_i\ol{a_{i-1}} = \ol{a'_{i-1}}$ and $k_ia_i = a'_i$.  By repeatedly using the decomposition of arc stabilisers as a product of rigid stabilisers, there is then $l_j \in G$ such that $l_ja_i = a'_i$ for all $j \le i$, and then since $G$ is closed, there is $l \in G$ such that $la_i = a'_i$ for all $i$.  In particular, the ends represented by $(a_1,a_2,\dots)$ and $(a'_1,a'_2,\dots)$ are in the same $G$-orbit, showing that the map $G\xi \mapsto [r(\xi)]$ is injective.  This completes the proof that $G\xi \mapsto [r(\xi)]$ is a one-to-one correspondence between $G$-orbits of locally invariant ends of the action on $T$ and locally invariant end quotients in $\Delta$.
\end{proof}

\begin{cor}\label{cor:no_local_end}
Let $\Delta = (\Gamma,(X_a),(G(v)))$ be a local action diagram.  Suppose that every infinite path in $\Gamma$ passes through some $v \in VT$ such that $G(v)$ is closed and the point stabilisers of $G(v)$ are pairwise incomparable.  Then $\Univ(\Delta)$ has no locally invariant ends.
\end{cor}

We now have all the ingredients to generalise the theorems \cite[Theorem~A and Theorem~3.9]{CapDeM} of Caprace--De Medts, as stated in the introduction.

\begin{proof}[Proof of Theorem~\ref{thm:open_primitive}]
Note that, since $G$ is nondiscrete, its open subgroup $G^+$ is nontrivial.  Let $\Delta = (\Gamma,(X_a),(G(v)))$ be the local action diagram of $G$.

Assume (i), that is, every proper open subgroup of $G$ has bounded action on $T$ and point stabilisers in $G$  are pairwise incomparable. 
Now $G$ cannot be horocyclic or focal, since such groups fix an end and thus contain comparable point stabilisers. Examining Table~\ref{fig:types} we see that $G$, preserving no proper subtree of $T$, must act with translation.  By Lemma~\ref{lem:general_type}, $G^+$ has unbounded action, so $G = G^+$.  Since $G$ is generated by arc stabilisers, it is parity-preserving; in particular, $G$ acts without inversion, so every bounded subgroup of $G$ fixes a vertex.  In particular, given a proper open subgroup $H$ of $G$, then $H \le G_v$ for some $v \in VT$. On the other hand, given any proper subgroup $H$ of $G$, if $G_v \leq H < G$ for some $v \in VT$, then $H$ is open so $H \leq G_{v'}$ for some $v' \in VT$. Hence $G_v \leq G_{v'}$ which is impossible.
Since $G$ itself does not fix a vertex, the maximal proper open subgroups of $G$ are thus exactly the vertex stabilisers, and all vertex stabilisers are distinct (abstractly) maximal proper subgroups of $G$. This completes the proof that (i) implies (ii). Moreover, $G$ is generated by any two distinct cosets of a vertex stabiliser, so by Lemma~\ref{lem:finite_type}, $G$ has finitely many orbits on $VT \sqcup AT$.  The action of $G$ cannot be horocyclic or focal; since $G$ does not preserve any proper subtree, the only possibility is that the action of $G$ is geometrically dense.  Hence $G$ is simple by Theorem~\ref{thm:Tits}.

Assume (ii).  Since $G_v$ and $G_w$ are distinct maximal subgroups of $G$, we have $G = \grp{G_v,G_w}$.  Corollary~\ref{cor:vertex_group} ensures that $\Gamma$ is a tree, so $|V\Gamma| \ge 2$.  Since $G$ is generated by vertex stabilisers, it acts without inversion.  Let $a$ be the arc $(v,w)$.  From the fact that $G$ is generated by $G_{v}$ and $G_{w}$, we see that $\bigcup_{g \in G}\{ga,g\ol{a}\}$ is the set of arcs of a connected subgraph of $T$; since $G$ does not preserve any proper subtree, in fact $AT = \bigcup_{g \in G}\{ga,g\ol{a}\}$.  In particular, we see that $Gv$ and $Gw$ are the only $G$-orbits on $VT$, so $\{Gv,Gw\}$ is the natural bipartition of $VT$.  The fact that $G_v$ and $G_w$ are both maximal then ensures that $G$ acts primitively on both $Gv$ and $Gw$, so (iii) holds.  We also see by Theorem~\ref{thm:BassSerre} that $G = G_v \ast_{G_{(v,w)}} G_w$.

We now claim that (iii) and (iv) are equivalent.  Let $\Delta = (\Gamma,(X_a),(G(v)))$ be the local action diagram of $(T,G)$.  If either (iii) or (iv) holds, it is clear that $|V\Gamma|=2$ and that $|o\inv(v)|=1$ for all $v \in V\Gamma$; the only possibility is that $\Gamma$ consists of a single undirected edge with two distinct endpoints $v$ and $w$.    In particular, we deduce that $G$ is the universal group $\Univ(G(v),G(w))$ defined in \cite{SmithDuke}.  By \cite[Theorem~26]{SmithDuke}, $G$ acts primitively on both parts of the natural bipartition if and only if $G(v)$ and $G(w)$ are both primitive but not regular; this establishes the equivalence of (iii) and (iv).

Finally, suppose (iii) and (iv) hold and suppose that $H$ is an open subgroup of $G$ with unbounded action; we aim to show $G = H$.  
Recall Theorem~\ref{thm:G_to_U} and note that the groups $G(v)$, being primitive but not regular, have pairwise incomparable point stabilisers. By Corollary~\ref{cor:no_local_end}, for any end $\xi$ of $T$ the end stabiliser $G_\xi$ is not open, and thus $H \neq H_\xi = H \cap G_\xi$, in other words, $H$ does not fix any end of $T$. In particular, $H$ cannot be horocyclic, and hence $H$ acts with translation.   
Let $h \in H$ be a translation. Then by Proposition~\ref{prop:quasiblock_subtree} there is a smallest invariant subtree $T'$ for the action of $O = \Res_G(h)\grp{h}$ on $T$ such that $T'$ is spanned by the axes of the translations of $O$, so in particular $T'$ is leafless; and $T'$ is the unique smallest $G$-quasiblock subtree of $T$ containing the axis of $h$. We claim that $T' = T$. Indeed, for each vertex $v$ of $T'$, because $T'$ is leafless, there are at least two arcs in $o\inv(v)$ that lie in $T'$. By Lemma~\ref{lem:prim_quasiblock}, a quasiblock of $G(v)$ must be either all of $X_v$ or a single point, so it follows then that $T'$ contains all arcs in $o\inv(v)$. Our claim is established. Applying Proposition~\ref{prop:quasiblock_subtree}(ii) we conclude that $O \ge G^+$.

We have shown that $T$ is the smallest invariant subtree of $H$ and that $H \ge G^+$.  Now given $v \in V\Gamma$, since $G(v)$ is primitive but not regular, $G(v)$ is generated by point stabilisers, and hence $H$ contains all the vertex stabilisers of $G$. Since the $G$-orbits on $VT$ are the natural parts of the bipartition, $\Gamma = G \backslash T$ is a single undirected edge with distinct endpoints; by Corollary~\ref{cor:vertex_group}, $G$ is generated by vertex stabilisers, so $G = H$.  Since each group $G(v)$ is transitive, it follows that for any vertex $w$ of $T$, the stabiliser $G_w$ fixes no vertex in $VT\smallsetminus \{w\}$, and hence point stabilisers in $G$ are pairwise incomparable.  Thus (iv) implies (i), and the cycle of implications is complete.
\end{proof}

\section{Invariant structures} \label{InvariantStructures}

Let $G$ be a group acting on a tree $T$.  In this section, we describe how certain kinds of $G$-invariant structure in $T$ can be detected from the local action diagram of the action. 

In \S\ref{sec:invariants} we see that for $\propP{}$-closed groups one can precisely characterise geometrically dense actions according to the existence or absence of natural combinatorial features of local action diagrams. We call these features strongly confluent partial orientations or \scpos, and they arise from natural structures like ends and something we call a cotree. The definitions for these are contained in \S\ref{sec:invariants}.

In \S\ref{sec:Tits_revisited} we revisit Tits' Theorem with our now complete understanding of Tits' independence property $\propP{}$ and geometric density. We also prove Corollaries~\ref{cor:DirectlyReadingSimplicity} and \ref{cor:DirectlyReadingSimplicity:finite}
and Theorem~\ref{thm:UDelta_simple} from the introduction, and determine from the local action diagram precisely when $G^+$ is trivial. We give a useful characterisation of when the local action diagram is irreducible (Proposition~\ref{prop:irreducible_check}): it is irreducible if and only if it is not a focal cycle and has no horocyclic ends, no stray half-trees and no stray leaves  (these terms are defined in Definition~\ref{def:FocalCycleHoroEndStrayHalftreeStrayLeaf}).

We revisit the six types of action on a tree (Fixed vertex, Inversion, Lineal, Horocyclic, Focal and General) in \S\ref{TypesOfActionRevisited} and characterise them according to the local action diagram of the action.

\subsection{Invariant partial orientations}\label{sec:invariants}

In this subsection we define
strongly confluent partial orientations (\scpos) of graphs (Definition~\ref{def:scopo}) and then define \scpos of local action diagrams  (see Definition~\ref{def:scopoDelta}, which is a restatement of Definition~\ref{def:introScopoDelta} from the introduction).
We define cotrees of graphs  (Definition~\ref{def:cotreeGraph}) and cotrees and invariant ends of local action diagrams (Definition~\ref{def:CotreeDelta}). In Definition~\ref{def:scposTypes} we define three types (\ref{scposTypeI})--(\ref{scposTypeIII}) of \scpos and prove in Theorem~\ref{thm:strongly_confluent_attractor} that these types in fact form a classification of  \scpos. From this we deduce that all \scpos arise from cotrees and ends, and this allows us to then prove Theorem~\ref{thm:invariants} from the introduction which relates the invariant subtrees and fixed ends of actions on trees with the \scpos of the associated local action diagram.

Recall that a \defbold{partial orientation} of a graph $\Gamma$ is a subset $O$ of $A\Gamma$ such that for each $a \in A\Gamma$, if $a \in O$ then $\ol{a} \not\in O$. Every $G$-invariant partial orientation of $T$ gives rise to a partial orientation of $G \backslash T$, and conversely.  In particular, the local action diagram provides enough information to give a list of the $G$-invariant partial orientations of $T$.

\begin{lem}\label{lem:invariant_orientation}
Let $\Gamma$ be a graph, let $G \le \Aut(\Gamma)$, let $\Gamma' = G \backslash \Gamma$ and let $\pi = \pi_{(\Gamma,G)}$.  Then a subset $O$ of $A\Gamma$ is a $G$-invariant partial orientation of $\Gamma$ if and only if $O = \pi\inv(O')$ for some partial orientation $O'$ of $\Gamma'$.  Moreover, $O$ is a full orientation of $\Gamma$ if and only if $O'$ is a full orientation of $\Gamma'$.
\end{lem}

\begin{proof}
Suppose $O$ is a $G$-invariant partial orientation of $\Gamma$.  Since $O$ consists of arcs and is $G$-invariant, we have $O = \pi\inv(O')$ for some subset $O'$ of $A\Gamma$.  Suppose $O'$ is not a partial orientation, that is, there is $a \in O'$ such that also $\overline{a} \in O'$.  Let $b \in \pi\inv(a)$.  Then
\[
\pi(\overline{b}) = \overline{\pi(b)} = \overline{a} \in O',
\]
so $\overline{b} \in O$ contradicting the assumption that $O$ is a partial orientation.  Thus every $G$-invariant partial orientation $O$ of $\Gamma$ arises as $\pi\inv(O')$ where $O'$ is a partial orientation of $\Gamma'$.

Conversely, suppose $O'$ is a partial orientation of $\Gamma'$ and let $O = \pi\inv(O')$.  Then certainly $O$ is a $G$-invariant set of arcs of $\Gamma$; moreover, given $a \in O$, then $\pi(a) \in O'$, and hence
\[
\pi(\overline{a}) = \overline{\pi(a)} \not\in O',
\]
so $\overline{a} \not\in O$.  Thus $O$ is a partial orientation of $\Gamma$.

If $O'$ is an orientation of $\Gamma'$, then for all $a \in A\Gamma$, either $\pi(a) \in O'$, in which case $a \in O$, or else $\pi(\overline{a}) = \overline{\pi(a)} \in O'$, in which case $\overline{a} \in O$; thus in this case, $O$ is an orientation of $\Gamma$.  Conversely if $O'$ is not an orientation of $\Gamma'$, say $O' \cap \{a,\overline{a}\} = \emptyset$ for $a \in A\Gamma'$, then for each $b \in \pi\inv(a)$, neither $b$ nor its reverse is contained in $O$, so $O$ is not an orientation of $\Gamma$.
\end{proof}

More interesting is to determine, given a partial orientation $O$ of the local action diagram, what kind of invariant structure is being described in the tree.  Given Theorem~\ref{thm:Tits}, partial orientations of $T$ that determine subtrees or ends are of particular interest.  Our goal in the rest of this subsection is to use partial orientations to characterise the existence of invariant subtrees or ends in terms of the local action diagram.

\begin{defn} \label{def:scopo}
Say that a partial orientation $O$ of a graph $\Gamma$ is \defbold{confluent} if for every vertex $v \in V\Gamma$, we have $|o\inv(v) \cap O| \le 1$; that is, $O$ contains at most one arc originating at each vertex.  A \defbold{strongly confluent partial orientation} (\scpo) of the graph $\Gamma$ is a confluent partial orientation such that in addition, for all $v \in V\Gamma$, we have
\[
|o\inv(v) \cap O| = 1 \quad \Rightarrow \quad \forall a \in o\inv(v), \, |\{a,\ol{a}\} \cap O| = 1.
\]
In words, a \scpo $O$ is a partial orientation that satisfies the following for all vertices $v$: if $O$ includes any arc originating at $v$, then
$O$ contains the reverse of all other arcs originating at $v$.
\end{defn}

Since the quotient map $\pi_{(T,G)}$ is locally surjective and the strongly confluent property is defined using local information, we can easily identify the $G$-invariant \scpos of the tree from the local action diagram.

\begin{lem}\label{lem:strongly_confluent_preimage}
Let $T$ be a tree, let $G \le \Aut(T)$, let $\Delta = (\Gamma,(X_a),(G(v)))$ be the associated local action diagram and let $\pi = \pi_{(T,G)}$.  Let $O$ be a partial orientation of $\Gamma$.  Then the preimage $\pi\inv(O)$ is confluent (resp. strongly confluent) on $T$ if and only if $O$ is confluent (resp. strongly confluent) on $\Gamma$ and $|X_a|=1$ for all $a \in O$. In particular, the $G$-invariant \scpos of $T$ are precisely the preimages $\pi\inv(O)$ of those \scpos $O$ of $\Gamma$ that satisfy $|X_{a}| = 1$ for all $a \in O$.
\end{lem}

\begin{proof}
Let $v \in VT$.  We can calculate the size of $o\inv(v) \cap \pi\inv(O)$ as follows:
\[
|o\inv(v) \cap \pi\inv(O)| =\sum \{|X_a| : a \in o\inv(\pi(v)) \cap O\}.
\]
In particular, we see that $|o\inv(v) \cap \pi\inv(O)| \le 1$ if and only if $|o\inv(\pi(v)) \cap O| \le 1$ and $|X_a|=1$ for all $a \in o\inv(\pi(v)) \cap O$.  This establishes that $\pi\inv(O)$ is confluent if and only if $O$ is confluent and $|X_a|=1$ for all $a \in O$.

Now suppose $O$ and $\pi\inv(O)$ are both confluent and that $|X_a|=1$ for all $a \in O$.  We see that
\[
o\inv(v) \cap \pi\inv(O) \neq \emptyset \Leftrightarrow o\inv(\pi(v)) \cap O \neq \emptyset.
\]
If $o\inv(v) \cap \pi\inv(O)$ is empty, we do not need to check the strong confluence condition at $v$ or $\pi(v)$, so let us assume that $o\inv(v) \cap \pi\inv(O)$ is nonempty.  Then for $O$ to be strongly confluent, it must induce a full orientation of the edges of $\Gamma$ incident with $\pi(v)$.   In fact, since $\pi$ is locally surjective, this is equivalent to the condition that $\pi\inv(O)$ induces a full orientation of the edges of $T$ incident with $v$.  Thus $O$ is strongly confluent if and only if $\pi\inv(O)$ is strongly confluent.

It now follows from Lemma~\ref{lem:invariant_orientation} that
the $G$-invariant \scpos of $T$ are precisely the preimages $\pi\inv(O)$ of those \scpos $O$ of $\Gamma$ that satisfy $|X_{a}| = 1$ for all $a \in O$.
\end{proof}

The above result motivates the following definition, which we restate from the introduction (Definition~\ref{def:introScopoDelta}).

\begin{defn} \label{def:scopoDelta}
Given a local action diagram $\Delta = (\Gamma,(X_a),(G(v)))$, we define a \defbold{(strongly) confluent partial orientation} $O$ of $\Delta$ to be a (strongly) confluent partial orientation of $\Gamma$ such that $|X_a|=1$ for all $a \in O$.  For $G \leq \Aut(T)$ with associated local action diagram $\Delta$, the $G$-invariant \scpos of $T$ are thus precisely preimages $\pi\inv(O)$ of \scpos $O$ of $\Delta$.
\end{defn}

As we shall see, \scpos of a graph only occur in a few special forms, and in the tree case they correspond exactly to subtrees and ends.
These special forms arise from a natural combinatorial feature we call a \defbold{cotree}. We first define the cotree of a graph (Definition~\ref{def:cotreeGraph}), and use this to define the cotree of a local action diagram (Definition~\ref{def:CotreeDelta}).

\begin{defn} \label{def:cotreeGraph}
Given a connected graph $\Gamma$, a directed path $(v_0,\dots,v_n)$ of length $n \ge 2$ is  \defbold{backtracking} if $v_i = v_{i+2}$ for some $0 \leq i \leq n-2$; directed paths of length $n = 0, 1$ are always non-backtracking. For an induced subgraph $\Gamma'$ of $\Gamma$, a \defbold{projecting path} from $v \in V\Gamma$ to $\Gamma'$ is a directed non-backtracking path $(v_0,\dots,v_n)$ of some finite length $n \ge 0$, such that $v = v_0$ and such that $v_n \in V\Gamma'$ and $v_i \not\in V\Gamma'$ for $i < n$.  We say a nonempty induced subgraph $\Gamma'$ is a \defbold{cotree} of $\Gamma$ if for every $v \in V\Gamma \smallsetminus V\Gamma'$ there is exactly one projecting path to $\Gamma'$, including the choice of arcs (that is, in the projecting path $(v_0,\dots,v_n)$ we require there to be only one arc of $\Gamma$ from $v_i$ to $v_{i+1}$ for $0 \le i < n$).  Note that, because there can be only one such path, a cotree of a connected graph is connected.

A \defbold{cycle graph} is a finite connected graph in which all vertices have degree $2$.  As conventions can differ here, we emphasise that the cycle graph of order $1$ consists of a vertex with a loop, but edge-reversal is nontrivial on the loop; the cycle graph of order $2$ consists of two vertices with two edges between them. A \defbold{cyclic orientation} of a cycle graph $\Gamma$ is an orientation that includes exactly one element of $o\inv(v)$ for each $v \in V\Gamma$.  One sees that each cycle graph admits two cyclic orientations, both of which are \scpos.  We say a graph is \defbold{acyclic} if it has no cycle subgraphs.  In particular, trees are precisely the orientable acyclic connected graphs.
\end{defn}

In any connected graph other than a tree, there is a \defbold{smallest cotree} described as $\Gamma'$ in the following lemma (which could be the whole graph), and cotrees are connected induced subgraphs that contain the smallest cotree (which is clearly unique). 

\begin{lem}\label{lem:minimal_cotree} 
Let $\Gamma$ be a connected graph.  Let $F$ be the set of vertices $v \in V\Gamma$ such that $v$ belongs to a cycle subgraph or $o\inv(v)$ contains a loop.  Let $\Gamma'$ be the subgraph induced by the union of all simple paths (including empty paths) with endpoints in $F$.  Then $\Gamma'$ is empty if and only if $\Gamma$ is a tree; otherwise, $\Gamma'$ is connected.  In either case, given a nonempty connected induced subgraph $\Gamma''$ of $\Gamma$, then $\Gamma''$ is a cotree if and only if $\Gamma' \subseteq \Gamma''$.
\end{lem}

\begin{proof}
We see that $\Gamma'$ is empty if and only if $\Gamma$ has no loops and no cycle subgraphs.  Since $\Gamma$ is connected, this is exactly the case when $\Gamma$ is a tree.  Now suppose $\Gamma'$ is nonempty.  Given a simple path $p = (v_0,\dots,v_n)$ in $\Gamma$ such that $v_0,v_n \in V\Gamma'$, we see that $v_0$ and $v_n$ lie on simple paths with endpoints $a_0,z_0 \in F$, respectively $a_n,z_n \in F$, and then $p$ is contained in some simple path whose endpoints are a subset of $\{a_0,z_0,a_n,z_n\}$, so $p$ is contained in $\Gamma'$.  In particular, $\Gamma'$ is connected.

Let $\Gamma''$ be a nonempty connected induced subgraph of $\Gamma$. Suppose first that $\Gamma''$ is a cotree of $\Gamma$. From the uniqueness of projecting paths, we see that $\Gamma''$ must contain all loops and cycle subgraphs of $\Gamma$, so $F \subseteq V\Gamma''$.  Moreover, given $v,w \in V\Gamma''$, a simple path $p$ from $v$ to $w$, and a vertex $x$ on this path, then $x$ cannot have a unique projecting path unless $x \in V\Gamma''$.  Since $\Gamma''$ is an induced subgraph, we conclude that $\Gamma''$ contains $\Gamma'$.

Conversely, let $\Gamma''$ be a nonempty connected induced subgraph of $\Gamma$ that contains $\Gamma'$, and let $v \in V\Gamma \smallsetminus V\Gamma''$.  Then $v$ certainly has a projecting path to $\Gamma''$; suppose for a contradiction that it has two distinct projecting paths $p = (v_0,\dots,v_m)$ and $p' = (v'_0,\dots,v'_n)$, and suppose that $m$ is the minimum length for which this occurs.  Since $\Gamma''$ contains all cycle subgraphs and $m$ is minimal, we see that there are no multiple edges between $v_i$ and $v_{i+1}$ for $0 \le i < m$ or between $v'_i$ and $v'_{i+1}$ for $0 \le i < n$.   We also see that $p$ is simple and there are no vertices other than endpoints shared by $p$ and $p'$.  We can thus replace $p'$ with a simple path that is still distinct from $p$.  Now we obtain a cycle graph $C$ that is the union of the undirected paths specified by $p$ and $p'$, with a shortest path in $\Gamma''$ from $v_m$ to $v'_n$.  But then $C \subseteq \Gamma' \subseteq \Gamma''$, a contradiction.  From this contradiction we conclude that $\Gamma''$ is a cotree.
\end{proof}

In particular, any nonempty proper subtree of a tree is a cotree.

\begin{defn} \label{def:scposTypes}
We now define three kinds of \scpos of the connected graph $\Gamma$. Each type in  the definition is a different type of \scpo, and each \scpo arises from a natural combinatorial feature of the graph $\Gamma$. To highlight this, for each type we denote the combinatorial feature by $z$ and the \scpo associated to $z$ by $O_{z}$ or $O^+_{z}$.
\begin{enumerate}[(a)]
\item \label{scposTypeI}
	Given a cotree $z$ of $\Gamma$, the associated partial orientation $O_{z}$ consists of all arcs $a$ such that $o(a) \not\in Vz$ and $a$ lies on the projecting path from $o(a)$ to $z$.  (In particular, $O_\Gamma = \emptyset$.)
\item \label{scposTypeII}
	Suppose $z$ is a cycle graph equipped with one of its two cyclic orientations, such that $z$ occurs as a cotree of $\Gamma$.  Then from part (\ref{scposTypeI}) we have a partial orientation $O_{z}$. For this type we define 
the associated partial orientation $O^+_{z}$ to be  the union of $O_{z}$ with the cyclic orientation of $z$.
\item \label{scposTypeIII}
	If $\Gamma$ is a tree and we are given an end $z$ of $\Gamma$, then for each arc $a \in A\Gamma$, exactly one of $a$ and $\overline{a}$ is directed towards $z$, 
    that is, it belongs to a ray in the equivalence class $z$.  The set $O_{z}$ is then defined to be those arcs in $A\Gamma$ that are directed towards $z$, and is thus an orientation of $\Gamma$. For this type $O_z$ is the associated partial orientation. 
\end{enumerate}
Here are some observations on these partial orientations:
\begin{enumerate}[(i)]
\item \label{scposTypeIObs}
	The associated partial orientations of type (\ref{scposTypeI}), (\ref{scposTypeII}) and (\ref{scposTypeIII}) are all strongly confluent.
\item \label{scposTypeIIObs}
	The partial orientations of types (\ref{scposTypeII}) and (\ref{scposTypeIII}) are in fact full orientations of $\Gamma$; a partial orientation of type (\ref{scposTypeI}) is full if and only if $z$ consists of a single vertex with no edges.
\item \label{scposTypeIIIObs}
	If $z$ is a cotree or end of $\Gamma$ and $G$ is a group of automorphisms of $\Gamma$, then $z$ is $G$-invariant if and only if $O_z$ is $G$-invariant.
\end{enumerate}
\end{defn}

\begin{defn} \label{def:CotreeDelta}
In light of Lemma~\ref{lem:strongly_confluent_preimage}, given a local action diagram $\Delta = (\Gamma,(X_a),(G(v)))$ we define a \defbold{cotree of $\Delta$} to be a cotree $\Gamma'$ of $\Gamma$ such that $|X_a|=1$ for all $a \in O_{\Gamma'}$, and an \defbold{invariant end of $\Delta$}, in the case that $\Gamma$ is a tree, to be an end $\xi$ of $\Gamma$ such that $|X_a|=1$ for all $a \in O_{\xi}$.  In each case we ensure that $O_{\Gamma'}$, respectively $O_\xi$ is a \scpo of $\Delta$.
\end{defn}

\begin{rem} \label{rem:SmallestCotreeOfDelta}
Suppose $\Delta = (\Gamma, (X_a), (G(v)))$ is the local action diagram for $G \leq \Aut(T)$. If there is a smallest invariant subtree $T'$ such that $|VT'| \geq 3$ then the unique smallest cotree $\Gamma'$ of the graph $\Gamma$ described in Lemma~\ref{lem:minimal_cotree}  is in fact a cotree of the local action diagram $\Delta$  and moreover $\pi(T') = \Gamma'$. In this situation we call $\Gamma'$ the (unique) \defbold{smallest cotree of $\Delta$}.

Indeed, given such a subtree $T'$ we have by Lemma~\ref{lem:minimal_invariant} (\ref{lem:minimal_invariant:translation}) that $G$ acts with translation. Now translations create loops or cycle graphs in the quotient graph $\Gamma$, so $\Gamma$ is not a tree. Thus, there exists a smallest cotree $\Gamma'$ of $\Gamma$ whose structure is given by Lemma~\ref{lem:minimal_cotree}. Furthermore, $T'$ is the union of the axes of translation of $G$ and so we see from the description of the structure of $\Gamma'$ that $\pi(T') = \Gamma'$.
Consider an arc $a$ in the (graph) \scpo $O_{\Gamma'}$ (which is of type (\ref{scposTypeI})) and fix some representative vertex $w^* \in VT$ such that $\pi(w^*) = o(a) \not \in V\Gamma'$. 
So $w^*$ does not lie in $T'$ and since $T'$ is a subtree there is a unique arc in $o\inv(w^*)$ pointing towards $T'$. Since $G$ leaves $T'$ invariant, this unique arc cannot be mapped to another arc in $o\inv(w^*)$ by the action of $G$. Hence (see Definition~\ref{def:AssocLocalActionDiagram})
$X_a = \{b \in o\inv(w^*) \mid \pi(b) = a\}$ has cardinality one. Thus $\pi(T')$ is a cotree of $\Delta$ (not just of $\Gamma$).
\end{rem}

Our next goal is to show that the types (\ref{scposTypeI})--(\ref{scposTypeIII}) actually account for all \scpos of connected graphs, so in particular, in the case of trees they correspond to subtrees and ends.

A confluent partial orientation $O$ of a graph $\Gamma$ defines a map $f_O$ on $V\Gamma$, as follows: if $o\inv(v) \cap O = \{a\}$ we set $f_O(v) = t(a)$, and if $o\inv(v) \cap O = \emptyset$ we set $f_O(v) = v$.  The \defbold{attractor} $K(O)$ of $O$ is then defined to consist of the following:
\begin{enumerate}[(i)]
\item All vertices of $\Gamma$ belonging to periodic orbits of $f_O$;
\item All ends of $\Gamma$ defined by an aperiodic orbit $(v,f_O(v),f^2_O(v),\dots)$.
\end{enumerate}
Thus each $v \in V\Gamma$ defines a nonempty finite subset $z_O(v)$ of the attractor: if $(v,f_O(v),f^2_O(v),\dots)$ is eventually periodic then $z_O(v)$ is the associated periodic orbit, whereas if $(v,f_O(v),f^2_O(v),\dots)$ is aperiodic then $z_O(v)$ is the associated end.  We then have $K(O) = \bigcup_{v \in V\Gamma}z_O(v)$.

Attractors of \scpos are of a special form, which allows us to recognise the types (\ref{scposTypeI})--(\ref{scposTypeIII}).

\begin{thm}\label{thm:strongly_confluent_attractor}
Let $\Gamma$ be a connected graph, let $O$ be a \scpo on $\Gamma$ and let $K$ be the attractor of $O$.  Then exactly one of the following occurs:
\begin{enumerate}[(a)]
\item \label{thm:strongly_confluent_attractor:item:cotree}
There is a cotree $\Gamma'$ of $\Gamma$ such that $V\Gamma' = K$ and $O = O_{\Gamma'}$.
In this case, $V\Gamma' = \{v \in V\Gamma : o\inv(v) \cap O = \emptyset\}$. 
\item \label{thm:strongly_confluent_attractor:item:cycle}
	There is a cotree $\Gamma'$ of $\Gamma$ forming a cycle graph such that $V\Gamma' = K$ and $O = O^+_{\Gamma'}$ for one of the cyclic orientations of $\Gamma'$. 
\item \label{thm:strongly_confluent_attractor:item:end}
There is an end $\xi$ of $\Gamma$ such that $K = \{\xi\}$, $\Gamma$ is a tree and $O = O_{\xi}$.
\end{enumerate}
\end{thm}

\begin{rem} \label{rem:ScopoThmWorksForLADS}
Notice that the statement of Theorem~\ref{thm:strongly_confluent_attractor} also holds if the graph $\Gamma$ is replaced with a local action diagram $\Delta = (\Gamma, (X_a), (G(v)))$ and the end $\xi$ in case (\ref{thm:strongly_confluent_attractor:item:end}) is replaced with an invariant end $\xi$.
Indeed, if $O$ is a \scpo on $\Delta$ then $|X_a| = 1$ for all $a \in O$. So, in cases (\ref{thm:strongly_confluent_attractor:item:cotree}) and (\ref{thm:strongly_confluent_attractor:item:cycle}) of the theorem, the cotree $\Gamma'$ of $\Gamma$ has the additional structure required of a cotree of $\Delta$, and in case (\ref{thm:strongly_confluent_attractor:item:end}) the end $\xi$ of $\Gamma$ has the additional structure of an invariant end of $\Delta$.
\end{rem}

Most of the proof of Theorem~\ref{thm:strongly_confluent_attractor} will consist of the next two lemmas.

\begin{lem}\label{lem:subgraph_confluent}
Let $\Gamma$ be a graph and let $O$ be a \scpo of $\Gamma$.
\begin{enumerate}[(i)]
\item If $\Gamma'$ is a subgraph of $\Gamma$, then $O \cap A\Gamma'$ is a \scpo of $\Gamma'$.
\item If $\Gamma$ is a cycle graph, then $O$ is either empty or it is one of the two  cyclic orientations of $\Gamma$.
\end{enumerate}
\end{lem}

\begin{proof}
(i) It is clear that any subset of $O$ is a confluent partial orientation.  We also see that
\[
o\inv_{\Gamma'}(v) \cap O = \{a\} \Rightarrow o\inv_{\Gamma}(v) \cap O = \{a\} \Rightarrow  t\inv_{\Gamma}(v) \subseteq O \cup \{\ol{a}\} \Rightarrow t\inv_{\Gamma'}(v) \subseteq O \cup \{\ol{a}\},
\]
which ensures that $O \cap A\Gamma'$ is strongly confluent on $\Gamma'$.

(ii) It is easy to see that the two cyclic orientations of $\Gamma$ are strongly confluent.
Conversely, suppose that $O$ is nonempty, that is, there exists $a \in O$.  Then the strong confluence condition means that we must also have $s(a) \in O$, where $s(a)$ is the unique element of $t\inv(o(a)) \smallsetminus \{\ol{a}\}$.  We then have $s^n(a) \in O$ for all $n \ge 0$, and since $\Gamma$ is finite, eventually the sequence repeats; without loss of generality, $s^k(a) = a$.  The sequence of arcs $a,s^{k-1}(a),s^{k-2}(a),\dots,s(a)$ then defines a directed path from $o(a)$ to $o(a)$ without backtracking; since $\Gamma$ is a cycle graph, we conclude that $O = \{a,s^{k-1}(a),s^{k-2}(a),\dots,s(a)\}$ and that $O$ is a cyclic orientation of $\Gamma$.
\end{proof}

\begin{lem}\label{lem:singleton_attractor}
Let $\Gamma$ be a connected graph and let $O$ be a \scpo of $\Gamma$.  
\begin{enumerate}[(i)]
\item The attractor $K(O)$ contains the vertices of every cycle subgraph of $\Gamma$ and the endpoint of every nonorientable loop of $\Gamma$.
\item For every simple path $(v_0,\dots,v_n)$ in $\Gamma$ such that
\[
O \cap \bigcup^n_{i=0} o\inv(v_i) \neq \emptyset,
\]
we have $f_O(v_0) = v_1$ or $f_O(v_n) = v_{n-1}$ (or both).  Consequently, if $(v_0,\dots,v_n)$ is a simple path such that $\{v_0,v_n\} \subseteq K$, then also $v_1,\dots,v_{n-1} \in K$.
\item Suppose there exist $v,w \in V\Gamma$ such that $z_O(v) \neq z_O(w)$.  Then $K(O)$ consists exactly of those $v \in V\Gamma$ such that $o\inv(v) \cap O = \emptyset$.
\end{enumerate}
\end{lem}

\begin{proof}
(i) Let $v$ be the endpoint of a loop, that is, there is $a \in A\Gamma$ such that $o(a) = t(a) = v$.  If no arc of $O$ originates at $v$, then $f_O(v) = v$.  Otherwise we see that $O$ must contain one of $a$ or $\ol{a}$; we thus end up with an arc in $O$ originating at $v$ that also terminates at $v$, so $f_O(v) = v$.  In either case, we see that $v \in K(O)$.

Let $\Gamma'$ be a cycle subgraph of $\Gamma$ of order $\ge 2$.  By Lemma~\ref{lem:subgraph_confluent}, the restriction $O' := O \cap A\Gamma'$ is either empty or one of the two cyclic orientations of $\Gamma'$.  If $O'$ is a cyclic orientation of $\Gamma'$, we immediately see that $\Gamma'$ is a periodic orbit of $f_O$, so $V\Gamma' \subseteq K(O)$.  If instead $O'$ is empty, then for each $v \in V\Gamma'$, $O$ is missing at least two of the arcs of $\Gamma$ that terminate at $v$, and hence $O$ is disjoint from $o\inv_{\Gamma}(v)$; this means $f_O(v) = v$, so $v \in K(O)$.

(ii)
We consider a simple path $(v_0,\dots,v_n)$ and $a\in o\inv(v_i) \cap O$ for some $0 \le i \le n$.  If $a$ is an arc from $v_0$ to $v_1$, we have $f_O(v_0) = v_1$, and similarly if $a$ is an arc from $v_n$ to $v_{n-1}$ we have $f_O(v_{n-1})$; so let us assume neither is the case.  Then $v_i$ has at least one neighbour other than $t(a)$ in the path; by symmetry we may assume this neighbour is $v_{i-1}$.  Then by strong confluence, $O$ must include all arcs from $v_{i-1}$ to $v_i$, then all arcs from $v_{i-2}$ to $v_{i-1}$, and so on, up to the arcs from $v_0$ to $v_1$; hence $f_O(v_0) = v_1$.

We now claim that if $(v_0,\dots,v_n)$ is a simple path whose endpoints are in $K$, then all vertices of the path are in $K$.  Suppose not, and suppose $(v_0,\dots,v_n)$ is a counterexample of minimal length.  Then $n \ge 2$ and $K \cap \{v_0,\dots,v_n\} = \{v_0,v_n\}$.  Since $f_O$ cannot leave its attractor, we have $f_O(v_0) \neq v_1$ and $f_O(v_n) \neq v_{n-1}$.  But then by the previous paragraph, $O$ is disjoint from $o\inv(v_i)$ for all $i$, which means that $f_O(v_i) = v_i$ and hence $v_i \in K$ for all $i$, a contradiction.

(iii)
Let $v_i = f^i_O(v)$ and let $w_i = f^i_O(w)$.  Choose $i,j \in \mathbb{N} \times \mathbb{N}$ in such a way that the distance $d(v_i,w_j)$ is minimised; note that $v_i \neq w_j$, so $d(v_i,w_j) > 0$.  Let $p$ be a path of minimal length from $v_i$ to $w_j$.  Then by the minimality of $d(v_i,w_j)$, neither $f_O(v_i)$ nor $f_O(w_j)$ is an interior vertex of the path.  By part (ii) it follows that $O$ is disjoint from $o\inv(v_i)$ and $o\inv(w_j)$.  In particular, $v_i$ and $w_j$ are both fixed by $f_O$; hence $z_O(v) = \{v_i\}$ and $z_O(w) = \{w_j\}$.  We then see by the same argument that given any $w' \in V\Gamma$, then $z_O(w') = \{w''\}$ for some $w''$ such that $o\inv(w'') \cap O = \emptyset$.  Conversely, if $w''$ is any vertex of $\Gamma$ such that $o\inv(w'') \cap O = \emptyset$, then $w''$ is fixed by $f_O$, so $w'' \in K(O)$. 
\end{proof}

\begin{proof}[Proof of Theorem~\ref{thm:strongly_confluent_attractor}]
Suppose that $K$ contains a vertex of $\Gamma$.  Then by Lemma~\ref{lem:singleton_attractor}(iii), we see that $K$ consists solely of vertices of $\Gamma$, and by parts (i) and (ii) of Lemma~\ref{lem:singleton_attractor}, together with Lemma~\ref{lem:minimal_cotree}, we see that $K$ is the set of vertices of a cotree $\Gamma'$ of $\Gamma$.  There are then two possibilities.  If $K$ consists of those $v \in V\Gamma$ such that $o\inv(v) \cap O = \emptyset$, then case (a) holds.  Otherwise, by Lemma~\ref{lem:singleton_attractor}(iii), $\Gamma'$ is a cycle graph and $O \cap A\Gamma'$ is a cyclic orientation of $\Gamma'$, and we see that case (\ref{thm:strongly_confluent_attractor:item:cycle}) holds.

The remaining possibility is that $K$ does not contain any vertex of $\Gamma$.  Then by Lemma~\ref{lem:singleton_attractor}(i), $\Gamma$ is a tree; by Lemma~\ref{lem:singleton_attractor}(iii), we have $K = \{\xi\}$ for a unique end $\xi$ of $\Gamma$.  It is then clear that case (\ref{thm:strongly_confluent_attractor:item:end}) holds.
\end{proof}

Theorem~\ref{thm:invariants} is a summary of the precise correspondence described in this chapter. For completeness, we give an explicit proof of the theorem below.

\begin{proof}[Proof of Theorem~\ref{thm:invariants}] Let $G$ be a group acting on a tree $T$, with associated local action diagram $\Delta = (\Gamma,(X_a),(G(v)))$. 
Let $z$ be an end of $T$ and let $z'$ be a nonempty proper subtree of $T$. Note that $z'$ is a cotree of $T$. Thus, using Definition~\ref{def:scposTypes}, we have \scpos $O_z$ of type (\ref{scposTypeIII}) and $O_{z'}$ of type (\ref{scposTypeI}).

If $z$ (resp.~$z'$) is $G$-invariant, then $O_z$ (resp.~$O_{z'}$) is $G$-invariant by observation (\ref{scposTypeIIIObs}) of Definition~\ref{def:scposTypes}.
As noted above, the $G$-invariant \scpos of $T$ are precisely preimages under $\pi$ of \scpos of the local action diagram $\Delta$. Clearly, if $z' = T$ then $O_{z'}$ is empty, and is the preimage under $\pi$  of the empty \scpo of $\Delta$.

On the other hand, given a \scpo $O$ of the local action diagram $\Delta$, the preimage $\pi\inv(O)$ is a $G$-invariant \scpo of $T$.  Because $T$ is a tree, Theorem~\ref{thm:strongly_confluent_attractor} guarantees that the $G$-invariant \scpos $\pi\inv(O)$ of $T$ are of type (\ref{scposTypeI}) or (\ref{scposTypeIII}). That is, $\pi\inv(O) = O_z$, where $z$ is either a cotree or an end of $T$. A cotree is a connected induced subgraph, so $z$ is either a subtree or an end of $T$. Since $O_z$ is $G$-invariant, by observation (\ref{scposTypeIIIObs}) of Definition~\ref{def:scposTypes} we have that $z$ is $G$-invariant.
\end{proof}

\subsection{Tits' theorem revisited}\label{sec:Tits_revisited}
\label{Section:TitsRevisited}
 
Theorem~\ref{thm:strongly_confluent_attractor} immediately provides a characterisation of geometrically dense actions in terms of the local action diagram $(\Gamma,(X_a),(G(v)))$: specifically, it should be \defbold{irreducible}, meaning that the only \scpo of $\Delta$ is the empty one.  (In fact, we only need to know $\Gamma$ and the colour sets $X_a$; the additional information provided by the groups $G(v)$ is not needed.)
In this subsection we explore the consequences of this observation. In Definition~\ref{def:FocalCycleHoroEndStrayHalftreeStrayLeaf} we define the four combinatorial features of a local action diagram (focal cycle, horocyclic end, stray half-tree and stray leaf) whose existence precisely characterises non-irreducible local action diagrams (see Proposition~\ref{prop:irreducible_check}). Together, the contents of this subsection give our new `local' version of Tits' theorem: Corollary~\ref{cor:DirectlyReadingSimplicity}. Tits' theorem leaves open the possibility that $G^+$ might be trivial, and we determine from the local action diagram precisely when $G^+$ is trivial in Lemma~\ref{lem:free_diagram}.

\begin{cor}\label{cor:geometrically_dense}
Let $T$ be a tree, let $G \le \Aut(T)$ and let $\Delta = (\Gamma,(X_a),(G(v)))$ be the associated local action diagram.  Then $G$ is geometrically dense if and only if $\Delta$ is irreducible.
\end{cor}

\begin{proof}
Suppose $G$ is not geometrically dense. Then either $G$ fixes an end or it leaves a proper nonempty subtree of $T$ invariant. Recalling that such a subtree of $T$ is a cotree, we see that in either case, using parts (\ref{scposTypeI}), (\ref{scposTypeIII}) and (\ref{scposTypeIIIObs}) of Definition~\ref{def:scposTypes}, there is a \scpo $O_{z}$ of $T$ that is $G$-invariant. Hence by Lemma~\ref{lem:strongly_confluent_preimage}, $\Delta$ is not irreducible.

On the other hand, if $\Delta$ is not irreducible, then $\Delta$ has a nonempty \scpo $O$. By Lemma~\ref{lem:strongly_confluent_preimage}, $\pi\inv(O)$ is a $G$-invariant \scpo of $T$. By Theorem~\ref{thm:strongly_confluent_attractor}, there is an end or cotree $z$ of $T$ such that $\pi\inv(O) = O_{z}$, and by observation (\ref{scposTypeIIIObs}) of Definition~\ref{def:scposTypes} we have that $z$ is $G$-invariant. Thus, $G$ is not geometrically dense on $T$.
\end{proof}

We can now derive Theorem~\ref{thm:UDelta_simple}.

\begin{proof}[Proof of Theorem~\ref{thm:UDelta_simple}]
Let us first consider the case that the action of $G$ on $T$ is not parity-preserving.  In that case, the parity-preserving subgroup $H$ of $G$ is an open normal subgroup of index $2$.  If $G$ is cyclic of order $2$, we see that $G$ is generated by a single edge inversion and hence the stabiliser of any vertex is trivial, so (i) is false; otherwise $H$ is nontrivial, so (i) is false because $G$ is not simple.  Given a parity-reversing $g \in G$, then $g$ is parity-reversing on any $G$-invariant subtree $T'$ of $T$, meaning that the quotient graph of the action is not bipartite, and in particular is not a tree; thus (ii) is also false. Hence, for the rest of the proof we may assume that $G$ is parity-preserving, and in particular inversion-free.

Suppose that (i) holds. Since $G$ is simple, a lineal or focal action of $G$ is ruled out by Lemma~\ref{lem:Busemann:bis}, and a bounded or horocyclic action is ruled out since $G$ acts with translation.  Thus the action is of general type, so there is a unique smallest $G$-invariant subtree $T'$, which has infinite diameter, and the action of $G$ on $T'$ is geometrically dense.  In particular, $G$ acts nontrivially, hence faithfully on $T'$.  The action of $G$ on $T'$ is $\propP{}$-closed by Lemma~\ref{lem:Pclosed_subtree}.  From now on we focus on the action $(T',G)$ and define subgroups of $G$ relative to this action. Let $\Delta = (\Gamma,(X_a),(G(v)))$ be the local action diagram of the faithful and $\propP{}$-closed action $(T',G)$, then $G = \Univ(\Delta)$ by Theorem~\ref{thm:G_to_U}.  Since the action is geometrically dense, $\Delta$ is irreducible by Corollary~\ref{cor:geometrically_dense}.  Our hypotheses ensure that all arc stabilisers in $G$ are nontrivial.  In particular, $G^+$ is nontrivial; since $G$ is simple it follows that $G = G^+$. For each $v \in VT'$, the vertex stabiliser $G_v$ has closed action $G(\pi_{(T',G)}(v))$ on $o_{T'}\inv(v)$, by Lemma~\ref{lem:closed_local_actions} and the fact that $G$ is $\propP{}$-closed.  By Theorem~\ref{thm:plus_subaction_diagram} and the description of $\Delta^+$ in Definition~\ref{defn:plus_subaction_diagram}, it follows from $G = G^+$ that $G(v)$ is generated by point stabilisers. We also see that $G$ is generated by vertex stabilisers, so $\Gamma$ is a tree by Corollary~\ref{cor:vertex_group}.  If all local actions $G(v)$ in $\Delta$ were trivial, then any vertex stabiliser in $(T', G)$, and thus in $(T,G)$, would be trivial --- a contradiction.  Thus (i) implies (ii).

For the remainder of the proof we suppose that (ii) holds.  The action of $G$ on $T'$ is faithful by assumption, and closed because $(T, G)$ is strongly closed; $(T', G)$ is $\propP{}$-closed by Lemma~\ref{lem:Pclosed_subtree}; and it is geometrically dense by Corollary~\ref{cor:geometrically_dense}.  The fact that $G$ has a geometrically dense action on an infinite subtree ensures that the original action is not bounded or horocyclic, so $G$  acts on $T$ with translation; moreover (as noted in Section~\ref{sec:action_types}) this geometrically dense action on $T'$ ensures that $T'$ must be leafless, and hence unbounded.
Let $v \in V\Gamma$ be such that $G(v) \neq \triv$.  Since we are assuming that $G(v)$ is generated by point stabilisers, there is a nontrivial point stabiliser in $G(v)$, which implies that there is a nontrivial arc stabiliser $G_a$ for some $a \in AT'$.  By property~$\propP{}$, the pointwise stabiliser of one of the half-trees of $T'$ defined by $a$, say $T'_a$, acts nontrivially on $T'$.  
Since $G$ has geometrically dense action on $T'$ we know by Lemma~\ref{lem:geom_dense} that for every $b \in AT'$, the half-tree $T'_b$ is contained in a $G$-translate of $T'_a$, so in fact the pointwise stabiliser of $T'_b$ also acts nontrivially on $T'$.  Applying property~$\propP{}$ again, this time to the action on $T$, we find that the pointwise stabiliser of $T_b$ is nontrivial; in turn, since $T'$ is unbounded, any finite set of vertices of $T$ is contained in $T_b$ for some $b \in AT'$.  Hence, there is no finite set of vertices of $T$ whose pointwise stabiliser is trivial.

Define $G^+$ with respect to the closed and faithful action $(T',G)$.  We have seen that arc stabilisers in $(T', G)$ are nontrivial, so $G^+$ is nontrivial and $G^+$ is simple by Theorem~\ref{thm:Tits}.  
The action $(T', G)$ is $\propP{}$-closed, so $G = \Univ(\Delta)$.
The fact that $G(v)$ is closed and generated by point stabilisers for every $v \in V\Gamma$ implies, by Theorem~\ref{thm:plus_subaction_diagram}, that $G^+$ and $G$ have precisely the same local action at $v$, for all vertices $v \in VT'$. Thus, for any arc $a \in VT'$ lying in $o\inv(v)$ 
we have that for all $g \in G_v$ there is $g' \in (G^+)_v$ such that $ga = g'a$; hence $g \in (G^+)_v$.
From this it follows that $G_v = (G^+)_v$. Therefore
\[
R := \langle G_v \mid v \in VT' \rangle = \langle (G^+)_v \mid v \in VT' \rangle \leq G^+ \leq R,
\]
so in particular $R$ is simple.
Now $(T', G)$ acts without inversion because $(T, G)$ is inversion-free, and so by Theorem~\ref{thm:BassSerreSimplicity} we have $G = R$ because $\Gamma = G \backslash T'$ is a tree. Thus $G$ is simple.

Finally we note that in (ii), since $G$ is a $\propP{}$-closed subgroup of $\Aut(T')$ with local action diagram $\Delta$, we have $G = \Univ(\Delta)$ by Theorem~\ref{thm:G_to_U}.
\end{proof}

The next proposition gives an alternative characterisation of when the local action diagram $\Delta = (\Gamma,(X_a),(G(v)))$ is irreducible. 

\begin{defn} \label{def:FocalCycleHoroEndStrayHalftreeStrayLeaf}
We say $\Delta$ is a \defbold{focal cycle} if $\Gamma$ is a cycle graph admitting a cyclic orientation $O$, such that $|X_a|=1$ for all $a \in O$.  For each end $\xi$ of $\Gamma$, say $\xi$ is a \defbold{horocyclic end} of $\Delta$ if $\Gamma$ is a tree and $|X_a|=1$ for every $a \in O_{\xi}$.  Recall that for $a \in A\Gamma$ the graph $\Gamma \smallsetminus \{a\}$ is obtained from $\Gamma$ by removing arcs $a, \overline{a}$. If $\Gamma \smallsetminus \{a\}$ is not connected, and $\Gamma_a$ is the connected component of $\Gamma \smallsetminus \{a\}$ containing $t(a)$, we say that $\Gamma_a$ is a \defbold{stray half-tree} of $\Delta$ if $\Gamma_a$ is a tree containing no leaves of $\Gamma$, and in $\Gamma_a$, we have $|X_{b}|=1$ for all arcs $b$ oriented towards $t(a)$. A \defbold{stray leaf} of $\Delta$ is a leaf $v$ such that $|X_v|=1$.
\end{defn}

\begin{prop}\label{prop:irreducible_check}
Let $\Delta = (\Gamma,(X_a),(G(v)))$ be a local action diagram.  Then $\Delta$ is irreducible if and only if $\Delta$ is not a focal cycle and has no horocyclic ends, no stray half-trees and no stray leaves.
\end{prop}

\begin{proof}
Firstly, we see that $\Delta$ is a focal cycle if and only if there is \scpo of type (\ref{scposTypeII}) forming an orientation of $\Gamma$.  
Indeed, if $\Delta$ is a focal cycle then the cyclic orientation of $\Gamma$ gives rise to a \scpo of type (\ref{scposTypeII}), as described in Definition~\ref{def:scposTypes}. On the other hand, suppose $\Delta$ admits a \scpo $O$ of type (\ref{scposTypeII}) that forms an orientation of $\Gamma$. By Theorem~\ref{thm:strongly_confluent_attractor} (\ref{thm:strongly_confluent_attractor:item:cycle}) there is a cotree $\Gamma'$ of $\Gamma$ that forms a cycle graph and $O = O_{\Gamma'}^+$ for one of the cyclic orientations for $\Gamma'$. Since $O$ is in fact an orientation of $\Gamma$ (not just a partial orientation) it follows immediately that $\Gamma$ must be a cycle graph.

Secondly, we note that the horocyclic ends correspond to the \scpos of $\Delta$ of type (\ref{scposTypeIII}).  
From these two observations we may thus, henceforth, assume $\Delta$ is not a focal cycle and has no horocyclic ends.  In this case, we see from Theorem~\ref{thm:strongly_confluent_attractor}, that $\Delta$ is irreducible if and only if there is no nonempty \scpo of type (\ref{scposTypeI}) or (\ref{scposTypeII}); in turn, by Lemma~\ref{lem:strongly_confluent_preimage}, no such \scpo exists if and only if $\Delta$ has no proper cotree (see Definition~\ref{def:scposTypes}).

To establish our proposition, we need only prove that $\Delta$ has no proper cotree if and only if it has no stray leaves and no stray half-tree.  Given a leaf $v$ of $\Gamma$, if $|X_v|=1$, then $\Gamma \smallsetminus \{v\}$ is a proper cotree of $\Delta$, whereas if $|X_v| > 1$ we see that $v$ must be included in every cotree of $\Delta$.  If $\Gamma_a$ is a stray half-tree defined by the arc $a$, then $\Gamma \smallsetminus V\Gamma_a$ is a proper cotree of $\Delta$.

Now we assume there are no stray half-trees or stray leaves, and suppose for a contradiction that $\Gamma'$ is a proper cotree of $\Delta$. 
Then:
	(i) the partial orientation $O_{\Gamma'}$ consists of all arcs $b$ such that $o(b) \not \in V\Gamma'$ and $b$ lies on the projecting path from $o(b)$ to $\Gamma'$;
	(ii) $|X_b| = 1$ for all arcs $b \in O_{\Gamma'}$; and 
	(iii) there is an arc $a \in A\Gamma$ such that $o(a) \in V\Gamma'$ and $t(a) \not\in V\Gamma'$ and $\Gamma \smallsetminus \{a\}$ is not connected. 
Let $\Gamma_a$ be the connected component of $\Gamma \smallsetminus \{a\}$ containing $t(a)$. Then we see that $\Gamma_a$ is disjoint from $\Gamma'$ and is hence a tree.  Since there are no stray leaves, we see that $\Gamma_a$ has no leaves of $\Gamma$; thus $\Gamma_a$ has no leaves except possibly $t(a)$, and $t(a)$ is not a leaf of $\Gamma$, so $|V\Gamma_a| \ge 2$.   We also see that every arc $b$ of $\Gamma_a$ oriented towards $t(a)$ belongs to a projecting path, so $|X_b|=1$.  But then $\Gamma_a$ is a stray half-tree, a contradiction.  This contradiction shows that $\Gamma' = \Gamma$ and hence that $\Delta$ is irreducible. 
\end{proof}

Corollary~\ref{cor:DirectlyReadingSimplicity} now follows immediately from Corollary~\ref{cor:geometrically_dense} and Proposition~\ref{prop:irreducible_check} together with Tits' Theorem (Theorem~\ref{thm:Tits}). The proof of Corollary~\ref{cor:DirectlyReadingSimplicity:finite} is also straightforward.

\begin{proof}[Proof of Corollary~\ref{cor:DirectlyReadingSimplicity:finite}]
By hypothesis, the underlying graph $\Gamma$ of the local action diagram $\Delta$ is finite and not a cycle graph.  We immediately see that $\Delta$ is not a focal cycle.  Moreover, since $\Gamma$ is finite, it has no ends and contains no leafless subtree, so $\Delta$ cannot have any horocyclic end or stray half-tree.  Following Proposition~\ref{prop:irreducible_check}, the sole remaining obstacle to the irreducibility of $\Delta$ is if it has a stray leaf, which is exactly the case that $\Gamma$ has a leaf such that $|X_v|=1$. Thus $\Delta$ is irreducible if and only if $|X_v|>1$ for every leaf of $\Gamma$.  The remaining conclusion follows by Tits' theorem.
\end{proof}

As stated, Theorem~\ref{thm:Tits} leaves open the question of whether or not the group $G^+$ generated by the arc stabilisers is simple or trivial.  In fact, this distinction is easy to detect in the local action diagram.

Say that a local action diagram $\Delta = (\Gamma,(X_v),(G(v)))$ is \defbold{free} if for all $v \in V\Gamma$, $G(v)$ acts freely on $X_v$.

\begin{lem}\label{lem:free_diagram}
Let $T$ be a tree, let $G \le \Aut(T)$ and let $\Delta = (\Gamma,(X_v),(G(v)))$ be the associated local action diagram.  Then $G^+$ is trivial if and only if $\Delta$ is free.
\end{lem}

\begin{proof}
Let $\pi = \pi_{(T,G)}$.  Suppose that $G^+$ is trivial.  Then for all $a \in A\Gamma$, $G_a$ is trivial.  Let $v \in VT$.  Then the action of $G_v$ on $o\inv(v)$ is free, since the set of stabilisers of this action is exactly $\{G_a \mid a \in o\inv(v)\}$.  Thus $G(\pi(v))$ acts freely on $X_{\pi(v)}$.

Conversely, suppose that for all $v \in V\Gamma$, $G(v)$ acts freely on $X_v$.  Let $a \in AT$ and let $g \in G_a$.  Suppose $g \neq 1$: then there is some arc $b \not\in \{a,\overline{a}\}$ such that $b$ is directed away from $a$, $g$ fixes $w = o(b)$, but $g$ does not fix $b$.  Then the action of $G_w$ on $o\inv(w)$ corresponds as a permutation group to the action of $G(\pi(w))$ on $X_{\pi(w)}$; that is, the action is free modulo the kernel.  If $w \in \{o(a),t(a)\}$, then clearly $g$ fixes an element of $o\inv(w)$; otherwise, $g$ fixes the unique arc in $o\inv(w)$ in the direction of $\{o(a),t(a)\}$.  Thus $g$ is an element of $G_w$ fixing some element of $o\inv(w)$; since the action of $G_w$ on $o\inv(w)$ is free modulo kernel, we conclude that $gb = b$, contradicting the choice of $b$.  Thus in fact $g=1$, showing that $G_a$ is trivial.  Since this holds for all $a \in AT$, we conclude that $G^+$ is trivial.
\end{proof}

\subsection{Types of action on a tree revisited} \label{TypesOfActionRevisited}

Recall the possible types of action of a group $G$ acting on a tree $T$ given in Theorem~\ref{thm:types}: fixed vertex, inversion, lineal, horocyclic, focal and general. In this subsection we describe how one can recognise these types directly from the local action diagram $\Delta = (\Gamma,(X_a),(G(v)))$, and thus develop a detailed analysis of the case when the action is not geometrically dense.
We make frequent references to correspondences between the action and $\Delta$. These correspondences are built from the following relationships. 
\begin{enumerate}[(i)]
\item
	A correspondence between $T$ and $\Gamma$ via the projection $\pi$.
\item
	A correspondence between the $G$-invariant \scpos of $T$ and the \scpos of $\Delta$ arising from  Lemma~\ref{lem:strongly_confluent_preimage} (and stated explicitly in Definition~\ref{def:scopoDelta}): {\it the $G$-invariant \scpos of $T$ are precisely the preimages $\pi\inv(O)$ of \scpos $O$ of $\Delta$}.
\item
	A correspondence between cotrees or ends and their associated \scpos arising from Definitions~\ref{def:scposTypes} and \ref{def:CotreeDelta} and Theorem~\ref{thm:strongly_confluent_attractor}.
\end{enumerate}

\begin{thm} \label{thm:TypesOfActionRevisited}
Suppose $G$ is a group acting on a tree $T$ and let $\Delta = (\Gamma,(X_a),(G(v)))$ be its local action diagram. The following statements characterise $\Delta$ for the types of action given in Theorem~\ref{thm:types}.
\begin{description}
\item[(Fixed vertex)]
	If $G$ fixes a vertex of $T$, then $\Gamma$ is a tree.  If $\Gamma$ is a tree, the fixed vertices of $G$ on $T$ (if there are any) correspond to cotrees $\Gamma'$ of $\Delta$ consisting of a single vertex of $\Gamma$ with no edges.
\item[(Inversion)]
	$G$ preserves a unique undirected edge (and includes a reversal of that edge) if and only if there is a cotree $\Gamma'$ of $\Delta$ consisting of a single vertex and a single nonorientable loop $a$ with $|X_a|=1$.
\item[(Lineal)]
	$G$ fixes exactly two ends of $T$ and translates the axis between them if and only if there is a cotree $\Gamma'$ of $\Delta$ that is a cycle graph, such that additionally $|X_a|=1$ for all $a \in A\Gamma'$.
\item[(Horocyclic)]
	If $G$ does not fix any vertices in $T$, then $G$ fixes a unique end of $T$ while not including any translations exactly in following situation: $\Gamma$ is a tree, and there is a unique horocyclic end of $\Delta$.
\item[(Focal)]
	The following characterises the situation where $G$ fixes a unique end of $T$ and includes a translation towards this end: There is a cotree $\Gamma'$ of $\Delta$ that is a cycle graph, and a cyclic orientation $O'$ of $\Gamma'$, such that $|X_a|=1$ for $a \in O'$, but $|X_a| \ge 2$ for some $a \in A\Gamma' \smallsetminus O'$. In particular, the case where $G$ has focal action and no proper invariant subtree corresponds exactly to the case where $\Delta$ is a focal cycle such that the colour set of some arc is nontrivial.
\item[(General type)]
	In the remaining case, the unique minimal subtree $T'$ on which $G$ acts geometrically densely corresponds to the unique smallest cotree $\Gamma'$ of $\Delta$, where $\Gamma'$ is not of the special form indicating a bounded (i.e.~fixed vertex or inversion type), lineal or focal action, and $\Delta$ does not have any horocyclic ends.  The action of $G$ on $T$ is geometrically dense if and only if $\Gamma' = \Gamma$.
\end{description}
In the inversion, lineal and focal cases, the cotree $\Gamma'$ is also the unique smallest cotree of $\Delta$.
\end{thm}

We prove the theorem below, examining each case in turn. Before doing so we establish the following correspondence. 

\begin{lem} \label{lem:TypeACorrespondence} There is a correspondence between the $G$-invariant subtrees of $T$, the $G$-invariant \scpos of $T$ of type (\ref{scposTypeI}), and the \scpos of $\Delta$ of type (\ref{scposTypeI}). This correspondence holds in both directions.
\end{lem}

\begin{proof}
Suppose $T'$ is a $G$-invariant subtree of $T$. Being a subtree, it is a cotree, so $T'$ is a $G$-invariant cotree of $T$. From Definition~\ref{def:scposTypes} and its observations (\ref{scposTypeIObs})--(\ref{scposTypeIIIObs})
we have the corresponding $G$-invariant \scpo $O_{T'}$ of type (\ref{scposTypeI}). On the other hand, if we are given a $G$-invariant \scpo of $T$ of type (\ref{scposTypeI}), then by Theorem~\ref{thm:strongly_confluent_attractor} there is a cotree $T'$ of $T$ such that the \scpo is equal to $O_{T'}$, and by observation (\ref{scposTypeIIIObs}) of Definition~\ref{def:scposTypes} this cotree $T'$ is $G$-invariant. Cotrees of connected graphs are connected (see Definition~\ref{def:cotreeGraph}), so $T'$ is a connected nonempty induced subgraph of $T$ and is therefore a subtree. Thus, we have our correspondence between the $G$-invariant subtrees of $T$ and the $G$-invariant type (\ref{scposTypeI}) \scpos of $T$.

We now invoke Lemma~\ref{lem:strongly_confluent_preimage} and deduce that there is \scpo $O$ of $\Delta$ corresponding to $O_{T'}$ via $O_{T'} = \pi\inv(O)$. By Lemma~\ref{lem:invariant_orientation}, $O_{T'}$ is a full orientation if and only if $O$ is a full orientation. By observation~(\ref{scposTypeIIObs}) of Definition~\ref{def:scposTypes}, the type (\ref{scposTypeI}) \scpo $O_{T'}$ is a full orientation if and only if $T'$ consists of a single vertex with no edges. Thus if $T'$ is not a single vertex with no edges then $O$ is not a full orientation so (again by observation~(\ref{scposTypeIIObs})) it must be a \scpo of type (\ref{scposTypeI}). If $T'$ is a single vertex with no edges then the action of $G$ on $T$ is of fixed vertex type, and $G$ must preserve setwise all spheres around the fixed vertex. This means the quotient graph $\Gamma$ is a tree, so there are no \scpos in $\Delta$ of type (\ref{scposTypeII}). 
Since $\pi\inv(O)$ is a \scpo of $T$ of type (\ref{scposTypeI}), one sees that $O$ cannot contain a ray of arcs directed towards an end, so $O$ is not of type (\ref{scposTypeIII}).
Hence $O$ is a \scpo of $\Delta$ of type (\ref{scposTypeI}).

Conversely, suppose we are given a \scpo $O$ of $\Delta$ of type (\ref{scposTypeI}). By Theorem~\ref{thm:strongly_confluent_attractor} there is a cotree $z$ of $\Delta$ such that $O = O_z$, and by Lemma~\ref{lem:strongly_confluent_preimage} we have a $G$-invariant \scpo $\pi\inv(O_z)$ of $T$.
By inspection it is again clear that $\pi\inv(O_z)$ cannot have the structure of a type (\ref{scposTypeIII}) \scpo of $T$, and since $T$ is a tree it cannot be of type (\ref{scposTypeII}). Thus $\pi\inv(O_z)$ is a $G$-invariant \scpo of type (\ref{scposTypeI}).
\end{proof}

Now we prove Theorem~\ref{thm:TypesOfActionRevisited} by considering each type in turn.

\begin{description}
\item[(Fixed vertex)] 
If $G$ fixes a vertex of $T$ then $G$ must preserve setwise all spheres around this fixed vertex. Thus the quotient graph $\Gamma = G \backslash T$ is a tree. If $\Gamma$ is a tree and $G$ fixes a vertex of $T$, then this vertex (with no edges) is a $G$-invariant subtree $T'$ of $T$, which by Lemma~\ref{lem:TypeACorrespondence} corresponds to a \scpo $O$ of $\Delta$ of type (\ref{scposTypeI}). Examining this correspondence, we have the associated \scpo $O_{T'}$ of $T$ with $O_{T'} = \pi\inv(O)$ for some \scpo $O$ of $\Delta$ of type (\ref{scposTypeI}). Thus by Theorem~\ref{thm:strongly_confluent_attractor} (and Remark~\ref{rem:ScopoThmWorksForLADS}) there is a cotree $\Gamma'$ of $\Delta$ such that $O = O_{\Gamma'}$. Now $O_{T'}$ is a full orientation of $T$ by observation~(\ref{scposTypeIIObs}) and so by Lemma~\ref{lem:invariant_orientation} the \scpo $O_{\Gamma'}$ must be a full orientation of $\Gamma$. By observation~(\ref{scposTypeIIObs}) again we have that $\Gamma'$ must consist of a single vertex of $\Gamma$ with no edges.\\

\item[(Inversion)] 
Assume that $G$ preserves a unique undirected edge of $T$ and reverses it, and let $T'$ be the $G$-invariant subtree of $T$ consisting of this edge and its two end vertices. Write $\Gamma' := \pi(T')$, and note that $\Gamma'$ consists of a vertex $w$ with a nonorientable loop $a$ such that $o(a) = t(a) = w$.
The action of $G$ on $T$ must preserve setwise all spheres around the invariant edge, so one can easily verify that $\Gamma'$ is a cotree of 
$\Delta$.
It is also clear that every cotree of $\Delta$ must include the loop $a$, so $\Gamma'$ is the smallest cotree. 
It remains then to show that the loop $a$ is \defbold{monochromatic} (that is, $|X_a| = 1$).
Using the notation of associated local action diagrams (Definition~\ref{def:AssocLocalActionDiagram}), $w = o(a)$ so we choose $w^* \in \pi\inv(w) \subseteq VT$ and see that $w^*$ has to be one of the two end vertices of the preserved edge. Then $X_a = \{b \in o\inv(w^*) \mid \pi(b) = a\}$, but $a$ represents the $G$-orbit  that consists of only the unique edge of $T$ invariant under $G$. Hence $|X_a|=1$.

On the other hand, suppose $G \leq \Aut(T)$ is any action and there is a cotree $\Gamma'$ of $\Delta$ consisting of a single vertex $w$ and a single nonorientable loop $a$ with $X_a = \{c\}$ for some colour $c$. We first construct the $\Delta$-tree $\mathbf{T}$ 
according to Definition~\ref{def:deltaTree} and the proof of Lemma~\ref{lem:tree_isomorphism}, starting from a vertex $w^* \in VT$ such that $\pi(w^*) = w$. Notice that there is precisely one arc $b \in o\inv(w^*)$ that projects onto $a$, and in the $\Delta$-tree $\mathbf{T}$ the arc $b$ is coloured $c$. This arc $b$ takes us to another vertex $w^{**} := t(b) \in VT$ and thus $\pi(w^{**}) = w$. There is now no choice for the colour of the arc $\overline{b}$ from $w^{**}$ to $w^*$: it must be coloured $c$ also. Adding coloured arcs to form $\mathbf{T}$ building out from $w^*$ and $w^{**}$ we see that $\Gamma'$ being a cotree of $\Delta$ forces upon us an arc colouring $\mc{L}$ such that no additional arcs are given the colour $c$. Thus the action of the universal group $\Univ(\Delta)$ (see Definition~\ref{def:THEUniversalGp}) on $\mathbf{T}$ must fix the edge $\{w^*, w^{**}\}$ setwise while also admitting a reversal of the edge. Since the action of $\Univ(\Delta)$ is just the action of the $\propP{}$-closure of $G$ on $T$ (see Theorem~\ref{thm:G_to_U}), we see that $G$ also preserves an undirected edge and includes a reversal of that edge. Of course, since $T$ is a tree, $G$ cannot simultaneously preserve and reverse two distinct edges in $T$.\\

\item[(Lineal)]
Suppose that $G$ fixes exactly two ends and translates the axis $L$ between them, so $L$ is a $G$-invariant subtree of $T$. In the quotient graph $\Gamma = G \backslash T$ the image of $L$ is a cycle graph we denote by $\Gamma'$. Because $G$ fixes the two ends of $L$, it does not invert any arc in $L$, so for all $a \in A\Gamma$ we have $a \neq \overline{a}$. Since $T$ is a tree, there is a closest point projection map from $VT$ to $VL$ (see Definition~\ref{Def:TitsP}), and the distance of a vertex $v\in VT$ from $L$ is the distance from $v$ to its closest point projection on $L$. The action of $G$ on $T$ must preserve this distance from $L$, and from this it is easy to check that $\Gamma'$ is  the unique smallest cotree of $\Delta$. 
Now fix $a \in A\Gamma'$ and set $w := o(a)$. We will show that $a$ is monochromatic.
We can choose a representative $w^* \in \pi\inv(w)$ with $w^* \in VL$. Since $L$ is $G$-invariant, only arcs in $L \cap o\inv(w^*)$ can lie in $X_a = \{b\in o\inv(w^*) \mid \pi(b) = a\}$ and $X_{\overline{a}}$. Both $X_a$ and $X_{\overline{a}}$ are nonempty and $|L \cap o\inv(w^*)| = 2$, therefore $|X_a|=1 = |X_{\overline{a}}|$.

Conversely, suppose now that $G \leq \Aut(T)$ is any action and that there is a cotree $\Gamma'$ of $\Delta$ that is a cycle graph, and a cyclic orientation $O'^+$ of $\Gamma'$, such that $|X_a|=1$ for $a \in O'^+$.  We claim that in this case, there is an end $\xi^+$ of $T$ that is fixed by $G$, and $G$ includes a translation towards the end $\xi^+$.

If we construct the associated $\Delta$-tree $\mathbf{T}$ as we did in the inversion case, we see that the preimage $\pi\inv(O'^+)$ determines a unique end $\xi^+$ of $T$ such that all arcs in $\pi\inv(O'^+)$ are directed towards $\xi^+$.
We can now clearly see how the associated \scpos of $\Delta$ and $T$ are related. 
Following Definition~\ref{def:CotreeDelta}, we have a type (\ref{scposTypeII}) \scpo $O_{\Gamma'}^+$ of $\Delta$ associated with the cycle graph $\Gamma'$ and its orientation $O'^+$. 
Now $O_{\Gamma'}^+$ includes the cyclic orientation $O'^+$, so by inspection the preimage $\pi\inv(O_{\Gamma'}^+)$ consists of all arcs in $T$ directed towards $\xi^+$. By Lemma~\ref{lem:strongly_confluent_preimage} and Theorem~\ref{thm:strongly_confluent_attractor}  we see that (i) the preimage $\pi\inv(O_{\Gamma'}^+)$ is a $G$-invariant \scpo in $T$ and (ii) $\pi\inv(O_{\Gamma'}^+)$ is the associated \scpo $O_{\xi^+}$.
Observation (\ref{scposTypeIIIObs}) of Definition~\ref{def:scposTypes} implies that $\xi^+$ is $G$-invariant. Because a directed ray in the end $\xi^+$ projects onto a directed finite cycle in the quotient graph $G \backslash T$, it is clear that $G$ admits a translation towards  $\xi^+$.

Now suppose that, in addition, the reverse cyclic orientation $O'^-$ of our cotree $\Gamma'$ satisfies $|X_a| = 1$ for all $a \in O'^-$. Then our claim tells us there is another end $\xi^-$ of $T$ that is $G$-invariant and $G$ includes a translation towards $\xi^-$. Since the directed rays in $T$ arising from $O'^+$ and $O'^-$ do not lie in the same end, we see that $\xi^-$ and $\xi^+$ are distinct.
Thus, we have a $G$-invariant axis $L$ between $\xi^-, \xi^+$ that projects to a finite cycle in the quotient graph $G \backslash T$, and there is a translation in $G$ along $L$.

Finally, we note that three distinct ends of $T$ determine a unique vertex of $T$. Therefore, if there were three distinct $G$-invariant ends then $G$ would fix a vertex. Because $G$ contains a translation, it does not fix any vertex. \\

\item[(Horocyclic)]
Suppose $G$ fixes a unique end $\xi$ of $T$ but does not include any translations. Then $G$ preserves every horoball around $\xi$ and it is easy to see that the quotient graph $\Gamma$ is a tree and $\xi$ corresponds to an end $z$ of $\Gamma$. This gives us a \scpo $O_z$ of $\Gamma$ of type (\ref{scposTypeIII}) consisting of all arcs in $A\Gamma$ directed towards $z$. If $a \in O_z$ then we set $w:=o(a)$ and choose  representatives $w^* \in \pi\inv(w)$ and $b \in o\inv(w^*)$ such that $\pi(b) = a$. 
Now $a$ is directed towards $z$, therefore $b$ is directed towards $\xi$. 
Because the action of $G$ on $T$ fixes $\xi$, no element of $G$ can move $b$ to another arc in $o\inv(w^*)$ and so $|X_a| = 1$. Thus $z$ is a horocyclic end of $\Delta$. To see why the horocyclic end is unique, suppose $z, z'$ are distinct horocyclic ends. We find that the line $L$ from $z$ to $z'$ is a cotree of $\Delta$ such that $|X_a|=1$ for all $a \in AL$, and hence $G$ fixes pointwise the preimage of $L$ in $T$.

Now let us instead suppose that $\Gamma$ is a tree with a unique horocyclic end $z$ of $\Delta$. Let $O_z$ be the \scpo of $\Delta$ arising from this end.
By Lemma~\ref{lem:strongly_confluent_preimage} the preimage $O = \pi\inv(O_{z})$ is a $G$-invariant \scpo in $T$. 
Because $|X_a|=1$ for all arcs in $\Gamma$ directed towards $z$, we see that $O$ consists of all arcs in $T$ directed towards some end $\xi$ of $T$.
By Theorem~\ref{thm:strongly_confluent_attractor}, we see that $O$ is of type (\ref{scposTypeIII}) with $O = O_\xi$. By observation (\ref{scposTypeIIIObs}) of Definition~\ref{def:scposTypes} the end $\xi$ is invariant under $G$. Because $\xi$ corresponds to an end of $\Delta$ rather than a cycle, there is no translation in $G$ fixing $\xi$. If there were two such ends then $G$ would fix pointwise the axis between them.\\

\item[(Focal)] 
Suppose that $G$ fixes a unique end $\xi$ of $T$ and includes a translation towards this end. Then $O_\xi$ is a $G$-invariant \scpo of $T$ (by observation (\ref{scposTypeIIIObs}) of Definition~\ref{def:scposTypes}) and so by Lemma~\ref{lem:strongly_confluent_preimage} the \scpo $O_\xi$ is the preimage $\pi\inv(O)$ of some \scpo $O$ of $\Delta$. Now $O_\xi$ consists of all arcs in $T$ directed towards $\xi$ and $\xi$ is fixed by $G$.
Also $G$ includes a translation towards $\xi$, and the periodic orbits of this translation guarantee: (i) the existence of a cycle graph $\Gamma'$ in the projection $\Gamma$ and (ii) one of the cyclic orientations of $\Gamma'$ must lie in $\pi(O_\xi) = O$. We thus have that some $v \in V\Gamma'$ lies in the attractor $K(O)$ and $o\inv(v) \cap O \neq \emptyset$. By Theorem~\ref{thm:strongly_confluent_attractor} (and Remark~\ref{rem:ScopoThmWorksForLADS}) we have that $O$ is a \scpo of type
(\ref{thm:strongly_confluent_attractor:item:cycle}) and $\Gamma'$ is a cotree of $\Delta$, with $O = O_{\Gamma'}^+$ for one of the cyclic orientations of $\Gamma'$.  
 It is clear that every cotree of $\Delta$ must include $\Gamma'$, so $\Gamma'$ is the unique smallest cotree of $\Delta$. 
If $|X_a|=1$ for all $a \in A\Gamma'$ then we would have the lineal case which is impossible (because $G$ fixes only one end of $T$). We have already established that all arcs in $O$ are monochromatic, therefore there must be a non-monochromatic arc in the cyclic orientation of $\Gamma'$ that is not in $O$.

Suppose now that $G \leq \Aut(T)$ is any action with a cotree $\Gamma'$ of $\Delta$ that is a cycle graph, and a cyclic orientation $O'$ of $\Gamma'$, such that $|X_a|=1$ for all $a \in O'$, but $|X_a| \ge 2$ for some $a \in A\Gamma' \smallsetminus O'$.  
 We see that $\Gamma'$ must be the unique smallest cotree of $\Delta$. 
By our claim in the lineal case, there is an end $\xi$ of $T$ that is fixed by $G$, and $G$ includes a translation towards the end $\xi$. Because $G$ includes a translation, it fixes no vertex and leaves no edge invariant and so by Theorem~\ref{thm:types} the action must be lineal or focal. However the fact that $|X_a| \ge 2$ for some $a \in A\Gamma'$ contradicts what has been established in the lineal case, so the action must be focal.\\ 

\item[(General type)] The remaining case arises precisely when $\Delta$ is not of the special form indicating a bounded (i.e.~fixed vertex or inversion type), lineal, horocyclic or focal action.
Theorem~\ref{thm:types} tells us that general type arises precisely when the other types do not occur. Suppose the action of $G$ on $T$ is of general type. Then $G$ acts with translation and by Lemma~\ref{lem:minimal_invariant}, there is a smallest invariant subtree $T'$ that is a line or an infinitely ended tree and $T'$ arises as the union of the axes of translation
of $G$, with $G$ acting geometrically densely on $T'$. As noted in Remark~\ref{rem:SmallestCotreeOfDelta} there is a unique smallest cotree $\Gamma'$ of $\Delta$ and $\pi(T') = \Gamma'$.
This concludes the proof of Theorem~\ref{thm:TypesOfActionRevisited}.
\end{description}

In the following remark we summarise our analysis and give some consequences that are easily verified.

\begin{rem}
 In all cases except when $G$ fixes a vertex, there are at most two invariant ends, and the invariant ends are easily identified.  If $G$ fixes a vertex $v$ of $T$, then the invariant ends correspond to rays starting at $v$ that are fixed pointwise by $G$, so they are accounted for by invariant subtrees.  We also note that for lineal and focal actions, there is a unique smallest invariant subtree $T'$ spanned by all axes of translation of $G$; the only distinction from general type is the existence of one or two fixed ends of this subtree.

If $\Gamma$ is a tree, then the action has a fixed vertex, or is horocyclic or of general type; in particular, by Corollary~\ref{cor:vertex_group}, if $G$ is generated by vertex stabilisers then the action must be of one of these types.  If $\Gamma$ is not a tree, then the possibilities are: inversion (in which case $\Gamma$ is a tree plus a single nonorientable loop), lineal, focal and general type.

We also have the following correspondence between types of \scpo given by $O \mapsto \pi_{(T,G)}(O)$:
\begin{enumerate}[(i)]
\item Invariant subtrees of $T$, or equivalently, invariant \scpos of $T$ of type (\ref{scposTypeI}), correspond to \scpos of $\Delta$ of type (\ref{scposTypeI}). (This was shown in Lemma~\ref{lem:TypeACorrespondence}.) 
\item Inversion and general type actions do not have invariant ends.  Otherwise, there are two kinds of invariant end to consider:
\begin{enumerate}[(1)]
\item If $\Gamma$ is a tree, the action could have a fixed vertex or be horocyclic, without translation.  In this case invariant ends of $T$ correspond to \scpos of $\Delta$ of type (\ref{scposTypeIII}).
\item If $\Gamma$ is not a tree, the action could be lineal or focal with a translation towards an invariant end.  In this case invariant ends of $T$ correspond to \scpos of $\Delta$ of type (\ref{scposTypeII}), with the order of the associated cycle graph in $\Gamma$ corresponding to the minimal translation length of a translation towards the fixed end.
\end{enumerate}
\end{enumerate}
\end{rem}

\section{The group topology}\label{sec:topology}

Our definitions ensure that whenever $G \le \Aut(T)$ is such that $G = G^{\propP{}}$, then $G$ is a closed subgroup of $\Aut(T)$ in the permutation topology.  In particular, it follows that $G$ is a non-Archimedean topological group in its own right.

There are natural characterisations of when $G^{\propP{}}$ is Polish (see \S\ref{PolishPropPGroups}) or locally compact (see \S\ref{LocallyCompactPClosedGroups}) as a subgroup of $\Aut(T)$. In \S\ref{sec:topology:proofs} we prove two topological results from the introduction: Theorem~\ref{thm:comp_gen+geom_dense} and Corollary~\ref{cor:comp_gen+simple}.

\subsection{Polish $\propP{}$-closed groups} \label{PolishPropPGroups}

Here is a characterisation of Polish $\propP{}$-closed groups.  For clarity, we note that the word `countable' here is understood to allow finite sets as well as countably infinite sets.

Let $X$ be a set, let $G$ act by permutations on $X$ and let $Y \subseteq X$ be $G$-invariant.  We say that the action of $G$ on $Y$ is \defbold{strongly faithful (relative to $X$)} if for all $x \in X$, there exists a finite subset $\{y_1,\dots,y_n\}$ of $Y$ such that $\bigcap^n_{i =  1}G_{y_i}$ fixes $x$.

\begin{lem}\label{lem:Polish_permutation}
Let $X$ be a set and let $G$ be a closed subgroup of $\Sym(X)$.  Then $G$ is Polish if and only if there is a countable subset $Y$ of $X$ on which $G$ acts strongly faithfully.  Moreover, if $Y$ is a countable strongly faithful set for $G$, then the induced homomorphism $\theta: G \rightarrow \Sym(Y)$ is a closed topological embedding.
\end{lem}

\begin{proof}We note first that $\Sym(X)$ is a non-Archimedean topological group, that is, it has a base of neighbourhoods of the identity consisting of open subgroups.  It follows that any subgroup of $\Sym(X)$ with the subspace topology is also non-Archimedean.

Suppose $Y = \{y_0,y_1,\dots\}$ is a countable strongly faithful set for $G$ and let $\theta: G \rightarrow \Sym(Y)$ be the natural homomorphism.  Then $\theta$ is clearly injective and continuous. To show that $\theta$ is a closed map (i.e.~the image of every closed set is closed), we need to show that given a net $(g_i)_{i \in I}$ in $G$ such that $\theta(g_i)$ converges to the identity, then $(g_i)$ converges to the identity.  Indeed, since $\theta(g_i) \rightarrow 1$ as $i \rightarrow \infty$, we see that for all $n \ge 0$, there exists $i_n$ such that $g_i$ fixes $y_0,\dots,y_n$ for all $i > i_n$, so $g_i \in G_n := \bigcap^n_{i =  0}G_{y_i}$. Since the groups $G_n$ are a base of neighbourhoods of the identity, it follows that $g_i \rightarrow 1$ as $i \rightarrow \infty$.  Thus $\theta$ is a closed topological embedding as claimed.

In particular, we see from the previous paragraph that if a countable strongly faithful set $Y$ exists for $G$, then $G$ is isomorphic to a closed subgroup of the Polish group $\Sym(Y)$; thus $G$ is Polish.

Conversely, suppose that $G$ is Polish.  In particular, $G$ is separable, so all open subgroups of $G$ have countable index; thus $G$ has countable orbits on $X$.  Since $G$ is non-Archimedean and metrisable, there is a countable set of open subgroups of $G$ forming a base of neighbourhoods of the identity; recalling the standard base of topology for $\Sym(X)$, in fact there is a sequence $(y_i)_{i \ge 0}$ of points such that $\{G_n \mid n \ge 0\}$ is a base of neighbourhoods of the identity, where
\[
G_n = \bigcap^n_{i =  0}G_{y_i}.
\]
Given $x \in X$, we then see that $G_n \le G_x$ for some $n$, showing that the action of $G$ on $Y$ is strongly faithful, where $Y = \{gy_i \mid g \in G, i \ge 0\}$.
\end{proof}

\begin{prop}\label{prop:Polish}
Let $T$ be a tree and let $G \le \Aut(T)$ be $\propP{}$-closed.  Then the following are equivalent.
\begin{enumerate}[(i)]
\item The permutation topology on $G$ is Polish.
\item There is a countable $G$-invariant subtree $T'$ of $T$ on which $G$ acts strongly faithfully relative to $T$.
\item There is a countable $G$-invariant subtree $T'$ of $T$ on which $G$ acts faithfully and such that for each $v \in VT'$, the action of $G_v$ on $o\inv_{T'}(v)$ is strongly faithful relative to $o\inv_T(v)$.
\item The local action diagram $\Delta = (\Gamma,(X_a),(G(v)))$ for $(T,G)$ is such that all the colour sets $X_a$ are countable and the permutation groups $G(v)$ are Polish, and there exists a countable cotree $\Gamma'$ of $\Delta$ such that $G(v) = \triv$ whenever $v \in V\Gamma \smallsetminus V\Gamma'$.
\end{enumerate}
\end{prop}

\begin{proof}
Suppose (i) holds.  Then by Lemma~\ref{lem:Polish_permutation}, there is a countable set $Y$ of vertices on which $G$ acts strongly faithfully relative to $VT$.  We then see that the unique smallest subtree $T'$ of $T$ containing $Y$ is countable; by the construction, $T'$ is $G$-invariant and $G$ acts strongly faithfully on $VT'$ relative to $VT$.  Thus (i) implies (ii).

Suppose (ii) holds; let $v \in VT'$ and let $a \in o\inv(v)$ such that $a \not\in AT'$.  Since $G$ acts strongly faithfully on $T'$, there are vertices $w_1,\dots,w_n$ in $T'$ such that any element of $G$ that fixes $w_1,\dots,w_n$ also fixes $t(a)$.  For each $w_i$, let $a_i$ be the first arc on the shortest directed path from $v$ to $w_i$, and let $g \in G_v$ be such that $g$ fixes $a_i$ for $1 \le i \le n$.  Using property $\propP{}$ and induction on $n$, we see that there is $g' \in G$ such that $g'$ fixes each of the half-trees $T_{a_i}$ pointwise, but has the same action as $g$ on the vertices outside of $\bigcup^n_{i=1}T_{a_i}$.  In particular, $g'$ fixes $w_1,\dots,w_n$, so $g'$ also fixes $t(a)$; since $a \not\in AT'$, we see that $t(a)$ is not contained in $\bigcup^n_{i=1}T_{a_i}$, from which it follows that $g$ also fixes $t(a)$, and hence $g$ fixes $a$.  This proves that relative to the action on $o\inv_T(v)$, the action of $G_v$ on $o\inv_{T'}(v)$ is strongly faithful.  Thus (ii) implies (iii).

Suppose (iii) holds.  Then every vertex stabiliser in $G$ also has countable orbits on $T$, ensuring that the colour sets $X_a$ are countable.  The countable $G$-invariant subtree $T'$ gives rise to a countable cotree $\Gamma'$.  Consider $v_0 \in VT \smallsetminus VT'$.  Then there is a unique shortest path $(v_0,\dots,v_n)$ from $v_0$ to $T'$; in particular, $G_{v_0}$ fixes the arc $a$ from $v_0$ to $v_1$.  Then $T'$ is contained in the half-tree $T_{a}$, so $G_a$ acts faithfully on $T_{a}$.  By property $\propP{}$, it follows that the action of $G_a$ on $T_{\ol{a}}$ is trivial; in particular, since $G_{v_0}$ also fixes $a$, the action of $G_{v_0}$ on $T_{\ol{a}}$ is trivial.  Since $T_{\ol{a}}$ contains all neighbours of $v_0$ other than $v_1$, we conclude that $G_{v_0}$ fixes $o\inv(v_0)$ pointwise.  Thus $G(\pi(v_0)) = \triv$ and we see that $|X_a| = 1$ whenever $a \in o\inv(\pi(v_0))$.  Now consider $v \in VT'$; we see that the condition that the action of $G_v$ on $o\inv_{T'}(v)$ is strongly faithful relative to $o\inv_T(v)$ translates exactly to the condition that for the action of $G(\pi(v))$ on $X_{\pi(v)}$, we have a strongly faithful action on the subset 
\[
Y = \bigsqcup_{a \in o\inv_{\Gamma'}(\pi(v))} X_a
\]
of $X$.  We observe that $Y$ is countable; $G(\pi(v))$ is closed by definition, and hence Polish by Lemma~\ref{lem:Polish_permutation}; in particular, $G(\pi(v))$ has countable orbits, so $|X_a| \le \aleph_0$ for all $a \in o\inv(\pi(v))$.  This completes the proof that (iii) implies (iv).

Suppose (iv) holds.  Consider the construction of a $\Delta$-tree $T$ starting from a root vertex $()$ mapping to a base vertex $v_0 \in V\Gamma'$.  The cotree $\Gamma'$ then gives rises to a $G$-invariant subtree $T'$ of $T$; the fact that both $\Gamma'$ and the colour sets $X_a$ are countable ensures that $T'$ is countable.  For each $v \in VT'$, we add a countable union of $G_v$-orbits of neighbours of $v$ to $T'$ to produce a new tree $T''$, such that the action of $G_v$ on $o\inv_{T''}(v)$ is strongly faithful relative to $o\inv_T(v)$; this is possible by Lemma~\ref{lem:Polish_permutation} since $G(\pi(v))$ is Polish, and we see that it can be done in such a way that $T''$ is $G$-invariant.  Now let $v_0 \in VT \smallsetminus VT''$ and let $(v_0,\dots,v_n)$ be the shortest path from $v_0$ to $T'$.  Let $a$ be the arc from $v_n$ to $v_{n-1}$.  Our choice of $T''$ ensures that there are arcs $a_1,\dots,a_m \in o\inv_{T''}(v_n)$ such that given $g \in G_{v_n}$ that fixes $a_1,\dots,a_m$, or equivalently, given $g \in G$ that fixes $t(a_1),\dots,t(a_m),v_n$, then $g$ fixes $a$, and hence $g$ fixes $v_{n-1}$.  In turn, using the fact that the local actions for vertices outside of $\Gamma'$ are trivial, we see that $G_{v_{n-1}}$ fixes each of the vertices $v_0,\dots,v_{n-2}$.  We have thus obtained a finite subset $\{t(a_1),\dots,t(a_m),v_n\}$ of $VT''$ whose pointwise stabiliser also fixes $v_0$.  Since $v_0 \in VT \smallsetminus VT''$ was arbitrary, we conclude that $G$ acts strongly faithfully on $T''$.  It then follows by Lemma~\ref{lem:Polish_permutation} that $G$ is Polish.  Thus (iv) implies (i) and the cycle of implications is complete.
\end{proof}

\subsection{Locally compact $\propP{}$-closed groups} \label{LocallyCompactPClosedGroups}

Here is the characterisation of local compactness of the $\propP{}$-closure.

\begin{prop}\label{prop:locally_compact}
Let $T$ be a tree and let $G \le \Aut(T)$.  Suppose that there is a unique minimal $G$-invariant subtree $T'$, such that $|VT'|\ge 3$.  Then the following are equivalent.
\begin{enumerate}[(i)]
\item The $\propP{}$-closure of $G$ is locally compact.
\item For all $a \in AT'$, the stabiliser of $a$ in the $\propP{}$-closure of $G$ is compact.
\item Let $\Delta = (\Gamma,(X_a),(G(v)))$ be the local action diagram for $(T,G)$ and let $\Gamma'$ be the unique smallest cotree of $\Delta$.  Then for all $a \in A\Gamma$ such that $\ol{a} \not\in O_{\Gamma'}$, every $G(o(a))$-stabiliser of every point in $X_a$ has finite orbits on $X_{v}$.
\end{enumerate}
\end{prop}

\begin{proof}
Lemma~\ref{lem:minimal_invariant} implies that $T'$  is the union of the axes of translation of $G$.   Let $\Gamma = G \backslash T$ and let $\pi: T \rightarrow \Gamma$ be the quotient map.

Suppose that (i) holds.  Then there is some finite set $B$ of vertices of $T$, such that the pointwise stabiliser $H$ of $B$ in $G^{\propP{}}$ is compact.  Let $a \in AT'$. Then $a$ belongs to the axis of some translation $h \in G$; we can choose $h$ so that $a$ is oriented towards the repelling end of $h$.  Then the half-trees $T_{h^na}$ form an increasing family whose union is $T$; thus there is some $n$ such that $T_{h^na}$ contains $B$ and $t(h^na) \not\in B$.

Set $g = h^n$ and consider the stabiliser of $ga$; we have
\[
G^{\propP{}}_{ga} = \rist_G(T_{ga}) \times \rist_G(T_{\ol{ga}}).
\]
Our choice of $g$ ensures that $B \subseteq T_{ga}$; thus $\rist_G(T_{\ol{ga}}) \le H$, so $\rist_G(T_{\ol{ga}})$ is compact.  In particular, $G^{\propP{}}_{ga}$ has finite orbits on $T_{\ol{ga}}$.  After conjugating by $g\inv$ we see that $G^{\propP{}}_a$ has finite orbits on the half-tree $T_{\ol{a}}$.  A similar argument using $\ol{a}$ in place of $a$ shows that $G^{\propP{}}_a = G^{\propP{}}_{\ol{a}}$ also has finite orbits on the complementary half-tree $T_{a}$, so $a$ satisfies (ii).  Thus (i) implies (ii).  It is immediately clear that (ii) implies (i), so (i) and (ii) are equivalent.

Suppose (ii) holds; note that the unique smallest cotree $\Gamma'$ of the local action diagram is the image of $T'$.  Let $v \in VT$ and $a \in o\inv(v)$ such that $\ol{\pi(a)} \not\in O_{\Gamma'}$.  If $\pi(a) \not\in O_{\Gamma'}$, then $\pi(a) \in A\Gamma'$ and it is clear from (ii) that $G_a$ has finite orbits on $VT$.  Otherwise, the fact that $\pi(a) \in O_{\Gamma'}$ ensures that $a$ points towards $T'$.   By property $\propP{}$ the action of $G^{\propP{}}_a$ on $T_{\ol{a}}$ is then the same as the action of the pointwise fixator of $T_a$ on $T_{\ol{a}}$; since $T' \subseteq T_a$, it follows from (ii) that $G_a$ has finite orbits on $T_{\ol{a}}$, and in particular on $o\inv(v)$.  It follows that in the local action diagram, the action of $G(\pi(v))$ on $X_v$ is such that every point stabiliser has finite orbits.  Thus (ii) implies (iii).

Suppose (iii) holds.  We can regard $G^{\propP{}}$ as the universal group of $\Delta$ and we see that $\Gamma' = \pi(T')$.

Consider a path $(v_0,\dots,v_n)$ in $VT$; let $a_i$ be the arc from $v_{i-1}$ to $v_{i}$ and suppose that $a_1 \in AT'$.  We then see that each arc $a_i$ is either contained in $T'$ or points away from it, so for all $i$ we have $\pi(a_i)\not\in O_{\Gamma'}$.  Thus for each $i \ge 0$, the stabiliser in $G(v_{i})$ of any point in $X_{\ol{\pi(a_i)}}$ has finite orbits on $X_{\pi(v_i)}$.  Translating this information back to the tree: in the action of the stabiliser in $G^{\propP{}}$ of $\ol{a_i}$, or equivalently of $a_i$, the orbit of $v_{i+1}$ is finite.  We conclude that in the action of $G^{\propP{}}_{a_1}$, the orbit of $v_n$ is finite.  Given the freedom of choice of the path $(v_0,\dots,v_n)$, we conclude that $H = G^{\propP{}}_{a_1}$ has finite orbits on the half-tree $T_{a_1}$; by replacing $a_1$ with $\ol{a_1}$, a similar argument shows that $H$ has finite orbits on the complementary half-tree $T_{\ol{a_1}}$.  Since $H$ is also closed in $\Aut(T)$, it is compact.  Moreover, $a_1$ can be chosen to be any arc of $AT'$; thus every stabiliser in $G^{\propP{}}$ of an arc of $T'$ is compact.  Thus (iii) implies (ii), completing the proof that all three statements are equivalent.
\end{proof}

\begin{cor} \label{cor:StronglyClosedIsNoBigDeal}
Let $T$ be a tree and let $G$ be a closed subgroup of $\Aut(T)$ that acts with translation and let $G^{\propP{}}$ be the $\propP{}$-closure of $G$ acting on $T$.  Suppose that $G^{\propP{}}$ is locally compact, and let $T'$ be a $G$-invariant subtree of $T$.  Then the kernel of the action of $G$ on $T'$ is compact, and $G$ and $G^{\propP{}}$ are both strongly closed.
\end{cor}

\begin{proof}
Since $G$ acts with translation we are in the situation of Proposition~\ref{prop:locally_compact}: there is a unique $G$-invariant subtree $T''$, such that $|VT''|\ge 3$.  In particular, $T'' \subseteq T'$. Since $G^{\propP{}}$ is locally compact, it follows from Proposition~\ref{prop:locally_compact} that $G^{\propP{}}$ has a compact arc stabiliser, for some arc $a \in AT'' \subseteq AT$; since $G$ is closed, the subgroup $G_a = G \cap G^{\propP{}}_a$ is also compact. Thus by Lemma~\ref{lem:StronglyClosedFromCompactArcs}  both $G$ and $G^{\propP{}}$ are strongly closed as subgroups of $\Aut(T)$.

Let $K$ be the kernel of the action of $G$ on $T'$.  Then $K$ is a closed subgroup of $G_a$, so it is compact.
\end{proof}

When $G$ is locally compact, it is natural to ask if $G$ is compactly generated.  As in Proposition~\ref{prop:locally_compact}, for simplicity we will avoid the case when $G$ fixes a vertex or preserves an undirected edge.  Excluding these degenerate cases, compact generation of $G$ is easily seen to be equivalent to compact generation of $G^{\propP{}}$, so the question of whether or not $G$ is compactly generated can be reduced to the local action diagram.

\begin{prop}\label{prop:compact_generation}
Let $T$ be a tree and let $G \le \Aut(T)$ be closed with unbounded action.  Let $\Delta = (\Gamma,(X_a),(G(v)))$ be the local action diagram.  Suppose that there exists $a \in AT$ for which $G_a$ is compact.  Then $G$ and $G^{\propP{}}$ are locally compact, and the following are equivalent.
\begin{enumerate}[(i)]
\item $G$ is compactly generated;
\item $G^{\propP{}}$ is compactly generated;
\item there is a unique smallest $G$-invariant subtree $T'$ such that $G$ has finitely many orbits on $VT' \sqcup AT'$ and $G_v$ is compactly generated for each $v \in VT'$;
\item there is a unique smallest cotree $\Gamma'$ of $\Delta$ such that $\Gamma'$ is finite and $G(v)$ is compactly generated for each $v \in V\Gamma'$.
\end{enumerate}
\end{prop}

\begin{proof}
Since $G_a$ is open in $G$, we see that $G$ is locally compact.  We also see that $G_a$ has finite orbits on $VT$; since $G^{\propP{}}_a$ is closed and has the same orbits it follows that $G^{\propP{}}_a$ is compact, and hence $G^{\propP{}}$ is locally compact.  Since $G^{\propP{}}$ has the same orbits on $AT$ as $G$ does, we have $G^{\propP{}} = GG^{\propP{}}_a$, so $G$ is cocompact in $G^{\propP{}}$.  Thus $G$ is compactly generated if and only if $G^{\propP{}}$ is compactly generated, that is, (i) and (ii) are equivalent.

Suppose that the action is horocyclic with unique fixed end $\xi$.  We see that $G$ preserves every horoball around $\xi$; in particular, there is no minimal $G$-invariant subtree, so (iii) is false.  For the same reason every cotree of $\Delta$ is infinite, so (iv) is false. We also see that (i) and (ii) are false by Corollary~\ref{cor:horocyclic}.  So if the action is horocyclic then (i)--(iv) are all false, which is consistent with them being equivalent.

For the remainder of the proof we may suppose that the action is not horocyclic.  It follows that $G$ acts with translation, and hence there is a unique smallest $G$-invariant subtree $T'$ of $T$, which is the union of the axes of translations in $T$. 

Suppose (i) holds.  Since $G$ is compactly generated, it is generated by finitely many cosets of any vertex stabiliser; hence by Lemma~\ref{lem:finite_type}, $G$ has finitely many orbits on $VT' \sqcup AT'$.  The rest of the proof that (i) implies (iii) is \cite[Proposition~4.1]{Castellano}, however for clarity we give a more elementary proof using Bass--Serre theory.

Write $\Gamma^i$ for the quotient graph of the action of $G$ on the inversion-free subdivision of $T'$.  Since $G^{\propP{}}$ is locally compact, Proposition~\ref{prop:locally_compact} implies that $G^{\propP{}}_e$ is compact for all $e \in AT'$, and hence $G$ acts on $T'$ with compact kernel; let $G^*$ be the group defined by this action.  Recall the decomposition of $G^*$ given by Theorem~\ref{thm:BassSerre}:
\[
\frac{F(E) \ast \Asterisk_{v \in VT^*}G^*_v}{\langle \langle s_a\tau_a(g)s_{\ol{a}}\tau_{\overline{a}}(g)^{-1} \; (a \in E, g \in G_a), \; s_as_{\ol{a}}  \; (a \in E), \; s_a \; (a \in AT^*) \rangle \rangle}.
\]
where now $T^*$ is a lift in $T'$ of a maximal subtree of $\Gamma^i$ and $E$ is a lift of the arcs of $\Gamma^i$ such that every arc in $E$ is incident with $T^*$.  We remark that this decomposition is also well-behaved with respect to the topology of $G^*$, and hence of $G$, since the factors are amalgamated along open subgroups.

We claim that $G_v$ is compactly generated for all $v \in VT'$; it is enough to show that $G^*_v$ is compactly generated for each $v \in VT^*$.  Fix $v \in VT^*$.  We have a compactly generated open subgroup $H_0$ of $G^*_v$ generated by $\tau_a(G^*_a)$ for all $a \in E$ such that $t(\pi_{(T',G)}(a)) = \pi_{(T',G)}(v)$.  In particular, in the expression for $G^*$, every element of $G^*_v$ that is amalgamated with other vertex groups is contained in $H_0$.

Now use Lemma~\ref{lem:tdlc_union} to write $G^*_v$ as a directed union $\bigcup_{i \in I}H_i$, where each $H_i$ is compactly generated and $0$ is the least element of $I$, and recall the normal form theorem for graphs of groups (Theorem~\ref{thm:NormalForm}).  Let $K_i$ be the set of elements of $G^*$ expressible as a reduced word (including the empty word), such that all letters taken from $G^*_v$ belong to $H_i$.  Then $G^* = \bigcup_{i \in I}K_i$; since $G^*$ is compactly generated and $H_0$ is open in $G^*$, in fact $G^* = K_i$ for some $i$.  In particular, every $g \in G^*_v$ is expressible as a reduced word using $F(E)$, $H_i$ and $G^*_{v'}$ for vertices $v' \in VT^*$ other than $v$.  Given the reduction rules for words in a graph of groups, we conclude that $g \in H_i$.  Thus $G^*_v = H_i$, showing that $G^*_v$ is compactly generated as required.  This completes the proof that (i) implies (iii).

Suppose now that (iii) holds.  Then $G^*_v$ is also compactly generated for each vertex $v$ of the inversion-free subdivision of $T'$, and hence we see from the free product decomposition that $G^*$ is compactly generated.  Since $G$ acts on $T'$ with compact kernel it follows that $G$ is also compactly generated.  Thus (iii) implies (i) and hence (i), (ii) and (iii) are equivalent.

For (iv), note that $\Gamma'$ is finite if and only if $G$ has finitely many orbits on $VT' \sqcup AT'$.  Since $G$ is locally compact, we know by Proposition~\ref{prop:locally_compact} that $G^{\propP{}}$ has compact arc stabilisers for all arcs $AT'$; since $G$ is closed the groups $G_a = G \cap G^{\propP{}}_a$ are compact for all $a \in AT'$.  Using Lemma~\ref{lem:StronglyClosedFromCompactArcs} we deduce that, given $v \in V\Gamma'$ and the chosen representative $v^*$ of $\pi\inv_{(T',G)}(v)$, the local action map $\theta: G_{v^*} \rightarrow G(v)$ is a continuous closed surjective map with compact kernel.  It follows that $G_{v^*}$ is compactly generated if and only if $G(v)$ is compactly generated, and in particular, (iii) and (iv) are equivalent.  This completes the proof that (i)--(iv) are equivalent.
\end{proof}

\subsection{Proofs of theorems from the introduction}\label{sec:topology:proofs}

We can now deduce Theorem~\ref{thm:comp_gen+geom_dense} and Corollary~\ref{cor:comp_gen+simple} as special cases of the previous propositions.

For the convenience of the reader, we restate the theorems below.

\renewcommand{\therestatedthm}{\ref{thm:comp_gen+geom_dense}}
\begin{restatedthm}
Let $\Delta = (\Gamma,(X_a),(G(v)))$ be a local action diagram.  Then the following are equivalent:
\begin{enumerate}[(i)]
\item $\Univ(\Delta)$ is compactly generated, locally compact and has geometrically dense action on its associated tree;
\item $\Delta$ is irreducible; $\Gamma$ is finite; and each of the groups $G(v)$ is compactly generated and subdegree-finite.
\end{enumerate}
Moreover, if (i) and (ii) hold, then $\Univ(\Delta)$ is Polish, acting on a countable tree, and all arc stabilisers of $\Univ(\Delta)$ are compact.
\end{restatedthm}

\begin{proof}[Proof of Theorem~\ref{thm:comp_gen+geom_dense}]
Let $T$ be the defining tree of $\Univ(\Delta)$.

Suppose (i) holds.  Then by Proposition~\ref{prop:locally_compact}, since $\Univ(\Delta)$ is locally compact and the action is geometrically dense, all arc stabilisers of $\Univ(\Delta)$ acting on $T$ are compact.  By Proposition~\ref{prop:compact_generation} and the fact that $\Univ(\Delta)$ is compactly generated, we see that $\Gamma$ is finite and $G(v)$ is compactly generated for all $v \in V\Gamma$.  The fact that arc stabilisers are compact implies that each of the groups $G(v)$ is subdegree-finite.  Finally, $\Delta$ is irreducible by Theorem~\ref{thm:invariants}.  Thus (i) implies (ii).

Conversely, suppose (ii) holds.  Since $\Delta$ is irreducible, $(T,\Univ(\Delta))$ is geometrically dense by Theorem~\ref{thm:invariants}.  Since the local actions are all closed and subdegree-finite, $\Univ(\Delta)$ has compact arc stabilisers and hence is locally compact.  We see that $\Univ(\Delta)$ is compactly generated by Proposition~\ref{prop:compact_generation}.  Thus (ii) implies (i).

Suppose (i) and (ii) hold.  Taking $v \in V\Gamma$, then as a permutation group, $G(v) = \overline{G(v)}$ has finitely many orbits and is a compactly generated locally compact group; it follows that $X_v$ is countable.  Since $\Gamma$ is finite we conclude that $T$ is countable.  Thus $\Sym(VT \sqcup AT)$ is Polish, and hence its closed subgroup $\Univ(\Delta)$ is Polish.
\end{proof}

\renewcommand{\therestatedcor}{\ref{cor:comp_gen+simple}}
\begin{restatedcor}
Let $(T,G)$ be a faithful $\propP{}$-closed and strongly closed action on a tree $T$.  Then the following are equivalent:
\begin{enumerate}[(i)]
\item We have $G \in \ms{S}$ and the action does not fix any vertex of $T$.
\item 
    There is a unique smallest invariant subtree $T'$ (possibly equal to $T$) on which $G$ acts faithfully.  Moreover, letting $\Delta = (\Gamma,(X_a),(G(v)))$ be the local action diagram of $(T',G)$, then $\Gamma$ is a finite tree, and each of the groups $G(v)$ is closed, compactly generated, subdegree-finite and generated by point stabilisers, with $G(v) \neq \triv$ for every leaf $v$ of $\Gamma$.
\end{enumerate}
Furthermore, in (\ref{item:lad_description:comp_gen}) the action $(T', G)$ is $(T', \Univ(\Delta))$.
\end{restatedcor}

\begin{proof}[Proof of Corollary~\ref{cor:comp_gen+simple}]
As in the proof of Theorem~\ref{thm:UDelta_simple} we may assume that the action of $G$ on $T$ is parity-preserving, and in particular inversion-free: if $G$ is not parity-preserving, then (i) and (ii) are both false.  (The case $G = \Zb/2\Zb$ is incompatible with (i) since groups in $\ms{S}$ are not discrete, otherwise the argument is the same as in the proof of Theorem~\ref{thm:UDelta_simple}.)

Suppose (i) holds.  By hypothesis the action does not fix a vertex, and the fact that $G$ is nondiscrete and topologically simple rules out actions that are of lineal or focal type
(by Lemma~\ref{lem:Busemann:bis}).  Corollary~\ref{cor:horocyclic} then rules out a horocyclic action.  By process of elimination, the action is of general type, with a unique smallest $G$-invariant subtree $T'$ on which the action of $G$ is faithful and geometrically dense.
In fact the action of $G$ on $T'$ is strongly faithful: the topology of $G$ is already $\sigma$-compact, so there is no coarser locally compact group topology on $G$.  
Thus (via Lemma~\ref{lem:Pclosed_subtree}) we can regard $G$ as a $\propP{}$-closed, hence closed, subgroup of $\Aut(T')$ as a topological group.
Since $G$ is a closed and geometrically dense subgroup of $\Aut(T')$, we immediately have that $G$ is strongly closed on $T'$. Furthermore, the group $G^+$ generated by arc stabilisers is an open normal subgroup, and since $G$ is nondiscrete and topologically simple, $G^+ = G$.  Thus in fact by Tits' Theorem \ref{thm:Tits}, $G$ is abstractly simple.
  We then conclude by Theorem~\ref{thm:UDelta_simple} that $\Delta$ is irreducible, $\Gamma$ is a tree, and each of the groups $G(v)$ is closed and generated by point stabilisers.  By Theorem~\ref{thm:comp_gen+geom_dense}, $\Gamma$ is finite and each of the groups $G(v)$ is compactly generated and subdegree-finite. For every leaf $v$ of $\Gamma$, we have $X_v = X_a$ for the unique arc $a\in o\inv(v)$, and so $X_v$ is a single $G(v)$ orbit.  By Corollary~\ref{cor:DirectlyReadingSimplicity:finite}, for every leaf $v$ of $\Gamma$ we have $|X_v| > 1$ and hence $G(v) \neq \triv$. This completes the proof that (i) implies (ii).

Conversely, suppose that (ii) holds; as in Theorem~\ref{thm:UDelta_simple} we have $(T',G) = (T',\Univ(\Delta))$, so as a topological group we may identify $G$ with the closed subgroup $\Univ(\Delta)$ of $\Aut(T')$. For every leaf $v$ of $\Gamma$, since $G(v) \neq \triv$ we have $|X_v| >1$, and hence $\Delta$ is irreducible via Corollary~\ref{cor:DirectlyReadingSimplicity:finite}.
We deduce via Theorem~\ref{thm:UDelta_simple} that $G$ is a nondiscrete simple group, and it is clear that $G$ does not fix any vertex of $T$.  By Theorem~\ref{thm:comp_gen+geom_dense}, $\Univ(\Delta)$ is compactly generated and locally compact.  Thus $G \in \ms{S}$, showing that (ii) implies (i).
\end{proof}

\section{Examples}\label{sec:examples}

We give three example applications of our theory of local action diagrams: in \S\ref{sec:vertex_transitive} we list all $\propP{}$-closed actions on trees whose degree is at most $5$; in \S\ref{sec:conn_one} we determine all automorphism groups of simple, nontrivial, vertex-transitive graphs with vertex connectivity one; and in \S\ref{sec:new_simple} we give a  technique for combining simple groups in $\mathscr{S}$ to make new simple groups in $\mathscr{S}$.

\subsection{Vertex-transitive actions on trees of small degree}\label{sec:vertex_transitive}

Isomorphism types of local action diagrams can be used to classify isomorphism types of $\propP{}$-closed groups acting on trees, via the one-to-one correspondence we have developed.
This classification is most useful for families of groups acting on trees where the associated local action diagram is `small'.

Let us consider the special case of a vertex-transitive $\propP{}$-closed group $G$ acting on a locally finite tree $T$.  The tree is necessarily regular, of some degree $d$; let us write $T = T_d$, to indicate that $T$ is a regular tree of degree $d$.  In the local action diagram $\Delta = (\Gamma, (X_a),G(v))$ for $(T, G)$, the graph $\Gamma$ has a single vertex $v$; the set $X_v$ has size $d$ and there is a single permutation group $G(v)$, which is defined on $X_v$.  The set $\{X_a \mid a \in A\Gamma\}$ is the partition of $X_v$ into $G(v)$-orbits.  The only remaining piece of information in the local action diagram is the edge-reversal map $r$ on $\Gamma$; since there is only one vertex, this can be any permutation of $A\Gamma$ whose square is the identity.

Thus, up to conjugacy in $\Aut(T_d)$, there are only finitely many vertex-transitive $\propP{}$-closed subgroups of $\Aut(T_d)$.  The relevant conjugacy classes are in one-to-one correspondence with the set $\mc{V}_d$ of equivalence classes of pairs $(H,r)$, where $H$ is a subgroup of $\Sym(d)$ and $r$ is an \defbold{orbit pairing} for $H$, meaning a permutation of the set $H \bs [d]$ of $H$-orbits whose square is the identity.  Here we say two pairs $(H_1,r_1)$ and $(H_2,r_2)$ are equivalent if there is $g \in \Sym(d)$ such that $gH_1g\inv = H_2$ and the map $g': H_1 \bs [d] \rightarrow H_2 \bs [d]$ induced by $g$ satisfies $g'r_1 = r_2g'$.  Write $\Univ(H,r)$ for the subgroup $\Univ(\Delta) \leq \Aut(T_d)$, where $\Delta$ is the local action diagram associated to $(H,r)$ (here $\Univ(H,r)$ should be understood as specified up to conjugacy in $\Aut(T_d)$).  The orbit pairing captures the difference between the number of arc-orbits of $\Univ(H,r)$ and the number of edge-orbits: the arc-orbits of $\Univ(H,r)$ correspond to orbits of $H$, whereas the edge-orbits correspond to orbits of $r$ on $H \bs [d]$.  The orbits of $H$ fixed by $r$ correspond to those arc-orbits of $\Univ(H,r)$ that are closed under the reverse map on $T_d$, in other words, those arcs that are reversed by some element of $\Univ(H,r)$.  We see that $\Univ(H,r)$ fixes an end exactly in the following situation: $H$ has a fixed point and acts transitively on the remaining points, and the orbit pairing is nontrivial.

Let us first deal with the special case that $H$ is a free permutation group, that is, point stabilisers of $H$ are trivial.  In this case, $\Univ(H,r)$ acts freely on the arcs of $T_d$, and it follows by Theorem~\ref{thm:BassSerre} that it can be expressed as a free product of copies of $H$, $\Zb$ and $C_2$ with no amalgamation.  Specifically, writing $K^{\ast n}$ to mean a free product of $n$ copies of $K$, we have
\[
\Univ(H,r) \cong H \ast C^{\ast a}_2 \ast \Zb^{\ast b}
\]
where $a$ is the number of fixed points of $r$ (in other words, the number of reversible arc orbits of $\Univ(H,r)$ on the tree) and $b$ is the number of nontrivial orbits of $r$.  In fact, in this case we see that every group acting on $T_d$ with the same local action diagram as $\Univ(H,r)$ is $\Aut(T_d)$-conjugate.

For small values of $d$, it is feasible to list the conjugacy classes of vertex-transitive $\propP{}$-closed subgroups of $\Aut(T_d)$; we will describe the list for $d \le 5$.  The pairs $(H,r)$ given below should be understood as being taken up to equivalence.  Where it is unambiguous we will indicate the orbit pairing simply by the size of the paired orbits, so for instance $[12,22]$ indicates an orbit pairing where an orbit of size $1$ is paired with an orbit of size $2$, another orbit of size $2$ is paired with a third orbit of size $2$, and all other orbits are fixed by $r$.  
This notation is especially convenient when $d \le 5$, as in this case, for any $H \le \Sym(d)$, all the orbits of $H$ of the same size lie in a single orbit of the normaliser of $H$. For brevity we will write $(H,\id)$ as $(H)$.
Write $S_n := \Sym(n)$, $A_n = \Alt(n)$, $C_n$ for a cyclic group of order $n$ and $D_n$ for a dihedral group of order $n$.  We first recall the conjugacy classes of subgroups of $S_d$.

{\it In $S_0$ and $S_1$:} there is only the trivial group $1$.
{\it In $S_2$:} there are two subgroups, namely $1$ and $S_2$ itself.  Henceforth, $S_2$ is the group of order $2$ acting as a local action at a vertex of degree $2$, whereas $C_2$ (without further decoration) will represent an edge-reversing involution.
{\it In $S_3$:} there are four conjugacy classes of subgroup, namely: $1$; one class of subgroup of order $2$ that we denote as $C_{2+1}$; the alternating group $A_3 = C_3$; and $S_3$ itself.
{\it In $S_4$:} there are 11 conjugacy classes of subgroup, namely: the trivial group $1$; two classes $C^{-}_2$ and $C^+_2$ of subgroup of order $2$ (acting with two and zero fixed points respectively); one class each of cyclic subgroups $C_{3+1}$ and $C_4$ of orders $3$ and $4$; two classes $V^{-}$ and $V^+$ of the Klein $4$-group (the plus sign denoting the regular action, and the minus sign the faithful intransitive action); one class of point stabilisers $S_{3+1}$; one class of the dihedral group $D_8$ of order $8$; the alternating group $A_4$; and the symmetric group itself $S_4$.

{\it In $S_5$:} there are 19 conjugacy classes of subgroup of $S_5$.  These are as follows. There are $11$ classes of subgroup that fix a point (corresponding to conjugacy classes of subgroup of $S_4$). There are three classes of subgroup with orbit partition $(3,2)$, as follows: cyclic group $C_{3+2}$ of order $6$; `twisted $S_3$', viz. $S^*_3 = \langle (1,2,3),(1,2)(4,5) \rangle$; direct product $S_{3+2} = S_3 \times S_2$. There are five classes of transitive subgroup, as follows: cyclic group $C_5$ of order $5$; dihedral group $D_{10}$ of order $10$; general affine group $GA(1,5)$, a group of order $20$; alternating group $A_5$; $S_5$ itself.

The number of conjugacy classes of
Burger--Mozes subgroups 
in $\Aut(T_d)$ is thus $1,1,2,4,11,19$ for $d=0,1,2,3,4,5$ respectively.  However, due to nontrivial orbit pairings, the total number of conjugacy classes of $\propP{}$-closed subgroup in $\Aut(T_d)$ is larger for $d \geq 2$. Indeed, for $d=2,3,4,5$ there are, respectively, a total of $3,6,19,40$ conjugacy classes --- see  Table~\ref{fig:vt_upto4} on page \pageref{fig:vt_upto4} and Table~\ref{fig:vt_5} on page \pageref{fig:vt_5}.

Given $G = \Univ(H,r)$, we can determine the quotient $G/G^+$ as the fundamental group of a graph of groups with trivial edge groups by passing to the inversion-free subdivision of the action, as in Corollary~\ref{cor:quotient_tree:G+}.  In particular, $G/G^+$ can easily be written as a free product with no amalgamation.

In Tables~\ref{fig:vt_upto4} and \ref{fig:vt_5}, a blank entry means a repeat of the previous entry.  Note that the same group may appear several times, but with different actions on the tree.  `l.p.c.' stands for `local prime content', in other words, the primes $p$ such that the $p$-Sylow subgroup of a compact open subgroup of $\Univ(H,r)$ is infinite; in the present context, the local prime content is empty if and only if $\Univ(H,r)$ is discrete.

\begin{table}[!htbp]
\caption{The 30 vertex-transitive $\propP{}$-closed actions on trees of degree $d \le 4$}\label{fig:vt_upto4}%
\begin{tabular}{@{}lllllll@{}}
\toprule
$d$ & Local action & orbit pairing & l.p.c. & fixed end & $G/G^+$  & $G^+$ local action \\ 
\midrule
$0$ & $1$ & $\id$ & $\emptyset$ & N/A & $1$ & $1$  \\
\midrule
$1$ & $1$ & $\id$ & $\emptyset$ & N/A & $C_2$ & $1$  \\
\midrule
$2$ & $1$ & $\id$ & $\emptyset$  & No & $C^{\ast 2}_2$ & $1$ \\
  &   & $[11]$ &  & Yes & $\Zb$ & $1$ \\ 
\midrule
  &  $S_2$ & $\id$ & $\emptyset$ & No & $S_2 \ast C_2$ & $1$ \\
\midrule
$3$ & $1$ & $\id$ & $\emptyset$ & No  & $C^{\ast 3}_2$ & $1$ \\
 & & $[11]$ &  & No & $C_2 \ast \Zb$ & \\
\midrule
  & $S_{2}$ & $\id$ & $\{2\}$ & No & $C^{\ast 2}_2$ & $S_{2}$ \\
  & & $[12]$ & & Yes & $\Zb$ & \\
\midrule
  & $C_3$ & $\id$ & $\emptyset$ & No & $C_3 \ast C_2$ & $1$ \\
\midrule
  & $S_3$ & $\id$ & $\{2\}$ & No & $C_2$ & $S_3$ \\
\midrule
$4$ & $1$ & $\id$ & $\emptyset$ & No  & $C^{\ast 4}_2$ & $1$ \\
 &  & $[11]$ &  & No  & $C^{\ast 2}_2 \ast \Zb$ & \\
 & & $[11,11]$ &  & No  & $\Zb^{\ast 2}$ & \\
\midrule
 & $C^-_2$ & $\id$ & $\{2\}$ & No  & $C^{\ast 3}_2$ & $C^-_2$ \\
 &  & $[11]$ &  & No  & $C_2 \ast \Zb$ &  \\
 &  & $[12]$ &  & No  & $C_2 \ast \Zb$ & \\
\midrule
 & $C^+_2$ & $\id$ & $\emptyset$ & No  & $C^+_2 \ast C^{\ast 2}_2$ &  $1$ \\
 &  & $[22]$ &  & No  & $C^+_2 \ast \Zb$ &  \\
\midrule
 & $C_{3}$ & $\id$ & $\{3\}$ & No  & $C^{\ast 2}_2$ & $C_{3}$ \\
 &  & $[13]$ &  & Yes  & $\Zb$ & \\
 \midrule
 & $C_{4}$ & $\id$ & $\emptyset$ & No  & $C_4 \ast C_2$ & $1$ \\ 
 \midrule
 & $V^-$ & $\id$ & $\{2\}$ & No  & $C^{\ast 2}_2$ & $V^-$ \\
 &  & $[22]$ &  & No  & $\Zb$ & \\ 
 \midrule
 & $V^+$ & $\id$ & $\emptyset$ & No  & $V^+ \ast C_2$ & $1$ \\ 
 \midrule
 & $S_{3}$ & $\id$ & $\{2,3\}$ & No  & $C^{\ast 2}_2$ & $S_{3}$  \\
 &  & $[13]$ & & Yes  & $\Zb$ & \\ 
 \midrule
 & $D_8$ & $\id$ & $\{2\}$ & No  & $S_2 \ast C_2$ & $V^-$ \\ 
 \midrule
 & $A_4$ & $\id$ & $\{3\}$ & No  & $C_2$ & $A_4$ \\ 
 \midrule
 & $S_4$ & $\id$ & $\{2,3\}$ & No  & $C_2$ & $S_4$ \\
\botrule
\end{tabular}
\end{table}

\begin{table}[!htbp]
\caption{The 40 vertex-transitive $\propP{}$-closed actions on trees of degree $5$}\label{fig:vt_5}%
\begin{tabular}{@{}llllll@{}}
\toprule
Local action & orbit pairing & l.p.c. & fixed end & $G/G^+$  & $G^+$ local action \\ 
\midrule
$1$ & $\id$ & $\emptyset$ & No  &  $C^{\ast 5}_2$ & $1$ \\ 
 & $[11]$ &  & No  & $C^{\ast 3}_2 \ast \Zb$ &   \\ 
 & $[11,11]$ &  & No  & $C_2 \ast \Zb^{\ast 2}$ &  \\ \midrule
$C^-_2$ & $\id$ & $\{2\}$ & No  & $C^{\ast 4}_2$ & $C^-_2$ \\ 
 & $[11]$ &  & No  & $C^{\ast 2}_2 \ast \Zb$ &  \\ 
 & $[12]$ &  & No  & $C^{\ast 2}_2 \ast \Zb$ &  \\ 
 & $[11,12]$ &  & No  & $C_2 \ast \Zb^{\ast 2}$ &  \\ \midrule
 $C^+_2$ & $\id$ & $\{2\}$ & No  & $C^{\ast 3}_2$ &  $C^+_2$ \\ 
 & $[12]$ & & No  & $C_2 \ast \Zb$ &   \\ 
 & $[22]$ &  & No  & $C_2 \ast \Zb$ &   \\ \midrule
$C_{3}$ & $\id$ & $\{3\}$ & No  & $C^{\ast 3}_2$ & $C_{3}$ \\ 
 & $[11]$ &  & No  & $C_2 \ast \Zb$ &  \\ 
 & $[13]$ &  & No  & $C_2 \ast \Zb$ &  \\ \midrule
$C_{4}$ & $\id$ & $\{2\}$ & No  & $C^{\ast 2}_2$ & $C_{4}$ \\ 
& $[14]$ &  & Yes  & $\Zb$ &  \\ \midrule
$V^-$ & $\id$ & $\{2\}$ & No  & $C^{\ast 3}_2$ & $V^-$  \\ 
& $[12]$ &  & No  & $C_2 \ast \Zb$ &  \\ 
 & $[22]$ &  & No  & $C_2 \ast \Zb$ &  \\ \midrule
$V^+$ & $\id$ & $\{2\}$ & No  & $C^{\ast 2}_2$ & $V^+$ \\ 
 & $[14]$ &  & Yes  & $\Zb$ &  \\ \midrule
$S_{3}$ & $\id$ & $\{2,3\}$ & No  & $C^{\ast 3}_2$ & $S_{3}$ \\ 
 & $[11]$ &  & No  & $C_2 \ast \Zb$ &  \\ 
 & $[13]$ &   & No  & $C_2 \ast \Zb$ &  \\ \midrule
$D_8$ & $\id$ & $\{2\}$ & No  & $C^{\ast 2}_2$ & $D_8$  \\ 
 & $[14]$ &  & Yes  & $\Zb$ & \\ \midrule
$A_{4}$ & $\id$ & $\{2,3\}$ & No  & $C^{\ast 2}_2$ & $A_{4}$ \\ 
 & $[14]$ &  & Yes  & $\Zb$ &  \\ \midrule
$S_{4}$ & $\id$ & $\{2,3\}$ & No  & $C^{\ast 2}_2$ & $S_{4}$ \\ 
 & $[14]$ &  & Yes  & $\Zb$ &  \\ \midrule
$C_{3+2}$ & $\id$ & $\{2,3\}$ & No  & $C^{\ast 2}_2$ & $C_{3+2}$ \\ 
 & $[23]$ &  & No  & $\Zb$ &  \\ \midrule
$S^*_{3}$ & $\id$ & $\{2,3\}$ & No & $C^{\ast 2}_2$ & $S^*_{3}$  \\ 
 & $[23]$ &  & No  & $\Zb$ &  \\ \midrule
$S_{3+2}$ & $\id$ & $\{2,3\}$ & No  & $C^{\ast 2}_2$ & $S_{3+2}$  \\ 
 & $[23]$ &  & No  & $\Zb$ & \\ \midrule
$C_5$ & $\id$ & $\emptyset$ & No  & $C_5 \ast \Zb$ & $1$ \\ \midrule
$D_{10}$ & $\id$ & $\{2\}$ & No  & $C_2$ & $D_{10}$ \\ \midrule
$GA(1,5)$ & $\id$ & $\{2\}$ & No  & $C_2$ & $GA(1,5)$ \\ \midrule
$A_5$ & $\id$ & $\{2,3\}$ & No  & $C_2$ & $A_5$  \\ \midrule
$S_5$ & $\id$ & $\{2,3\}$ & No  & $C_2$ & $S_5$  \\
\botrule
\end{tabular}
\end{table}

\

We should remember that the list of groups we have obtained so far is up to equivalence of action on the tree, not up to isomorphism as groups.  For example, the group $\Zb \ast C_2$ appears as a vertex-transitive $\propP{}$-closed subgroup of both $\Aut(T_3)$ and $\Aut(T_4)$, but clearly the actions are not equivalent.  It is not clear what group isomorphisms could exist between the nondiscrete groups in the list (that is, all the groups listed such that the local action is not free).

Recall the class $\ms{S}$ of \tdlc groups that are compactly generated, nondiscrete and topologically simple.  If $G$ is a vertex-transitive $\propP{}$-closed subgroup of $\Aut(T_d)$ for $d \in \Nb$, then $G^+$ is compactly generated if and only if $G^+ \backslash T$ is finite, which means that $G/G^+$ must be finite; $G^+$ is nondiscrete if and only if the local prime content is nonempty.
Applying Corollary~\ref{cor:comp_gen+simple}, we see that of the 70 entries in Tables~\ref{fig:vt_upto4} and \ref{fig:vt_5}, only seven have $G^+ \in \ms{S}$, namely the Burger--Mozes groups $\Univ(F)$ with transitive local action for $F \in \{S_3,A_4,S_4,D_{10},GA(1,5),A_5,S_5\}$.  However, in a further 36 cases (those with nontrivial local prime content where $G$ does not fix an end), $G^+$ is nondiscrete and simple, but fails to be compactly generated.  The latter simple groups have a complicated structure in general, which may merit further investigation.  For instance, by \cite[Theorem~1.8]{CapraceWesolek}, in every nondiscrete Burger--Mozes group there is a compactly generated closed subgroup $K$ (where without loss of generality $K \le G^+$), and a discrete normal subgroup $D$ of $K$, such that $K/D \in \ms{S}$.  We do not know if any of these 36 noncompactly generated simple groups $G^+$ are isomorphic to one another as abstract or topological groups.

One motivation for studying vertex-transitive groups acting on trees of small degree is to understand compactly generated \tdlc groups in terms of their degree.  The \defbold{degree} $\deg(G)$ of a compactly generated \tdlc group $G$ is the smallest degree of a Cayley--Abels graph for $G$.
The notion of a Cayley--Abels graph was introduced by Abels in \cite{Abels}, and the degree of compactly generated \tdlc groups is investigated in \cite{DegreeCayleyAbels}.
The degree $0$ groups are the compact groups and the degree $2$ groups are the compact-by-cyclic groups, but even for degree $3$ the structure is not well-understood, except that the degree must be larger than the maximum of the local prime content.  What can be said in general, given a group $G$ acting with kernel $K$ on a Cayley--Abels graph $\Gamma$ of minimal degree $d$, is that the action lifts to an action of a group $\widetilde{G}$ acting vertex-transitively on $T_d$, with an associated homomorphism $\theta: \widetilde{G} \rightarrow G/K$ with discrete kernel.  We can then consider the $\propP{}$-closure $\widetilde{G}^{\propP{}}$ as a first step towards understanding $\widetilde{G}$ and the original group $G$.  Both the local action and the orbit pairing for $\widetilde{G}^{\propP{}}$ come from the action of $G$ on $\Gamma$.

\subsection{Simple vertex-transitive graphs with vertex connectivity one}\label{sec:conn_one}

We illustrate how to use our theory of local action diagrams to describe all automorphism groups of simple vertex-transitive graphs with vertex connectivity one.

Let $\Gamma$ be a simple connected graph; we shall call $\Gamma$ \defbold{trivial} if it consists of a single vertex and no edges. Recall that $\Gamma$ has \defbold{vertex connectivity one} (also \defbold{connectivity one} and \defbold{$1$-connected}) if there is a vertex $v$ in $\Gamma$ such that the induced graph $\Gamma \smallsetminus \{v\}$ (which arises by removing the vertex $v$ from $\Gamma$, together with all edges in $\Gamma$ containing $v$) is not connected.
Such a vertex $v$ is called a \defbold{cut vertex} of $\Gamma$. A connected graph with no cut vertices is \defbold{$2$-connected} (here we  consider the complete graph on two vertices $K_2$ to be $2$-connected, but a single vertex not to be $2$-connected). 
If $\Gamma$ has connectivity one, then the maximal $2$-connected subgraphs of $\Gamma$ are called the \defbold{lobes} (also sometimes called \defbold{blocks}) of $\Gamma$. For example, the lobes of the infinite $3$-regular tree $T_3$ are pairs of adjacent vertices together with the edge between them; each lobe of $T_3$ is thus isomorphic to $K_2$.

Vertex-transitive graphs with connectivity one were completely described by Heinz A.~Jung and Mark E.~Watkins in \cite{JungWatkins}. We adopt and extend their notation here. 
Suppose $\Gamma$ has connectivity one and let $L$ be the set of lobes of $\Gamma$, with $\{L_i : i \in I\}$ denoting the set of  isomorphism classes of $L$.
In each $L_i$ we fix some distinguished element $\hat{\Lambda}_i$ and for each $\Lambda \in L_i$ we fix some isomorphism $\theta_{\Lambda} : \hat{\Lambda}_i \rightarrow \Lambda$. For each $i \in I$ let $\{\hat{\Lambda}_i^{(j)} : j \in J_i\}$ be a decomposition of $V\hat{\Lambda}_i$ into $\Aut(\hat{\Lambda}_i)$-orbits. Now for any $\Lambda \in L_i$ and $j \in J_i$, let $\Lambda^{(j)}:= \theta_\Lambda (\hat{\Lambda}_i^{(j)})$.
If $\Lambda, \Lambda' \in L_i$ with $\tau : \Lambda \rightarrow \Lambda'$ an isomorphism, then 
for each $j \in J_i$ we have $\tau(\Lambda^{(j)}) = \Lambda'^{(j)}$.
Thus, for each $i \in I$ and $\Lambda \in L_i$ the set $\{\Lambda^{(j)} : j \in J_i\}$ is a decomposition of $V\Lambda$ into $\Aut(\Lambda)$-orbits.

One can easily verify that any two pairs $(v, \Lambda), (v', \Lambda')$, consisting of vertices $v, v' \in V\Gamma$ and lobes $\Lambda, \Lambda'$ of $\Gamma$ with $v \in V\Lambda$ and $v' \in V\Lambda'$, are isomorphic as rooted graphs if and only if there exists $i \in I$ and $j \in J_i$ such that $\Lambda, \Lambda' \in L_i$ and $v \in \Lambda^{(j)}$ and $v' \in \Lambda'^{(j)}$. 

 Now for any vertex $v \in V\Gamma$ and $i \in I$ and $j \in J_i$, define $L_i^{(j)}(v):=\{\Lambda \in L_i : v \in \Lambda^{(j)}\}$ and
$AL_i^{(j)}(v) := \{A\Lambda : \Lambda \in L_i^{(j)}(v)\}$ and $m_i^{(j)}(v) := |L_i^{(j)}(v)|$.

In \cite[Lemma 3.1, Theorem 3.2 and proof]{JungWatkins} Jung and Watkins show the following.

\begin{thm} \label{JungWatkins} Let $\Gamma$ be a simple connected graph with vertex connectivity one. Then:
\begin{enumerate}[(i)]
\item
	$\Gamma$ is vertex-transitive if and only if the functions $m_i^{(j)}$ are constant on $V\Gamma$.
\item
	If $\Gamma$ is vertex-transitive, and $v_1, v_2$ (resp.~$\Lambda_1, \Lambda_2$) are vertices (resp.~lobes) of $\Gamma$ with $v_i \in \Lambda_i$ for $i=1,2$, then any isomorphism of rooted graphs between $(v_1, \Lambda_1)$ and $(v_2, \Lambda_2)$ extends to an automorphism of $\Gamma$.
\end{enumerate}
\end{thm} 

Now suppose that $\Gamma$ is nontrivial, vertex-transitive with connectivity one.
Such a graph is infinite with each vertex $v$ lying in at least two lobes.
Notice that 
$v$ lies in precisely $\sum_{i \in I, j \in J_i} m_i^{(j)}(v)$ many lobes. 
The graph $\Gamma$ is tree-like, and this can be seen  by considering the \defbold{block-cut-vertex tree} of $\Gamma$ (also called a \defbold{structure tree}) $T$, defined as follows: $VT := V\Gamma \sqcup L$ and $ET := \{\{v, \Lambda\} : v \in V\Gamma \text{ and } \Lambda \in L \text{ and } v \in V\Lambda\}$ (it is easy to see that such a graph is indeed a tree). We call elements in $V\Gamma \subseteq VT$ (resp.~$L \subseteq VT$) the \defbold{graph vertices} (resp.~\defbold{lobe vertices}) of $T$.
Every lobe vertex lies between two graph vertices, so the automorphism group $\Aut(\Gamma)$ acts faithfully on $T$. We can use this action on $T$ to describe $\Aut(\Gamma)$ in the language of local action diagrams.

Let $T$ be the structure tree of $\Gamma$
 and let $G$ be the subgroup of $\Aut(T)$ induced by the action of $\Aut(\Gamma)$ on its structure tree $T$. In (A\ref{ItemConnectivityOne:AOne})--(A\ref{ItemConnectivityOne:AFour}) below we determine the associated local action diagram $\Delta$ of $G$ in its action on $T$; we then equip $T$ with the structure of a $\Delta$-tree in the usual way.
\begin{enumerate}[({A}1)]
\item \label{ItemConnectivityOne:AOne}
	The underlying graph $\Gamma'$ of the local action diagram is $G \backslash T$. By Theorem~\ref{JungWatkins}, any two isomorphic lobes lie in the same $G$-orbit. Thus the vertex set of $G \backslash T$ is $\{v^*\} \cup \{L_i : i \in I\}$ where $v^*$ is some representative element from $V\Gamma$.  For convenience, write $m_{i,j} := m_i^{(j)}(v^*)$.
\item
	Again by Theorem~\ref{JungWatkins}, the undirected edges in $\Gamma'$ correspond to orbits of rooted graphs $(v, \Lambda)$, where $\Lambda \in L$ and $v \in V\Lambda \subseteq V\Gamma$. As noted above, these orbits correspond precisely to isomorphism classes of rooted graphs, with $(v, \Lambda)$ and $(v', \Lambda')$ being isomorphic if and only if there exists $i \in I$ and $j \in J_i$ such that $\Lambda, \Lambda' \in L_i$ and $v \in \Lambda^{(j)}$ and $v' \in \Lambda'^{(j)}$. Choose a representative vertex $v_i^{(j)}$ from each orbit $\hat{\Lambda}_i^{(j)}$. The undirected edges in $\Gamma'$ between $v^*$ and each $L_i$ thus correspond to elements of the set $\{(v_i^{(j)}, \hat{\Lambda}_i) : j \in J_i\}$.
\item
	For all $\Gamma'$ arcs $a \in o\inv(v^*)$ there is some $i \in I$ and $j \in J_i$ such that $a$ corresponds to $(v_i^{(j)}, \hat{\Lambda}_i)$. The colour set $X_a$ is now easy to picture: it consists of all lobes $\Lambda$ in $\Gamma$ that contain $v^*$ and lie in $L_i$ such that $v^*$ lies in the orbit $\Lambda^{(j)}$. In our notation, $X_a = L_i^{(j)}(v^*)$.
	For the reverse arcs, the colour set is $X_{L_i} = V\hat{\Lambda}_i$, naturally partitioned according to the $\Aut(\hat{\Lambda}_i)$-orbits on $V\hat{\Lambda}_i$.
\item \label{ItemConnectivityOne:AFour}
	We now describe the vertex groups in $\Gamma'$. Clearly, for each lobe vertex $L_i$ in $V\Gamma'$, we have $X_{L_i} = V\hat{\Lambda}_i$ with $G(L_i) = \Aut(\hat{\Lambda}_i)$. So, it remains for us to determine $G(v^*)$ for the single graph vertex $v^*$ in $\Gamma'$. 
For each $i \in I$ and $j \in J_i$, observe that $G_{v^*}$ induces the symmetric group $S_{m_{i, j}}$ on $L_i^{(j)}(v^*)$.
As $G$ is the automorphism group of $\Gamma$, there are no further restrictions on $G_{v^*}$, and thus $G_{v^*}$ induces $S_{m_{i,j}}$ on $X_a$.
For arcs $a, b$ the actions of $G_{v^*}$ on disjoint subsets $X_a, X_b$ of $X_{v^*}$ are independent, and so $G(v^*)$ is the direct product,
	\[\prod_{i \in I, j \in J_i} S_{m_{i,j}},\]
acting on $X_{v^*} = \bigsqcup X_a$, where the $(i,j)$-th component in the above product acts on $X_a$ for the arc $a$ corresponding to $(v_i^{(j)}, \hat{\Lambda}_i)$.
\end{enumerate}

We have thus described the local action diagram $\Delta(T, G)$ of $G$, and hence also $\mathbf{T}$, the associated $\Delta$-tree structure on $T$. By Theorem~\ref{thm:G_to_U}, the universal group 
$\Univ(\mathbf{T},(G(v)))$ is the $\propP{}$-closure of $G$. Being the group induced by $ \Aut(\Gamma)$ acting on $T$, we have that $G$ is $\propP{}$-closed. Hence $G = \Univ(\mathbf{T},(G(v)))$. The group $\Aut(\Gamma)$ is thus the restriction of $\Univ(\mathbf{T},(G(v)))$ to the set of graph vertices in $\mathbf{T}$; that is, $\Univ(\mathbf{T},(G(v))) \big|_{V\Gamma}$.\\

Now we reverse this process, and show how any local action diagram $\Delta''$ with the above structure gives rise to the automorphism group of a nontrivial, vertex-transitive graph with connectivity one. We construct $\Delta''$ as follows.
\begin{enumerate}[({B}1)]
\item \label{ItemConnectivityOne:BZero}
Choose a nonempty set $\{\hat{\Lambda}_i : i \in I\}$ of pairwise non-isomorphic, $2$-connected graphs with $\{\hat{\Lambda}_i^{(j)} : j \in J_i\}$ a decomposition of $V\hat{\Lambda}_i$ into $\Aut(\hat{\Lambda}_i)$-orbits, with $v_i^{(j)}$ a representative vertex chosen from each orbit. For each $i \in I$ and $j \in J_i$ choose some cardinal $m_{i,j} \geq 1$,
such that if $|I| = 1$ (say $I = \{i\}$) and $|J_i| = 1$ (say $J_i = \{j\}$) then $m_{i,j} > 1$; this latter condition is to guarantee that in the connectivity one graph we construct, every vertex lies in at least two lobes.
\item \label{ItemConnectivityOne:BOne}
	The underlying graph $\Gamma''$ of the local action diagram will be a (possibly multi-edged) star graph, consisting of a central vertex $v^*$ together with $|I|$ other vertices labelled $L_i$ for $i \in I$.
\item
	For each $i \in I$ we draw some undirected edges between $v^*$ and the vertex labelled $L_i$; we draw precisely one edge for each rooted graph $(v_i^{(j)}, \hat{\Lambda}_i)$, $j \in J_i$. Recall that undirected edges consist of a pair of arcs, one in each direction; the two arcs in the edge corresponding to $(v_i^{(j)}, \hat{\Lambda}_i)$ are labelled $(v_i^{(j)}, \hat{\Lambda}_i)$.
\item
	For an arc $a \in o\inv(v^*)$ labelled $(v_i^{(j)}, \hat{\Lambda}_i)$, choose a colour set $C_a$ of cardinality $m_{i,j}$. These colour sets should be disjoint, so that if an arc $b \in o\inv(v^*)$	differs from $a$ then $C_a \cap C_b = \emptyset$. Set $X_{v^*}$ to be the union over all such arcs $a$.
	
	For the reverse arcs, we choose colours from the colour set $X_{L_i} = V\hat{\Lambda}_i$. Recall that the arcs in $o\inv(L_i)$ are labelled $(v_i^{(j)}, \hat{\Lambda}_i)$, $j \in J_i$, so the arcs in $o\inv(L_i)$ are in bijective correspondence with the orbits of $\Aut(\hat{\Lambda}_i)$ on $V \hat{\Lambda}_i$. The colour set of an arc $a \in o\inv(L_i)$ is taken to be the vertices in the $\Aut(\hat{\Lambda}_i)$-orbit corresponding to $a$.
\item \label{ItemConnectivityOne:BFour}
	For the vertex $L_i$ of $\Gamma''$, let $G(L_i)$ be $\Aut(\hat{\Lambda}_i)$, viewed as a subgroup of $\Sym(V\hat{\Lambda}_i)$.
 Let $G(v^*)$ be the direct product $\prod_{i \in I, j \in J_i} S_{m_{i,j}} \leq \Sym(X_{v^*}) = \Sym(\bigsqcup X_a)$, 
where each component acts on $X_a$ for the arc $a \in o\inv(v^*)$ corresponding to $(v_i^{(j)}, \hat{\Lambda}_i)$. Note that $G(v^*)$ is closed in the permutation topology on $\Sym(X_{v^*})$.
\end{enumerate}

We now have a local action diagram $\Delta''= (\Gamma'', (X_a), (G(v))$. Let $\mathbf{T''}$ be a $\Delta''$-tree, and let $U := \Univ(\mathbf{T''},(G(v)))$. We will show that $U$ arises from a connectivity one graph.

We construct a simple graph $\Gamma$ with connectivity one as follows (see \cite{WillisFree} for a more general free product construction along similar lines).
 Let $\Gamma_0$ be the graph consisting of a single vertex $v^*$. For each $i \in I$ and $j \in J_i$, glue $m_{i,j}$ copies of $\hat{\Lambda}_i$ to $v^*$ by associating $v^*$ with the vertex $v_i^{(j)} \in V\hat{\Lambda}_i$. Call the graph thus obtained $\Gamma_1$.
For $n>1$ construct $\Gamma_n$ by, for each vertex $v \in V\Gamma_{n-1} \smallsetminus V\Gamma_{n-2}$, gluing lobes to $v$ in the following way. Notice that there is a single lobe of $\Gamma_{n-1}$ that contains $v$. This lobe is isomorphic to $\hat{\Lambda}_{i'}$ for some $i' \in I$, with $v$ lying in some $\Aut(\hat{\Lambda}_{i'})$-orbit $\hat{\Lambda}_{i'}^{(j')}$. We glue $m_{i',j'} - 1$ new copies of $\hat{\Lambda}_{i'}$ to $v$ by associating $v$ with the vertex $v_{i'}^{(j')} \in V\hat{\Lambda}_{i'}$. Then, for all remaining $i, j$, glue $m_{i,j}$ new copies of $\hat{\Lambda}_{i}$ to $v$ by associating $v$ with the vertex $v_{i}^{(j)} \in V\hat{\Lambda}_{i}$. The limit of this process as $n \rightarrow \infty$ is a graph $\Gamma$ which is simple, nontrivial and has connectivity one; any lobe of $\Gamma$ is isomorphic to $\hat{\Lambda}_i$ for some $i \in I$.

Let $T$ be the block-cut-vertex tree of $\Gamma$ and let $G$ be the subgroup of $\Aut(T)$ induced by the action of $\Aut(\Gamma)$ on $T$. We now apply our prior analysis (A\ref{ItemConnectivityOne:AOne})--(A\ref{ItemConnectivityOne:AFour}) to obtain the associated local action diagram $\Delta(T,G)$ of $G$ and the associated $\Delta$-tree structure $\mathbf{T}$ on $T$. By construction, $\Delta''$ and $\Delta$ are the same, and so by Lemma~\ref{lem:sameT} we can consider $\mathbf{T}$ and $\mathbf{T''}$ to be identical. In particular, the set $V\Gamma \subseteq VT$ of graph vertices in $\mathbf{T}$ corresponds to $\pi^{-1}(v^*)$ in $\mathbf{T''}$.
We thus have $U = \Univ(\mathbf{T''},(G(v))) = \Univ(\mathbf{T},(G(v))) = G$, and so $\Aut(\Gamma) = \Univ(\mathbf{T''},(G(v))) \big|_{\pi^{-1}(v^*)}$.\\

We summarise the above in the following theorem.

\begin{thm}\label{thm:connectivity_one} Automorphism groups of simple, nontrivial, vertex-transitive graphs  $\Gamma$ with vertex connectivity one are precisely the groups $\Univ(\mathbf{T''},(G(v))) \big|_{\pi^{-1}(v^*)}$ where $\mathbf{T''}$ is a $\Delta''$-tree constructed via (B\ref{ItemConnectivityOne:BZero})--(B\ref{ItemConnectivityOne:BFour}) above.
\end{thm}

\subsection{Building more compactly generated simple groups}\label{sec:new_simple}

Recall that Corollary~\ref{cor:comp_gen+simple} provides  straightforward conditions that allow us to construct a local action diagram that will yield a group in $\ms{S}$. In particular, for each leaf of the local action diagram, we must provide a transitive subdegree-finite permutation group that is nontrivial, compactly generated and generated by point stabilisers.  In fact, given a compactly generated \tdlc group $G$, then every compact open subgroup $U$ of $G$ will give rise to a closed transitive subdegree-finite action of $G$ on $G/U$.  The only difficulty is in making sure that the action of $G$ on $G/U$ is generated by point stabilisers.  However, this latter condition will hold, for example, whenever $U \neq \triv$ and $G$ does not admit any proper discrete quotient. Indeed, if $H$ is the closed normal subgroup of $G$ generated by point stabilisers in $G$ then $G/H$ in its action on the set of $H$-orbits is semi-regular and thus discrete.

As an illustration, we prove Theorem~\ref{thm:combination_simple}, which shows that we can `combine' finitely many groups in $\ms{S}$ (chosen from $\ms{S}$ arbitrarily) to make another group in $\ms{S}$ of the form $\Univ(\Delta)$ for a suitably chosen $\Delta$.

\begin{proof}[Proof of Theorem~\ref{thm:combination_simple}]
Define $\Delta = (\Gamma,(X_a),(G(v)))$ as follows:
$\Gamma$ is a star with central vertex $v_0$, leaves $v_1,\dots,v_n$, and exactly one arc $a_i$ from $v_0$ to $v_i$ for $1 \le i \le n$.
If $n=1$ we set $G(v_0) = \Sym(3)$ acting on $X_{a_1} = \{1,2,3\}$.  If $n \ge 2$ we set $X_{a_i} = \{i\}$ for $1 \le i \le n$ and let $G(v_0)$ be the trivial group.

For $1 \le i \le n$, let $G(v_i)$ be $G_i$ acting by left translation on the left coset space $G_i/U_i =: X_{\ol{a_i}}$. 
It is now easy to see that we have defined a valid local action diagram, and taking the action of $\Univ(\Delta)$ on its defining tree, all the conditions of Corollary~\ref{cor:comp_gen+simple}(ii) are immediately apparent.  Thus $\Univ(\Delta) \in \ms{S}$.

Let $(T,\pi,\mc{L})$ be the $\Delta$-tree defining $\Univ(\Delta)$.  For $1 \le i \le n$ let $v^*_i \in \pi\inv(v_i)$ and let $O_i = \Univ(\Delta)_{v^*_i}$.  Then the action of $O_i$ on $o\inv(v^*_i)$ is exactly $G(v_i)$, which is isomorphic to $G_i$;  indeed, by Proposition~\ref{prop:semidirect}, the action homomorphism of $O_i$ on $o\inv(v^*_i)$ splits, so $O_i \cong K_i \rtimes G_i$ where $K_i$ is the kernel of the action.  In turn, $K_i$ fixes an arc, hence is compact by Proposition~\ref{prop:locally_compact}.
\end{proof}

A similar construction of groups in $\ms{S}$, this time as automorphism groups of locally finite graphs of connectivity one, can be extracted from Section~\ref{sec:conn_one}.

\begin{prop}
Let $n$ be a positive integer and let $\hat{\Lambda}_1,\dots,\hat{\Lambda}_n$ be pairwise non-isomorphic locally finite simple graphs, each of which is $2$-connected.  Suppose that for $1 \le i \le n$, the automorphism group $\Aut(\hat{\Lambda}_i)$ is vertex-transitive and generated by vertex stabilisers.  Let $m_1,\dots,m_n$ be positive integers, where $m_1 \ge 3$ if $n=1$.  Form the simple vertex-transitive graph $\Gamma$ of connectivity one as in (B\ref{ItemConnectivityOne:BZero})--(B\ref{ItemConnectivityOne:BFour}) of Section~\ref{sec:conn_one}, so that $\Gamma$ has lobe types $\hat{\Lambda}_1,\dots,\hat{\Lambda}_n$ and each vertex belongs to $m_i$ lobes of the $i$-th type.  Then $\Gamma$ is locally finite and $\Aut(\Gamma) \in \ms{S}$.
\end{prop}

\begin{proof}
By Theorem~\ref{thm:connectivity_one}, as a topological group we can identify $\Aut(\Gamma)$ with $\Univ(\mathbf{T''},(G(v)))$, where $\mathbf{T''}$ is a $\Delta''$-tree and $\Delta'' = (\Gamma'', (X_a), (G(v))$ is constructed as in (B\ref{ItemConnectivityOne:BZero})--(B\ref{ItemConnectivityOne:BFour}) of Section~\ref{sec:conn_one}.  It is clear from the construction that $\Gamma$ is locally finite, specifically of degree $\sum^n_{i=1}m_i\deg(\hat{\Lambda}_i)$.  It remains to check that $\Delta''$ satisfies the conditions of Corollary~\ref{cor:comp_gen+simple}(ii).

Now $\Gamma''$ has $n+1$ vertices $v^*,v_1,\dots,v_n$, where $v^*$ represents the vertices of $\Gamma$ and for $1 \le i \le n$, the vertex $v_i$ corresponds to the $i$-th type of lobe.  For $1 \le i \le n$ we have $G(v_i) = \Aut(\hat{\Lambda}_i)$ acting on $V\hat{\Lambda}_i$; since this is a transitive action, there is exactly one undirected edge between $v^*$ and $v_i$.  There are no other edges in $\Gamma''$, so $\Gamma''$ is a finite tree.  The fact that $\hat{\Lambda}_i$ is connected and locally finite ensures that $G(v_i)$ has compact point stabilisers, and the fact that $G(v_i)$ acts transitively ensures it is nontrivial and compactly generated (see for example \cite[Theorem~1$^+$]{KronMoller}).  Meanwhile, $G(v^*)$ is the permutation group acting on $\sum^n_{i=1}m_i$ points as a product of symmetric groups, acting independently on orbits of size $m_i$: this permutation group is generated by point stabilisers with the sole exception of the case $n=1$ and $m_1=2$, which we have excluded by hypothesis.  Thus the local actions of $\Delta''$ are all compactly generated, subdegree-finite and generated by point stabilisers.  If $n=1$ the leaves of $\Gamma''$ are $v^*$ and $v_1$: we have ensured $G(v^*)$ is nontrivial in this case by setting $m_1 \ge 3$.  Otherwise the leaves are $v_1,\dots,v_n$, all of which have nontrivial associated local actions.  We have now verified all the conditions of Corollary~\ref{cor:comp_gen+simple}(ii), and hence $\Univ(\mathbf{T''},(G(v)))$ is as in Corollary~\ref{cor:comp_gen+simple}(i); in particular, $\Aut(\Gamma) \in \ms{S}$.
\end{proof}

\section{Questions} \label{questions}

In this section we highlight some questions that might shape further research on local action diagrams.

The theory of local action diagrams developed in this paper is a classification of closed groups acting on trees with property $\propP{}$. It is plausible that one could develop a companion theory for closed groups acting on trees with property $\propP{k}$, using a modified version of the local action action diagram that is built around $k$-arcs rather than arcs.

\begin{que} Is it possible to classify closed groups of automorphisms of trees with property $\propP{k}$ via a modified version of the local action diagram?
\end{que}

Let $T_d$ be the regular tree of finite degree $d$ and let $N_d$ be the number of conjugacy classes of vertex-transitive $\propP{}$-closed subgroups of $\Aut(T_d)$.  As discussed in Section~\ref{sec:vertex_transitive}, $N_d$ is the number of $\Sym(d)$-conjugacy classes of pairs $(H,r)$, where $H$ is a subgroup of $\Sym(d)$ and $r$ is a function on the set of $H$-orbits on $[d]$ such that $r^2 = \mathrm{id}$.  For comparison, the number of conjugacy classes of Burger--Mozes groups of degree $d$ is simply the number $C_d := |\Sym(d)\backslash\mathrm{Sub}(\Sym(d))|$ of conjugacy classes of subgroups of $\Sym(d)$.  The problem of enumerating the vertex-transitive $\propP{}$-closed actions on locally finite trees is thus reduced to a fairly natural problem in finite permutation groups.  In Appendix~\ref{app:vt_1cl_gap} we give a GAP (\cite{GAP4}) implementation due to S.~Tornier that can be used to perform these enumerations for small $d$, and it would be very interesting to continue in this vein, both for larger values of $d$ and for biregular trees, to build a database of $\propP{}$-closed actions.  Given the limitations on what is practical to compute exactly, it would also be interesting to have some bounds on the growth rates of $N_d$ and $C_d$ coming from the theory of finite permutation groups.

\begin{que}What are the asymptotics of the number $N_d$ as a function of $d$?  How much faster does it grow than $C_d$?\end{que}

In this paper, we provide a method for determining the local action diagrams of subgroups of $\propP{}$-closed groups, via the notion of local subaction diagrams. However, as noted in Remark~\ref{rem:limitations_of_subaction}, this description falls short of a classification in two respects: (i) in general determining all local subaction diagrams of a given local action diagram appears to be intractable, and (ii) local subaction diagrams only permit the classification of $\propP{}$-closed subgroups of a $\propP{}$-closed action $(T,G)$ up to conjugacy in $\Aut(T)$, rather than up to conjugacy in $G$.

\begin{que}\label{que:subgroups} Is it possible to extend the idea of a local subaction diagram to address limitations (i) and (ii)? Such a structure will need to carry more information than a local subaction diagram, but possibly not significantly more.
\end{que}

Some insight into Question~\ref{que:subgroups} could come from considering the following question, which cannot be meaningfully addressed using local subaction diagrams alone.

\begin{que}\label{que:open_subgroups} Suppose $G \leq \Aut(T)$ is $\propP{}$-closed. Can one classify the open subgroups of $G$ that act with translation using only the local action diagram of $(T,G)$?
\end{que}

A general theme of interesting research would be to continue to identify global properties of $G \leq \Aut(T)$ that are completely characterised by properties of the  action of $G^{\propP{}}$ on $T$.  By the results of Section~\ref{sec:correspondence}, such properties can equivalently be described as those properties of $G$ characterised by its local action diagram. Such properties might be called {\it locally determined global properties of $(T,G)$}. For example, by Theorem~\ref{thm:invariants}, having geometrically dense action is a locally determined global property.

\begin{que} \label{q:local_to_global} What are further examples of locally determined global properties of $(T,G)$? 
\end{que}

Theorem~\ref{thm:combination_simple} suggests an interesting preorder on $\ms{S}$: say that $G_1 \prec_{OK} G_2$ if there is an open subgroup $O$ of $G_2$ and a compact normal subgroup $K$ of $O$ such that $O/K \cong G_1$.  The theorem shows that $\prec_{OK}$ is a directed preorder on $\ms{S}$, that is, any finite subset has an upper bound, and moreover, within $\ms{S}$, the groups admitting faithful $\propP{}$-closed actions on trees are cofinal.  On the other hand, every element of $\ms{S}$ is `close to the bottom' in the following sense: given $G \in \ms{S}$, there are only $\aleph_0$ compactly generated open subgroups, each of which has at most finitely many quotients in $\ms{S}$, so for each $G \in \ms{S}$ there are at most $\aleph_0$ different isomorphism types of $H \in \ms{S}$ such that $H \prec_{OK} G$.  At the same time, by \cite{SmithDuke}, $\ms{S}$ as a whole has $2^{\aleph_0}$ isomorphism classes.  In particular, writing $\ms{S}/OK$ for the poset generated by $\prec_{OK}$, it follows that $\ms{S}/OK$ has infinite ascending chains.  We are naturally led to a `well-foundedness' question:

\begin{que}\label{que:ordering}
Does $\ms{S}/OK$ have infinite descending chains?  That is, does there exist a sequence $G_0,G_1,\dots$ in $\ms{S}$ such that $G_{i+1} \prec_{OK} G_i$ but $G_i \not\prec_{OK} G_{i+1}$ for all $i$?
\end{que}

It would also be interesting to find a nontrivial $\prec_{OK}$-equivalence class, i.e. a set $\mc{X}$ of two or more pairwise nonisomorphic groups in $\ms{S}$ such that for any $G,H \in \mc{X}$, then $G$ can be realised as a quotient with compact kernel of an open subgroup of $H$ and \textit{vice versa}.

\section*{Declarations}

\begin{itemize}
\item Conflict of interest: On behalf of all authors, the corresponding author states that there is no conflict of interest.
\item Data availability: There is no data connected to this publication.
\end{itemize}

\begin{appendices}

\section{GAP implementation by Stephan Tornier}\label{app:vt_1cl_gap}

The following GAP code \cite{GAP4} implements the classification of vertex-transitive $\propP{}$-closed actions on regular trees explained in Section~\ref{sec:vertex_transitive}. Given $d\in\mathbb{N}_{\ge 2}$ it outputs a list of representatives of all associated local action diagrams. Here, a local action diagram takes the form
\begin{displaymath}
    \text{[local action, arcs, edge-reversal]}
\end{displaymath}
where the arcs are given as the list of orbits of the local action and the edge-reversal as an element of order $2$ of the symmetric group on said list.

\vspace{0.5cm}
\begin{lstlisting}[language=GAP]
LocalActionDiagrams:=function(d)
local list, G, cSubG, cGv, Gv, arcs, i, NGv, S, actNGv, R, cr, r;
# list to contain all the relevant local action diagrams
list:=[];
# initialize Sym(d) and its subgroup conjugacy classes
G:=SymmetricGroup(d);
cSubG:=ConjugacyClassesSubgroups(G);
# for each conjugacy class of subgroups of G...
for cGv in cSubG do
	# ...choose a representative and find orbits as a list of sets
	Gv:=cGv[1];
	arcs:=ShallowCopy(Orbits(Gv,[1..d]));
	for i in [1..Length(arcs)] do arcs[i]:=Set(arcs[i]); od;

	# initialize the normalizer of Gv in G
	NGv:=Normalizer(G,Gv);
	
	# choose edge-reversal, i.e. an element of order at most 2 of 
	# Sym(arcs), up to the action of NGv on arcs (Gv-orbits)
	S:=SymmetricGroup(Size(arcs));
	actNGv:=ActionHomomorphism(NGv,arcs,OnSets);
	R:=Image(actNGv);
	# in the line below, 'OrbitsDomain(R,S)' returns the list of 
	# orbits for the conjugation action of R on S
	for cr in OrbitsDomain(R,S) do
		r:=cr[1];
		if not r*r=() then continue; fi;
		Add(list,[Gv,arcs,r]);	
	od;
od;
return list;
end;

\end{lstlisting}

For example, we obtain the following output for the case $d=3$, which corresponds to the entries with $d=3$ in Table~\ref{fig:vt_upto4}.
\begin{lstlisting}[language=GAP]
LocalActionDiagrams(3);
[ [ Group(()), [ [ 1 ], [ 2 ], [ 3 ] ], () ], 
  [ Group(()), [ [ 1 ], [ 2 ], [ 3 ] ], (2,3) ], 
  [ Group([ (2,3) ]), [ [ 1 ], [ 2, 3 ] ], () ], 
  [ Group([ (2,3) ]), [ [ 1 ], [ 2, 3 ] ], (1,2) ], 
  [ Group([ (1,2,3) ]), [ [ 1, 2, 3 ] ], () ], 
  [ Group([ (1,2,3), (2,3) ]), [ [ 1, 2, 3 ] ], () ] ]
\end{lstlisting}

Table~\ref{tab:ConjClasses} records the number (up to conjugacy in $\Aut(T_d)$) of vertex-transitive $\propP{}$-closed actions on the regular tree $T_{d}$ for $d\in\{3,\dots,11\}$ in comparison to the number of conjugacy classes $|S_{d}\backslash\mathrm{Sub}(S_{d})|$ of subgroups of $S_{d}$, which is also the number of Burger--Mozes groups for $T_{d}$.

\renewcommand{\thetable}{\arabic{table}} 
\setcounter{table}{3}

\begin{table}
\caption{Number of vertex-transitive $\propP{}$-closed actions up to degree $11$}\label{fig:nb_vt_upto11}%
\begin{tabular}{@{}lll@{}}
\toprule
$d$ & $|S_{d}\backslash\mathrm{Sub}(S_{d})|$ & vertex-transitive, $\propP{}$-closed actions on $T_{d}$\\
\midrule
$3$ & $4$ & $6$ \\ 
$4$ & $11$ & $19$ \\ 
$5$ & $19$ & $40$ \\ 
$6$ & $56$ & $125$ \\ 
$7$ & $96$ & $285$ \\ 
$8$ & $296$ & $904$ \\ 
$9$ & $554$ & $2240$ \\ 
$10$ & $1593$ & $7213$ \\ 
$11$ & $3094$ & $19326$ \\ 
$12$ & $10723$ & ? \\
\botrule
\end{tabular}
\label{tab:ConjClasses}
\end{table}

We remark that according to the Online Encyclopaedia of Integer Sequences \cite[A000638]{Slo}, the number of conjugacy classes of subgroups of $S_{d}$ has been calculated up to $d\le 18$, with the most recent reference being the work of Derek Holt \cite{Holt}.

\end{appendices}



\end{document}